\documentclass{amsart}
\usepackage{amssymb,amsmath,amsfonts}

\date{July 8, 2020}

\numberwithin{equation}{section}

\newcommand{\Poisson}{\mathrm{Poi}}

\newcommand{\Graph}{\Gamma}
\newcommand{\Petal}{\Gamma}
\newcommand{\dVolMV}{d\!\Vol}
\newcommand{\shift}{\operatorname{a}}
\newcommand{\NumCycles}{{\mathrm K}}  

\newcommand{\ProbaCyls}{p}            

\newcommand{\ProbaCylsOne}{p^{(1)}}     

\newcommand{\ContributionH}{\widetilde{H}}

\newcommand{\ProbaGamma}{q}

\newcommand{\MultiHarmonicH}{H}
\newcommand{\MultiHarmonicZ}{Z}

\newlength{\halfbls}

\newcommand{\cC}{\mathcal{C}}

\newcommand{\cG}{\mathcal{G}}
\newcommand{\cH}{\mathcal{H}}

\newcommand{\cM}{\mathcal{M}}
\newcommand{\cML}{\mathcal{ML}}

\newcommand{\cQ}{\mathcal{Q}}

\newcommand{\cT}{\mathcal{T}}
\newcommand{\cY}{\mathcal{Y}}
\newcommand{\cZ}{\mathcal{Z}}
\newcommand{\cSTg}{\mathcal{ST}{\hspace*{-3pt}}_{g}}
\newcommand{\cSTgn}{\mathcal{ST}{\hspace*{-3pt}}_{g,n}}

\newcommand{\E}{{\mathbb E}}
\newcommand{\N}{{\mathbb N}}
\newcommand{\Z}{{\mathbb Z}}
\newcommand{\R}{{\mathbb R}}
\newcommand{\C}{{\mathbb C}}
\newcommand{\Q}{{\mathbb Q}}
\newcommand{\Proba}{{\mathbb P}}
\newcommand{\Var}{{\mathbb V}}

\newcommand{\Area}{\operatorname{Area}}
\newcommand{\Vol}{\operatorname{Vol}}
\newcommand{\CP}{{\mathbb C}\!\operatorname{P}^1}
\newcommand{\card}{\operatorname{card}}
\newcommand{\Mod}{\operatorname{Mod}}
\newcommand{\Stab}{\operatorname{Stab}}

\renewcommand{\epsilon}{\varepsilon}

\newtheorem{Theorem}{Theorem}[section]
\newtheorem{Proposition}[Theorem]{Proposition}
\newtheorem{Lemma}[Theorem]{Lemma}
\newtheorem{Corollary}[Theorem]{Corollary}
\newtheorem{Question}{Question}
\newtheorem{Conjecture}{Conjecture}

\newtheorem*{NNTheorem}{Theorem}
\newtheorem*{NNProposition}{Proposition}
\newtheorem*{NNCorollary}{Corollary}

\theoremstyle{definition}
\newtheorem{Definition}[Theorem]{Definition}

\theoremstyle{remark}
\newtheorem{Remark}[Theorem]{Remark}

\title[Asymptotic geometry in large genera]
{Large genus asymptotic geometry of random square-tiled surfaces and of random multicurves
}

\author[V.~Delecroix]{Vincent Delecroix}
\address{
LaBRI,
Domaine universitaire,
351 cours de la Lib\'eration, 33405 Talence, FRANCE
}
\email{vincent.delecroix@u-bordeaux.fr}

\author[\'E.~Goujard]{\'Elise Goujard}
\thanks{The research of the second author was partially supported by
PEPS}
\address{
Institut de Math\'ematiques de Bordeaux,
Universit\'e de Bordeaux,
351 cours de la Lib\'eration, 33405 Talence, FRANCE
}
\email{elise.goujard@gmail.com}

\author[P.~G.~Zograf]{Peter~Zograf}
\address{
St.~Petersburg Department, Steklov Math. Institute, Fontanka 27,
St. Petersburg 191023, and Chebyshev Laboratory,
St. Petersburg State University, 14th
Line V.O. 29B, St.Petersburg 199178 Russia}
\email{zograf@pdmi.ras.ru}

\author[A.~Zorich]{Anton Zorich}
\thanks{This material is based upon work supported by the
ANR-19-CE40-0021 grant. It was also supported by the
NSF Grant DMS-1440140 while
part of the authors were in residence at the MSRI during
the Fall 2019 semester}
\address{
Center for Advanced Studies, Skoltech;
Institut de Math\'ematiques de Jussieu --
Paris Rive Gauche,
Case 7012,
8 Place Aur\'elie Nemours,
75205 PARIS Cedex 13, France}
\email{anton.zorich@gmail.com}

\begin{document}

\begin{abstract}
We study the combinatorial geometry of a random closed
multicurve on a surface of large genus $g$ and of a random
square-tiled surface of large genus $g$. We prove that
primitive components $\gamma_1, \dots,\gamma_k$ of a random
multicurve $m_1\gamma_1+\dots +m_k\gamma_k$ represent
linearly independent homology cycles with asymptotic
probability $1$ and that all its weights $m_i$ are equal to
$1$ with asymptotic probability $\sqrt{2}/2$. We prove
analogous properties for random square-tiled surfaces. In
particular, we show that all conical singularities of a
random square-tiled surface belong to the same leaf of the
horizontal foliation and to the same leaf of the vertical
foliation with asymptotic probability $1$.

We show that the number of components of a random
multicurve and the number of maximal horizontal cylinders
of a random square-tiled surface of genus $g$ are both very
well-approximated by the number of cycles of a random
permutation for an explicit non-uniform measure on the
symmetric group of $3g-3$ elements. In particular, we
prove that the expected value of these quantities is
asymptotically equivalent to $(\log(6g-6) + \gamma)/2 +
\log 2$.

These results are based on our formula for the Masur--Veech
volume $\Vol\cQ_g$ of the moduli space of holomorphic
quadratic differentials combined with deep large genus
asymptotic analysis of this formula performed by
A.~Aggarwal and with the uniform asymptotic formula for
intersection numbers of $\psi$-classes on
$\overline{\cM}_{g,n}$ for large $g$ proved by A.~Aggarwal
in 2020.
\end{abstract}

\maketitle

\tableofcontents

\section{Introduction and statements of main results}
\label{s:intro}

We aim to study random multicurves on surfaces of large
genus $g$. Before proceeding to the statements of our main
results we consider the classical setting of random
integers and of random permutations which allow to set up
the concept of a random compound object.
For more information on the probabilistic analysis of
decomposition of combinatorial objects into elementary
components we recommend the monograph of R.~Arratia,
\mbox{A.~D.~Barbour},
S.~Tavar\'e~\cite{Arratia:Barbour:Tavare}. An enlightening
introduction can be found in the blog post of
T.~Tao~\cite{Tao}.
\medskip

\noindent
\textbf{Prime decomposition of a random integer.}
The Prime Number Theorem states that an integer number $n$
taken randomly in a large interval $[1,N]$ is prime with
asymptotic probability $\frac{\log N}{N}$. Denote by
$\omega(n)$ the number of prime divisors of an integer $n$
counted without multiplicities. In other words, if $n$ has
prime decomposition $n=p_1^{m_1}\dots p_k^{m_k}$, let
$\omega(n)=k$. By the Erd\H{o}s--Kac
theorem~\cite{Erdos:Kac}, the centered and rescaled
distribution prescribed by the counting function
$\omega(n)$ tends to the normal distribution:
\begin{equation}
\label{eq:CLT1}
\lim_{N\to+\infty}\frac{1}{N}
\card\left\{n\le N\,\Big|\,
\frac{\omega(n)-\log\log N}{\sqrt{\log\log N}}
\le x\right\}=\frac{1}{\sqrt{2\pi}}
\int_{-\infty}^x e^{-\tfrac{t^2}{2}} dt\,.
\end{equation}
The subsequent papers of A. R\'enyi and
P.~Tur\'an~\cite{Renyi:Turan}, and of
A.~Selberg~\cite{Selberg} describe the rate of
convergence.
\medskip

\noindent
\textbf{Cycle decomposition of a uniform random permutation.}
Denote by $\NumCycles_n(\sigma)$ the number of disjoint
cycles in the cycle decomposition of a permutation $\sigma$
in the symmetric group $S_n$. Consider the uniform
probability measure on $S_n$. A random permutation $\sigma$
of $n$ elements has exactly $k$ cycles in its cyclic
decomposition with probability
$\Proba\big(\NumCycles_n(\sigma) = k\big) =
\frac{s(n,k)}{n!}$, where $s(n,k)$ is the unsigned Stirling
number of the first kind. It is immediate to see that
$\Proba\big(\NumCycles_n(\sigma) = 1\big) = \tfrac{1}{n}$.
V.~L.~Goncharov proved in~\cite{Goncharov} the following
expansions for the expected value and for the variance of
$\NumCycles_n$ as $n\to+\infty$:
\begin{equation}
\label{eq:mean:variance:perm}
\E(\NumCycles_n) = \log n + \gamma + o(1)\,,
\qquad
\Var(\NumCycles_n) = \log n + \gamma - \zeta(2) + o(1)\,,
\end{equation}
as well as the following central limit theorem
\begin{equation}
\label{eq:CLTperm}
\lim_{n\to+\infty}\frac{1}{n!}
\card\left\{\sigma\in S_n\,\Big|\,
\frac{\NumCycles_n(\sigma)-\log n}{\sqrt{\log n}}
\le x\right\}=\frac{1}{\sqrt{2\pi}}
\int_{-\infty}^x e^{-\tfrac{t^2}{2}} dt\,.
\end{equation}

As can be seen in~\eqref{eq:mean:variance:perm}
and in~\eqref{eq:CLTperm}, the number of cycles
in the cycle decomposition of a random permutation is of
the order of $\log n$, so cycles are ``rare events''. In
such situation one expects the distribution to be close to
a Poisson distribution. Recall that the \emph{Poisson
distribution with parameter $\lambda$} is
\begin{equation}
\label{eq:Poisson:def}
\Poisson_\lambda(k) = e^{-\lambda}\ \frac{\lambda^k}{k!}\,,
\qquad\text{where }k=0,1,2,\dots
\end{equation}
H.~Hwang proved in~\cite{Hwang:PhD} that the distribution of the
random variable $\NumCycles_n$ is approximated by the
Poisson distribution $\Poisson_{\log n}$ in a very strong
sense, which can be formalized as ``\textit{mod-Poisson
converges with parameter $\log n$ and limiting function
$1/\Gamma(t)$}'', using the terminology of E.~Kowalski and
A.~Nikeghbali~\cite{KowalskiNikeghbali}. We discuss the
notion of mod-Poisson convergence in
Section~\ref{ss:non:uniform:permutations}. We emphasize
that such approximation is much stronger than the
central limit theorem.

The result of H.~Hwang~\cite{Hwang:stirling} (representing
a particular case of results in~\cite[Chapter
5]{Hwang:PhD}) implies, that for large $n$ and for any
positive $x$, the distribution of the number of cycles is
uniformly well-approximated in a neighborhood of $x\log n$
by the Poisson distribution with parameter $\log n+\shift(x)$,
where the explicit correctional constant term $\shift(x)$ is
completely determined by the limiting function and does not
depend on $n$. Namely, for any $x > 0$ we have uniformly in
$0 \leq k \leq x \log n$
\begin{equation}
\label{eq:LDperm}
\Proba\big(\NumCycles_n = k + 1\big)
= \frac{(\log n)^{k}}{n \cdot k!}
\left( \frac{1}{\Gamma(1 + \frac{k}{\log n})} + O\left( \frac{k}{(\log n)^2} \right) \right).
\end{equation}
\medskip

\noindent
\textbf{Shape of a random multicurve on a surface of fixed
genus.}
Consider a smooth oriented closed surface $S$ of genus $g$.
A \textit{multicurve} $\gamma=\sum m_i\gamma_i$ (as the one in the picture by
D.~Calegari from~\cite{Calegari} presented in
Figure~\ref{fig:multicurve}) is a formal weighted sum of
curves $\gamma_i$ with strictly positive integer weights
$m_i$ where $\gamma_1,\dots,\gamma_k$ is a collection of
non-contractible simple curves on $S$ that are pairwise
non-isotopic. Following the usual convention, we do not
distinguish between the free homotopy class of a
multicurve and the multicurve itself.

\begin{figure}[htp]
\includegraphics{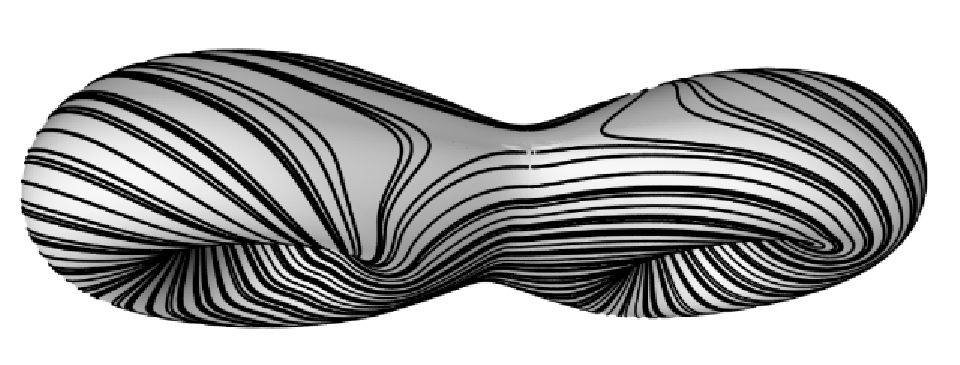}
\vspace{65bp}
\caption{
\label{fig:multicurve}
Simple closed multicurve
on a surface of genus two}
\end{figure}

Every multicurve $\gamma=\sum m_i\gamma_i$ defines
a \textit{reduced} multicurve
$\gamma_{\mathit{reduced}}=\sum \gamma_i$. Note that the
number of reduced multicurves on a surface of a fixed genus
$g$ considered up to the action of the mapping class
group $\Mod_g$ is finite. We say that two
multicurves have the same topological type if they belong
to the same orbit of $\Mod_g$. For example, a simple closed
curve has one of the following topological types: either it
is non-separating, or it separates the surface into
subsurfaces of genera $g'$ and $g - g'$ for some
$1 \leq g' \leq g/2$.

Multicurves on a closed surface of genus $g$ (considered up
to free homotopy) are parameterized by integer points
$\cML_g(\Z)$ in the space of measured laminations $\cML_g$
introduced by W.~Thurston~\cite{Thurston}. Any hyperbolic
metric on $S$ provides a length function $\ell$ that
associates to a closed curve $\gamma$ the length
$\ell(\gamma)$ of its unique geodesic representative. The
length function $\ell$ extends to multicurves as
$\ell(\gamma) = m_1 \ell(\gamma_1) + \ldots + m_k
\ell(\gamma_k)$. Fixing some upper bound $L$ for the length
of a multicurve, one can consider the finite set of
multicurves of length at most $L$ on $S$ with respect to
the length function $\ell$. See also the paper of
M.~Mirzakhani~\cite{Mirzakhani:Thesis} and works of
V.~Erlandsson, H.~Parlier, K.~Rafi and
J.~Souto~\cite{Rafi:Souto}, \cite{Erlandsson:Parlier:Souto}
and~\cite{Erlandsson:Souto} for alternative ways to measure
the length of a multicurve.

Choosing the uniform measure on all integral multicurves
of length at most $L$ and letting $L$ tend to infinity we
define a ``random multicurve'' on a surface of fixed
genus $g$ in the same manner as we considered ``random
integers'', see
Section~\ref{ss:Frequencies:of:simple:closed:curves} for
details. We emphasize that studying asymptotic statistical
geometry of multicurves as the bound $L$ tends to infinity
we always keep the genus $g$, considered as a parameter,
fixed. One can ask, for example, what is the probability
that a random simple closed curve separates the surface of
genus $g$ in two components? Or, more generally, what is
the probability that the reduced multicurve
$\gamma_{\mathit{reduced}}$ associated to a random
multicurve $\gamma$ separates the surface of genus $g$ into
several components? With what probability a random
multicurve $m_1\gamma_1+m_2\gamma_2+\dots+m_k\gamma_k$ has
$k=1,2,\dots,3g-3$ primitive connected components
$\gamma_1,\dots,\gamma_k$? What are the typical weights
$m_1,\dots,m_k$?

A beautiful answer to all these questions was found by
M.~Mirzakhani
in~\cite{Mirzakhani:grouth:of:simple:geodesics}. She
expressed the frequency of multicurves of any fixed
topological type in terms of the intersection numbers
$\int_{\overline{\cM}_{g',n'}}\psi_1^{d_1}\dots\psi_{n'}^{d_{n'}}$,
where $2g'+n'\le 2g$. (These intersection numbers are also
called \textit{correlators of Witten's two dimensional
topological gravity}). For small genera $g$ the formula of
M.~Mirzakhani provides explicit rational values for the
quantities discussed above. For example, the reduced
multicurve associated to a random multicurve
on a surface of genus $2$ without cusps as in
Figure~\ref{fig:multicurve} separates the surface with
probability $\tfrac{67}{315}$ and has $1$, $2$ or $3$
components with probabilities
$\tfrac{7}{27},\tfrac{5}{9},\tfrac{5}{27}$ respectively.

The formulae of Mirzakhani are applicable to surfaces of
any genera. The exact values of the intersection numbers
can be computed through Witten--Kontsevich
theory~\cite{Witten}, \cite{Kontsevich}. However, despite
the fact that these intersection numbers were extensively
studied, there were no uniform estimates for Witten
correlators for large $g$ till the recent results
of A.~Aggarwal~\cite{Aggarwal:intersection:numbers}. This
is one of the reasons why the following question remained
open.

\begin{Question}
\label{question:multicurves}
What shape has a random multicurve on a surface of large genus?
\end{Question}

The current paper aims to answer this question to some extent.
Denote by $K_g(\gamma)$ the number of components
$k$ of the multicurve $\gamma=\sum_{i=1}^k m_i\gamma_i$
counted without multiplicities.

\begin{Theorem}
\label{th:multicurves:a:b:c}
Consider a random multicurve $\gamma=\sum_{i=1}^k
m_i\gamma_i$ on a surface $S$ of genus $g$. Let
$\gamma_{\mathit{reduced}}=\gamma_1+\dots+\gamma_k$ be the
underlying reduced multicurve. The
following asymptotic properties hold as
$g\to+\infty$.
\begin{itemize}
\item[(a)]
The multicurve $\gamma_{\mathit{reduced}}$ does not separate the surface (i.e. $S-\sqcup\gamma_i$ is connected) with probability which tends to 1.
\item[(b)]
The probability that a random
multicurve $\gamma=\sum_{i=1}^k m_i\gamma_i$
is primitive
(i.e. that $m_1=m_2=\dots=1$)
tends to $\frac{\sqrt{2}}{2}$.
\item[(c)]
For any sequence of positive integers $k_g$ with $k_g =
o(\log g)$ the probability that a random multicurve
$\gamma=\sum_{i=1}^{k_g} m_i\gamma_i$
is primitive (i.e. that
$m_1=\dots=m_{k_g}=1$) tends to $1$ as $g\to +\infty$.
\end{itemize}
\end{Theorem}
There is no contradiction between parts (b) and (c) of the
above Theorem since in (c) we consider only those random
multicurves for which the underlying primitive multicurve
has an imposed number $k_g$ of components, while in
(b) we consider all multicurves. In other words, in part
(c) we consider the conditional probability. Part (b) of
the above Theorem admits the following generalization.

\begin{Theorem}
\label{th:multicurves:bounded:weights}
For any positive integer $m$, the probability that all
weights $m_i$ of a random multicurve
$\gamma=m_1\gamma_1+m_2\gamma_2+\dots$ on a surface of
genus $g$ are bounded by a positive integer $m$ (i.e. that
$m_1\le m, m_2\le m,\dots$) tends to $\sqrt{\frac{m}{m+1}}$
as $g\to+\infty$.
\end{Theorem}

We describe the probability distribution of the random
variable $K_g(\gamma)$ later in this section. However, to
follow comparison with prime decomposition of random
integers and with cycle decomposition of random
permutations we present here the central limit theorem
stated for random multicurves.

\begin{Theorem}
\label{thm:CLT:multicurve}
Choose a non-separating simple closed curve $\rho_g$ on a
surface of genus $g$. Denote by $\iota(\rho_g,\gamma)$ the
geometric intersection number of $\rho_g$ and $\gamma$. The
centered and rescaled distribution defined by the counting
function $K_g(\gamma)$ tends to the normal distribution:
\begin{multline*}
\lim_{g\to+\infty}
\sqrt{\frac{3\pi g}{2}}\cdot 12g
\cdot (4g-4)!
\cdot\left(\frac{9}{8}\right)^{2g-2}
\\
\lim_{N\to+\infty}
\frac{1}{N^{6g-6}}
\card\Bigg(\bigg\{\gamma\in\cML_g(\Z)\,\bigg|\,
\iota(\rho_g,\gamma)\le N\quad \text{and}
\\
\frac{K_g(\gamma)-\frac{\log g}{2}}{\sqrt{\frac{\log g}{2}}}
\le x\bigg\}
/\Stab(\rho_g)\Bigg)
=\frac{1}{\sqrt{2\pi}}
\int_{-\infty}^x e^{-\tfrac{t^2}{2}} dt\,.
\end{multline*}
\end{Theorem}
\noindent
Here $\Stab(\rho_g)\subset\Mod_g$ is the stabilizer of the
simple closed curve $\rho_g$ in the mapping class group
$\Mod_g$.

In plain words, the above theorems say that the components
$\gamma_1,\dots,\gamma_k$ of a random multicurve
$\gamma=\sum_{i=1}^k m_i\gamma_i$ on a surface of large
genus $g$ have all chances to go around $k$ independent
handles, where $k$ is close to $\frac{1}{2}\log g$, and that with a
high probability all the weights $m_i$ of a random
multicurve are uniformly small. In particular, with
probability greater than $0.7$ a random multicurve is
primitive, i.e. all the weights $m_i$ are equal to $1$.

Our description of the asymptotic geometry of random
multicurves on surfaces of large genus and of random
square-tiled surfaces of large genera relies on
fundamental recent
results~\cite{Aggarwal:intersection:numbers} of
A.~Aggarwal, who proved, in particular, the large genus
asymptotic formulae for the Masur--Veech volume $\Vol\cQ_g$
and for the intersection numbers of $\psi$-classes on
$\overline{\cM}_{g,n}$, conjectured by the authors
in~\cite{DGZZ:volume}.
\medskip

\noindent
\textbf{Random square-tiled surfaces of large genus.}
A \textit{square-tiled surface} is a closed oriented
quadrangulated surface (i.e. a surface built by gluing identical
squares along their edges), such that the quadrangulation
satisfies the following
properties. Consider the flat metric on the surface induced
by the flat metric on the squares. We assume that
edges of the squares are identified by isometries, which
implies that the induced flat metric is non-singular on the
complement of the vertices of the squares. We require that
the parallel transport of a vector $\vec v$ tangent to the
surface along any closed path avoiding conical
singularities brings the vector $\vec v$ either to itself
or to $-\vec v$. In other words, we require that the
holonomy group of the metric is $\Z/2\Z$ (compared to
$\Z/4\Z$ for a general quadrangulation). This holonomy
assumption implies that defining some edge
to be ``horizontal'' or ``vertical'' we
uniquely determine for each of the remaining edges, whether
it is ``horizontal'' or ``vertical''.
Speaking of square-tiled surfaced we
always assume that the choice of horizontal and vertical
edges is done.

Our holonomy assumption implies that the number of squares
adjacent to any vertex is even. In this article we restrict
ourselves to consideration of square-tiled surfaces with no
conical singularities of angle $\pi$. In other words,
vertices adjacent to exactly two squares are not allowed.
Square-tiled surfaces satisfying the above restrictions can
be seen as integer points in the total space $\cQ_g$ of the
vector bundle of holomorphic quadratic differentials over
the moduli space of complex curves $\cM_g$.

A stronger restriction on the quadrangulation imposing
trivial linear holonomy to the induced flat metric defines
\textit{Abelian} square-tiled surfaces; they correspond to
integer points in the total space $\cH_g$ of the vector
bundle of holomorphic Abelian differentials over the moduli
space of complex curves $\cM_g$. The subset of square-tiled
surfaces having prescribed linear holonomy and prescribed
cone angle at each conical singularity corresponds to the
set of integer points in the associated \textit{stratum} in
the moduli space of quadratic or Abelian differentials
respectively.

A square-tiled surface admits a natural
decomposition into maximal horizontal cylinders. For
example, the square-tiled surface in the left picture of
Figure~\ref{fig:square:tiled:surface:and:associated:multicurve}
(which, for simplicity of illustration, contains conical
points with cone angles $\pi$)
has four maximal horizontal cylinders highlighted by
different shades of grey. Two of these cylinders are
composed of two horizontal bands of squares. Each of the
remaining two cylinders is composed of a single horizontal
band of squares.

For any positive integer $N$, the set $\cSTg(N)$ of
square-tiled surfaces of genus $g$ having no singularities
of angle $\pi$ and having at most $N$ squares in the tiling
is finite. Choosing the uniform measure on the set
$\cSTg(N)$ and letting the bound $N$ for the number of
squares tend to infinity, we define a ``random square-tiled
surface'' of fixed genus $g$ in the same manner as we
considered ``random multicurves'' on a fixed surface, see
Section~\ref{ss:Frequencies:of:square:tiled:surfaces} for
details. We emphasize that studying asymptotic statistical
geometry of square-tiled surfaces as the bound $N$ tends to
infinity we always keep the genus $g$, considered as a
parameter, fixed.
One can study the decomposition of a random square-tiled
surface into maximal horizontal cylinders in the same sense
as we considered prime decomposition of random integers or
cyclic decomposition of random permutations.

For each stratum in the moduli space of Abelian
differentials, we computed
in~\cite{DGZZ:one:cylinder:Yoccoz:volume} the probability
that a random square-tiled surface in this stratum has a
single cylinder in its horizontal cylinder decomposition.
This result can be seen as an analog of the Prime Number
Theorem for square-tiled surfaces. In particular, using
results~\cite{Aggarwal:Volumes}
and~\cite{Chen:Moeller:Sauvaget:Zagier} we proved that for
strata of Abelian differentials corresponding to large
genera, this probability is asymptotically $\frac{1}{d}$,
where $d$ is the dimension of the stratum. However, more
detailed description of statistics of square-tiled surfaces
in individual strata of Abelian differentials is currently
out of reach with the exception of several low-dimensional
strata.
Conjecturally, for any stratum of Abelian differentials of
dimension $d$, the statistics of the number of maximal
horizontal cylinders of a random square-tiled surface in
the stratum becomes very well-approximated by the
statistics of the number $\NumCycles_n(\sigma)$ of disjoint
cycles in a random permutation of $d$ elements as
$d\to+\infty$; see Section~\ref{s:speculations} for
details.

In the current paper we address more general question.

\begin{Question}
\label{question:square:tiled}
What shape has a random square-tiled surface of large genus
assuming that it does not have conical points of angle
$\pi$?
\end{Question}

Denote by $K_g(S)$ the number of maximal horizontal cylinders
in the cylinder decomposition of a square-tiled surface
$S$ of genus $g$.

\begin{Theorem}
\label{th:square:tiled:a:b:c}
A random square-tiled surface $S$ of genus $g$ with no
conical singularities of angle $\pi$ has the following
asymptotic properties as $g\to+\infty$.
\begin{itemize}
\item[(a)]
All conical singularities of $S$ are located at the same
leaf of the horizontal foliation and at the same leaf of
the vertical foliation with probability which tends to 1.
\item[(b)]
The probability that each maximal horizontal cylinder of
$S$ is composed of a single band of squares tends to
$\frac{\sqrt{2}}{2}$.
\item[(c)]
For any sequence of positive integers $k_g$ with $k_g =
o(\log g)$ the probability that each maximal horizontal cylinder of a
random $k_g$-cylinder square-tiled surface of genus $g$ is composed of a
single band of squares tends to $1$ as the genus $g$ tends
to $+\infty$.
\end{itemize}
\end{Theorem}

Similarly to the case of multicurves, part (b) of the above
Theorem admits the following generalization.

\begin{Theorem}
\label{th:square:tiled:bounded:weights}
For any $m\in\N$,
the probability that
all maximal horizontal cylinders of
a random square-tiled surface of genus $g$
have at most $m$ bands of squares
tends to $\sqrt{\frac{m}{m+1}}$ as $g\to+\infty$.
\end{Theorem}

We state now the central limit theorem for
square-tiled surfaces.

\begin{Theorem}
\label{thm:CLT:square:tiled}
The centered and rescaled
distribution defined by the counting function $K_g(S)$
tends to the normal distribution as $g \to +\infty$:
\begin{multline*}
\lim_{g\to+\infty}
3\pi g\cdot\left(\frac{9}{8}\right)^{2g-2}
\\
\lim_{N\to+\infty}
\frac{1}{N^{6g-6}}
\card\left\{S\in\cSTg(N)\,\bigg|\,
\frac{k(S)-\tfrac{\log g}{2}}{\sqrt{\tfrac{\log g}{2}}}
\le x\right\}
=\frac{1}{\sqrt{2\pi}}
\int_{-\infty}^x e^{-\frac{t^2}{2}} dt\,.
\end{multline*}
\end{Theorem}
\medskip

\noindent
\textbf{Approach to the study of random multicurves and of
random square-tiled surfaces of large genera: from $p_g(k)$ to $q_g(k)$.}
It is time to admit that the parallelism between
Theorems~\ref{th:multicurves:a:b:c}--\ref{thm:CLT:multicurve}
and respectively
Theorems~\ref{th:square:tiled:a:b:c}--\ref{thm:CLT:square:tiled}
is not accidental.

Recall that we denote by $K_g(\gamma)$ the number of
components $k$ of the multicurve $\gamma=\sum_{i=1}^k
m_i\gamma_i$ on a surface of genus $g$ counted without
multiplicities and by $K_g(S)$ the number of maximal
horizontal cylinders in the cylinder decomposition of a
square-tiled surface $S$ of genus $g$.
The following theorem is a direct corollary of
Theorem~1.21 from Section~1.8 in~\cite{DGZZ:volume}.
(For the sake of completeness we reproduce the original Theorem
in Section~\ref{ss:Frequencies:of:square:tiled:surfaces} below.)

\begin{Theorem}
\label{th:same:distribution}
For any genus $g\ge 2$ and for any $k\in\N$, the probability
$\ProbaCyls_g(k)$
that a random multicurve $\gamma$ on a surface of genus $g$
has exactly $k$
components counted without multiplicities
coincides with the probability
that a random square-tiled surface
$S$ of genus $g$ has exactly $k$
maximal horizontal cylinders:
\begin{equation}
\label{eq:same:p:g}
\ProbaCyls_g(k)
=\Proba\big(K_g(\gamma)=k\big)
=\Proba\big(K_g(S)=k\big)\,.
\end{equation}
In other words, $K_g(\gamma)$ and $K_g(S)$, considered as
random variables, determine the same probability
distribution $\ProbaCyls_g(k)$, where $k=1,2,\dots,3g-3$.
\end{Theorem}

The Theorem above shows that Questions~1 and~2 are,
basically, equivalent. The description of the large
genus asymptotic properties of the resulting probability
distribution $\ProbaCyls_g(k)$ can be seen as the main
unified goal of the current paper.

The starting point of our approach to the study of the
probability distribution $\ProbaCyls_g(k)$ is the the
formula for the Masur--Veech volume $\Vol\cQ_g$ of the
moduli space of holomorphic quadratic differentials derived
in our recent paper~\cite{DGZZ:volume}. This formula
represents $\Vol\cQ_g$ as a finite sum of contributions of
square-tiled surfaces of all possible topological types
(Section~\ref{ss:intro:Masur:Veech:volumes} describes this
in detail). However, the number of such topological types
grows exponentially as genus grows. Moreover, the
contribution of square-tiled surfaces of a fixed
topological type to $\Vol\cQ_g$ is expressed in terms of
the intersection numbers of $\psi$-classes (Witten
correlators) which are difficult to evaluate explicitly in
large genera.

We conjectured in~\cite{DGZZ:volume} that in large genera,
the dominant part of the contribution to $\Vol\cQ_g$ comes
from square-tiled surfaces having all conical singularities
at the same horizontal level. The topological type (see
Section~\ref{ss:MV:volume} for the rigorous definition of
the ``topological type'') of such square-tiled surfaces is
completely determined by the number $k$ of maximal
horizontal cylinders which varies from $1$ to $g$. This
conjecture suggested a strategy for overcoming the first
difficulty, reducing the study of all immense variety of
topological types of square-tiled surfaces to the study of
$g$ explicit topological types. We also conjectured
in~\cite{DGZZ:volume} that under certain assumptions on $g$
and $n$, the intersection numbers
$\int_{\overline{\cM}_{g,n}}\psi_1^{d_1}\dots\psi_{n}^{d_{n}}$
are uniformly well-approximated by an explicit closed
expression in the variables $d_1, \dots, d_n$, and that the
error term becomes uniformly small with respect to all
possible partitions $d_1+\dots+d_n=3g-3+n$ for large values
of $g$. This conjecture suggested a plan for overcoming the
second difficulty reducing analysis of volume contributions
of square-tiled surfaces of $g$ distinguished topological
types to analysis of closed expressions in multivariate
harmonic sums. Such analysis led us, in particular, to the
conjectural large genus asymptotics of the Masur--Veech
volume $\Vol\cQ_g$.

In terms of the probability distributions, we replace the
original distribution $\ProbaCyls_g(k)$ with an auxiliary
probability distribution $\ProbaGamma_g(k)$ in this
approach. The distribution $\ProbaGamma_g(k)$ describes the
contributions of square-tiled surfaces of $g$ distinguished
topological types (corresponding to the situation when all
conical singularities are located at same horizontal layer
and the surface has $k=1,\dots,g$ maximal horizontal
cylinders), where, moreover, we replace the Witten
correlators with the corresponding asymptotic expressions.
The precise definition of $\ProbaGamma_g(k)$ is given in
Equation~\eqref{eq:def:q_g} in
Section~\ref{ss:Volume:contribution:single:vertex}.
Informally, our conditional asymptotic result
in~\cite{DGZZ:volume} stated that for large genera $g$ the
auxiliary distribution $\ProbaGamma_g(k)$ well-approximates
the original probability distribution $\ProbaCyls_g(k)$
modulo the conjectures mentioned above.

Deep analysis
of volume contributions of square-tiled surfaces of
different topological types was performed by A.~Aggarwal
in~\cite{Aggarwal:intersection:numbers}. Moreover, in the
same paper A.~Aggarwal established uniform asymptotic
bounds for Witten correlators using elegant approach
through biased random walk.  In particular, he proved all
conjectures from~\cite{DGZZ:volume} (in a stronger form)
transforming conditional results from~\cite{DGZZ:volume}
into unconditional statements.

In the current paper we follow the original approach,
approximating the probability distribution
$\ProbaCyls_g(k)$ with a slight modification of the probability distribution
$\ProbaGamma_g(k)$ as described above. However, the fine
asymptotic analysis of A.~Aggarwal allows to state that
$\ProbaGamma_g(k)$ ``well-approximates'' $\ProbaCyls_g(k)$
in much stronger sense than it was claimed in the original
preprint~~\cite{DGZZ:volume}. Moreover, we realized that
our ``slight modification of the
probability distribution $\ProbaGamma_g(k)$'' has
combinatorial interpretation of independent interest and
admits a detailed description based on technique developed
by H.~Hwang in~\cite{Hwang:PhD}.

Having explained the scheme of our approach we can state
now the main results concerning the probability
distribution $\ProbaCyls_g(k)$. We
start with a formal definition of the
``slight modification of the
probability distribution $\ProbaGamma_g(k)$''
through random permutations. It
plays an important role in the current paper.
\medskip

\noindent
\textbf{Random non-uniform permutations and distribution
$\boldsymbol{q_{n,\infty,1/2}}$.}
Let $\theta$ be a sequence $\{\theta_k\}_{k \geq 1}$ of
non-negative real numbers. Given a permutation $\sigma \in
S_n$ with cycle type $(1^{\mu_1} 2^{\mu_2} \ldots
n^{\mu_n})$, where
$1\cdot\mu_1+2\cdot\mu_2+\dots+n\cdot\mu_n=n$, we define
its \textit{weight} $w_\theta(\sigma)$ by the following
formula:
$$
w_\theta(\sigma) = \theta_1^{\mu_1} \theta_2^{\mu_2} \cdots \theta_n^{\mu_n}.
$$
To every collection of positive numbers $\theta = \{\theta_k\}_{k \geq 1}$, we associate a probability measure on the
symmetric group $S_n$ by means of the weight function defined above:
\begin{equation}
\label{eq:proba:sym:group}
\Proba_{\theta,n}(\sigma) := \frac{w_{\theta}(\sigma)}{n!\cdot W_{\theta,n}}\,,
\quad \text{where} \qquad
W_{\theta,n} := \frac{1}{n!} \sum_{\sigma \in S_n} w_\theta(\sigma)
\quad \text{and} \quad k\in\N\,.
\end{equation}
Denote by $\Proba_{n,\infty,1/2}$ the non-uniform probability
measure on the symmetric group $S_n$ associated to the
collection of strictly positive numbers
$\theta_k=\zeta(2k)/2$, where $k=1,2,\dots$ and $\zeta$ is the
Riemann zeta function. Consider the
random variable $\NumCycles_n(\sigma)$ on the symmetric
group $S_n$, where the random permutation $\sigma$
corresponds to the law $\Proba_{n,\infty,1/2}$ and
$\NumCycles_n(\sigma)$ is the number of disjoint cycles in
the cycle decomposition of such random permutation
$\sigma$. The random variable $\NumCycles_n(\sigma)$ takes
integer values in the range $[1,n]$. We introduce the following
notation:
\begin{equation}
\label{eq:q:3g:minus:3:infty:1:2:as:proba}
q_{n,\infty,1/2}(k)
 = \Proba_{n, \infty, 1/2}\big(\NumCycles_{3g-3}(\sigma) = k\big)
\end{equation}
for the law of the random variable $\NumCycles_{n}(\sigma)$
with respect to the probability measure $\Proba_{n,\infty,1/2}$.
We prove in Section~\ref{s:sum:over:single:vertex:graphs}
series of results which informally can be summarized by the
following claim:
\textit{the probability distribution $q_{3g-3, \infty,
1/2}$ well-approximates the probability distribution
$\ProbaGamma_g$.}
We admit that the approximating distribution
$\ProbaGamma_g$ will be formally defined only later,
namely, in Equation~\eqref{eq:def:q_g} in
Section~\ref{ss:Volume:contribution:single:vertex} and,
strictly speaking, would not be used explicitly. The above
claim explains, however, our interest for the probability
distribution $q_{3g-3, \infty, 1/2}$ which would be
actually used for approximation. An important step of
comparison of distributions $\ProbaCyls_g$ and $q_{3g-3,
\infty, 1/2}$ is established in
Lemma~\ref{lem:cycle:distribution:vs:harmonic:sum} stated
and proved in Section~\ref{ss:non:uniform:permutations}.
Theorems~\ref{thm:permutation:mod:poisson:introduction}
and~\ref{thm:permutation:asymptotics} below carry
comprehensive information on the probability distribution
$q_{3g-3, \infty, 1/2}(k)
=\Proba_{3g-3,\infty,1/2}\big(\NumCycles_{3g-3}(\sigma) = k\big)$.
\begin{Theorem}
\label{thm:permutation:mod:poisson:introduction}
Let $\Proba_{n,\infty,1/2}$ be the probability distribution
on $S_n$ associated to the collection $\theta_k =
\zeta(2k)/2$. Then for all $t \in \C$ we have as $n \to
+\infty$
\begin{equation}
\label{eq:mod:Poisson:for:permutations}
\E_{n,\infty,1/2}\left(t^{\NumCycles_n}\right) =
(2n)^{\tfrac{t-1}{2}}
\cdot \frac{t\cdot\Gamma(\tfrac{3}{2})}{\Gamma(1+\tfrac{t}{2})}\
\left(1 + O \left( \frac{1}{n} \right) \right)\,,
\end{equation}
where the error term is uniform for $t$ in any compact
subset of $\C$.
\end{Theorem}

For any $\lambda > 0$, let $u_{\lambda, 1/2}(k)$ for $k\in\N$ be the
coefficients of the following Taylor expansion
\begin{equation}
\label{eq:poisson:gamma:1:2}
e^{\lambda (t-1)}\cdot
\displaystyle
\frac{t \cdot \Gamma\left(\tfrac{3}{2}\right)}{\Gamma\left(1 + \tfrac{t}{2}\right)}
=
\sum_{k \geq 1} u_{\lambda, 1/2}(k) \cdot t^k
\end{equation}
Recall that
$\Gamma\left(\tfrac{3}{2}\right)=\frac{\Gamma\left(\tfrac{1}{2}\right)}{2}=\frac{\sqrt{\pi}}{2}$.
We have
\[
u_{\lambda,1/2}(k)
=\sqrt{\pi}\cdot e^{-\lambda}\cdot
\frac{1}{k!}\cdot
\sum_{i=1}^k \binom{k}{i} \cdot \phi_i \cdot
\left(\frac{1}{2}\right)^i \cdot \lambda^{k-i}\,.
\]
where the $\phi_j$ are defined through the Taylor expansion
\begin{equation}
\label{eq:Taylor:for:1:Gamma}
\frac{t}{\Gamma(1+t)}
= \frac{1}{\Gamma(t)}
= \sum_{j=1}^{+\infty} \phi_j\cdot\frac{t^j}{j!}\,.
\end{equation}
The first values are given by
\begin{equation*}
\phi_1=1;\quad
\phi_2=2\gamma;\quad
\phi_3=3\big(\gamma^2-\zeta(2)\big).
\end{equation*}
Theorem~\ref{thm:permutation:mod:poisson:introduction} has the following consequence.
\begin{Corollary}
\label{cor:approximation:q:u:introduction}
Uniformly in $k \geq 1$ we have as $n \to \infty$
\[
q_{n, \infty, 1/2}(k) = u_{\lambda_{n}, 1/2}(k) + O \left(\frac{1}{n}\right)
\]
where $\lambda_{n} = \frac{\log(2n)}{2}$.
\end{Corollary}
Theorem~\ref{thm:permutation:mod:poisson:introduction}
and its Corollary~\ref{cor:approximation:q:u:introduction} are
particular cases of Theorem~\ref{cor:approximation:q:u} and
Corollary~\ref{cor:permutation:mod:poisson}
stated and proved in Section~\ref{s:mod:poisson}.
We also illustrate the numerical aspects of
this approximation in Section~\ref{s:numerics}.

\begin{Theorem}
\label{thm:permutation:asymptotics}
Let $\lambda_n = \log(2n)/2$.
Then, for any $x > 0$, we have uniformly in $0 \leq k \leq x \lambda_n$
the following asymptotic behavior as $n\to+\infty$
\begin{multline}
\label{eq:th:permutations:probabiliy:k}
q_{n,\infty,1/2}(k+1)=
\Proba_{n,\infty,1/2}(\NumCycles_n(\sigma) = k + 1)
=\\=
e^{-\lambda_n}
\cdot \frac{(\lambda_n)^k}{k!}
\cdot \left(
\frac{\sqrt{\pi}}{2 \Gamma\left(1 + \frac{k}{2 \lambda_n}\right)}
+ O\left(\frac{k}{(\log n)^2}\right)
\right)\,.
\end{multline}
For any $x > 1$ such that $x \lambda_n$ is an integer
we have
\begin{multline}
\label{eq:th:permutations:tail}
\sum_{k=x\lambda_n+1}^n q_{n,\infty,1/2}(k+1)
=\Proba_{n,\infty,1/2}\big(\NumCycles_n(\sigma)
> x\lambda_n  + 1\big)
=\\=
\frac{(2n)^{-\tfrac{x \log x - x +1}{2}}}
{\sqrt{2\pi\lambda_n x} }
\cdot
\frac{x}{x - 1}
\cdot \left(
\frac{\sqrt{\pi}}{2 \Gamma\left(1 + \frac{x}{2}\right)}
 + O\left( \frac{1}{\log n} \right)\right)\,,
\end{multline}
where the error term is uniform over $x$ in compact subsets of $(1, +\infty)$.
Similarly, for any $0 < x < 1$ such that $x \lambda_n$ is an integer we have
\begin{multline}
\label{eq:th:permutations:head}
\sum_{k=0}^{x\lambda_n} q_{n,\infty,1/2}(k+1)
=\Proba_{n,\infty,1/2}\big(\NumCycles_n(\sigma) \leq x\lambda_n  + 1\big)
=\\=
\frac{(2n)^{-\tfrac{x \log x - x +1}{2}}}
{\sqrt{2\pi\lambda_n x} }
\cdot
\frac{x}{1 - x}
\cdot \left(
\frac{\sqrt{\pi}}{2 \Gamma\left(1 + \frac{x}{2}\right)}
 + O\left( \frac{1}{\log n} \right)\right)\,,
\end{multline}
where the error term is uniform over $x$ in compact subsets of $(0, 1)$.
\end{Theorem}

Theorem~\ref{thm:permutation:asymptotics} is a particular case of Corollary~\ref{cor:multi:harmonic:asymptotic:all:k}
stated and proved in Section~\ref{ss:LD:and:CLT}.
Note that for $x \not= 1$, we have $x \log x - x + 1 > 0$. Hence,
Equations~\eqref{eq:th:permutations:tail} and~\eqref{eq:th:permutations:head}
provide explicit polynomial bounds in $n$ for the tails of the distribution.

\begin{Remark}
\label{rm:lambda:plus:c}
Let $\displaystyle G(x) = \frac{\sqrt{\pi}}{2 \Gamma(1 + \tfrac{x}{2})}$
and define
\begin{equation}
\label{eq:correctional:term}
\shift(x)= \frac{ \log G(x) }{x - 1}.
\end{equation}
Since $\log G(1) = 0$, the function $\shift(x)$ admits a continuous extension at $x=1$
\[
\lim_{x \to 1} \shift(x) = G'(1) = \frac{\gamma}{2} + \log(2) - 1\,,
\]
where $\gamma$ is the Euler--Mascheroni constant. Now for any $x > 0$,
uniformly for $0 \leq k \leq x \lambda_n$ we have
\[
(\lambda_n)^k \cdot G \left( \frac{k}{\lambda_n} \right)
=
e^{-\shift\left(\frac{k}{\lambda_n}\right)} \left(\lambda_n + \shift\left( \tfrac{k}{\lambda_n} \right)\right)^k
\cdot
\left(1 + O\left( \frac{k}{\lambda_n^2} \right) \right) \,.
\]
We can hence rewrite the right-hand side of~\eqref{eq:th:permutations:probabiliy:k}:
for any $x > 0$ we have uniformly in $0 \leq k \leq x \lambda_n$
the following asymptotic behavior as $n\to+\infty$
\[
q_{n,\infty,1/2}(k+1)=
e^{-\left(\lambda_n + \shift\left(\frac{k}{\lambda_n}\right)\right)}
\cdot \frac{\left(\lambda_n + \shift\left( \tfrac{k}{\lambda_n} \right)\right)^k}{k!}
\cdot \left( 1 + O\left(\frac{k}{(\log n)^2}\right) \right)\,.
\]
In the latter expression, the right-hand side reads as the value of a Poisson
random variable with parameter $\lambda_n + \shift\left(\frac{k}{\lambda_n}\right)$.
\end{Remark}

The extended version of the above results as well as the closely
related notion of mod-Poisson convergence are discussed in
Section~\ref{ss:non:uniform:permutations}. The above
theorems follow from singularity analysis
at the boundary of the domain of definition
of holomorphic functions
representing the relevant generating series
performed by H.~Hwang in~\cite{Hwang:PhD}.
\medskip

\noindent
\textbf{Properties of the probability distribution
$\boldsymbol{\ProbaCyls_g(k)}$.}
The key theorems below strongly rely on asymptotic analysis
of the Masur--Veech volume of the moduli space of quadratic
differentials performed by A.~Aggarwal
in~\cite{Aggarwal:intersection:numbers} and on uniform
asymptotic bounds for Witten correlators obtained
in~\cite{Aggarwal:intersection:numbers}.

\begin{Theorem}
\label{thm:mod:poisson:pg}
Let $K_g$ be the random variable satisfying the probability
law~\eqref{eq:same:p:g}.
For all $t\in\C$ such that $|t|<\frac{8}{7}$
the following asymptotic relation is
valid as $g\to+\infty$:
\begin{equation}
\label{eq:mod:poisson:pg:introduction}
\E\left(t^{K_g}\right) =
(6g-6)^{\tfrac{t-1}{2}}
\cdot \frac{t \Gamma(\tfrac{3}{2})}{\Gamma(\tfrac{t}{2})}\
\left(1 + o(1)\right)\,.
\end{equation}
Moreover, for any compact set $U$ in the open disk
$|t|<\frac{8}{7}$ there exists $\delta(U)>0$, such that for
all $t\in U$ the error term has the form
$O(g^{-\delta(U)})$.
\end{Theorem}

Note that the right-hand side of expression~\eqref{eq:mod:poisson:pg:introduction}
is very close to the right-hand side of
the analogous expression~\eqref{eq:mod:Poisson:for:permutations}
from Theorem~\ref{thm:permutation:mod:poisson:introduction}
evaluated at $n=3g-3$.

We expect that the mod-Poisson
convergence~\eqref{eq:mod:poisson:pg:introduction} holds in
a large domain than the disk
$|t|<\frac{8}{7}$. If our guess is correct, the
asymptotics~\eqref{eq:pg:equivalent} below for the
distribution $p_g$ should hold for larger
interval of $x$ than described below. We also expect that
the mod-Poisson
convergence analogous to~\eqref{eq:mod:poisson:pg:introduction}
holds for all non-hyperelliptic components of all strata
of holomorphic quadratic differentials;
see Conjecture~\ref{conj:quadratic:strong:form} in
Section~\ref{s:speculations} for more details.

\begin{Theorem}
\label{thm:pg:asymptotics}

Let $\lambda_{3g-3} = \log(6g-6)/2$. For any $x\in
\left[0, \frac{1}{\log\frac{9}{4}}\right)$ we have uniformly in $0 \leq k \leq x \lambda_{3g-3}$
\begin{multline}
\label{eq:pg:equivalent}
\ProbaCyls_g(k + 1) =
\Proba\big(K_g(\gamma) = k + 1\big)
=\\=
e^{-\lambda_{3g-3}}
\cdot \frac{\lambda_{3g-3}^{k}}{k!}
\cdot \left(
\frac{\sqrt{\pi}}{2 \Gamma\left(1 + \frac{k}{2 \lambda_{3g-3} }\right)}
 + O\left( \frac{k}{(\log g)^2} \right)\right)\,.
\end{multline}
For any $x \in (1, 1.236\,]$ such that $x \lambda_{3g-3}$ is an integer
we have
\begin{multline}
\label{eq:tail:p:1:236}
\sum_{k = x \lambda_{3g-3} + 1}^{3g-3} \ProbaCyls_g(k + 1)
= \Proba\big(K_g(\gamma) > x \lambda_{3g-3} + 1\big)
=\\=
\frac{(6g-6)^{-\tfrac{x \log x - x + 1}{2}}}{ \sqrt{2 \pi \lambda_{3g-3} x} }
\cdot
\frac{x}{x - 1}
\cdot \left(
\frac{\sqrt{\pi}}{2 \Gamma\left(1 + \frac{x}{2} \right)}
 + O\left( \frac{1}{\log g} \right)\right)\,,
\end{multline}
where the error term is uniform over $x$ in compact subsets of $(1, 1.236\,]$. Similarly
for any $x \in (0, 1)$ such that $x \lambda_{3g-3}$ is an integer we have
\begin{multline}
\label{eq:head:p:0:1}
\sum_{k = 0}^{x \lambda_{3g-3}} \ProbaCyls_g(k + 1)
= \Proba\big(K_g(\gamma) \leq x \lambda_{3g-3} + 1\big)
=\\=
\frac{(6g-6)^{-\tfrac{x \log x - x + 1}{2}}}{ \sqrt{2 \pi \lambda_{3g-3} x} }
\cdot
\frac{x}{1 - x}
\cdot \left(
\frac{\sqrt{\pi}}{2 \Gamma\left(1 + \frac{x}{2} \right)}
 + O\left( \frac{1}{\log g} \right)\right)\,,
\end{multline}
where the error term is uniform over $x$ in compact subsets of $(0, 1)$.
Finally,
\begin{multline}
\label{eq:tail:1:4}
\sum_{k=\lfloor 0.09 \log g \rfloor}^{\lceil 0.62\log g\rceil}
\ProbaCyls_g(k)
=\\
= \Proba\Big(0.09 \log g < K_g(\gamma) < 0.62 \log g\Big)
= 1 - O\left((\log g)^{24} g^{-1/4}\right)\,.
\end{multline}
\end{Theorem}

Similarly to Remark~\ref{rm:lambda:plus:c},
Equation~\eqref{eq:pg:equivalent} tells,
in particular, that
any $x$ in the interval $[0, 1.236]$
(which carries, essentially, all but $O(g^{-1/4})$
part of the total mass of the distribution)
and for large $g$, the values
$\ProbaCyls_g(k+1)$
for $k$ in a neighborhood of $x\frac{\log g}{2}$ of size
$o(\log g)$ are uniformly well-approximated by the Poisson
distribution $\Poisson_\lambda(k)$
with parameter $\lambda=\frac{\log(6g-6)}{2}+\shift(x)$,
where $\shift(x)$ is defined in~\eqref{eq:correctional:term}.

The approximation results given in
Theorem~\ref{thm:permutation:mod:poisson:introduction} for
$q_{n,\infty,1/2}$ and in Theorem~\ref{thm:mod:poisson:pg} for $p_g$ imply
an asymptotic expansion of the moments that we present now.
Recall that the \textit{Stirling number of the second
kind}, denoted $S(i,j)$, is the number of ways to partition
a set of $i$ objects into $j$ non-empty subsets.

\begin{Theorem}
\label{th:pg:cumulants}
For any fixed $k\in\N$ the difference between
the $i$-th moments of random variables with
the probability distributions $\ProbaCyls_g$
and $q_{3g-3,\infty,1/2}$ tends to zero as $g\to+\infty$.

Furthermore, the $i$-th cumulant $\kappa_i(K_g(\sigma))$ of
the random variable $K_g$ associated to the probability
distribution $\ProbaCyls_g$ admits the following asymptotic
expansion:
\begin{multline}
\label{eq:k:cumulant}
\kappa_i(K_g)
=\frac{\log(6g-6)}{2} + \frac{\gamma}{2} + \log 2 -
\\
- \sum_{j=2}^i S(i,j)
\cdot
(-1)^{j} \cdot \zeta(j) \cdot (j-1)! \cdot \left(2^{j} - 1 \right)
\cdot \left(\frac{1}{2}\right)^j
+ O\left(\frac{1}{g}\right)
\ \text{as }g\to+\infty\,,
\end{multline}
where $S(i,j)$ are the Stirling numbers of the second kind.
In particular, the mean value $\E(K_g)$ and the variance
$\Var(K_g)$ satisfy:
\begin{align*}
\E(K_g)   & = \kappa_1(K_g) = \frac{\log(6g-6)}{2} + \frac{\gamma}{2} + \log 2 + o(1)\,,\\
\Var(K_g) & = \kappa_2(K_g) = \frac{\log(6g-6)}{2} + \frac{\gamma}{2} + \log 2
- \frac{3}{4} \zeta(2) + o(1)\,,
\end{align*}
where $\gamma = 0.5572\ldots$ denotes the Euler--Mascheroni constant.
The third and the fourth cumulants $\kappa_3(K_g)$ and
$\kappa_4(K_g)$
admit the following asymptotic expansions:
\begin{align*}
\kappa_3(K_g) & = \frac{\log(6g-6)}{2} + \frac{\gamma}{2} + \log 2
- \frac{9}{4} \zeta(2) + \frac{7}{4} \zeta(3) + o(1)\,,\\
\kappa_4(K_g) & = \frac{\log(6g-6)}{2} + \frac{\gamma}{2} + \log 2
- \frac{21}{4} \zeta(2) + \frac{21}{2} \zeta(3)
- \frac{45}{8} \zeta(4) + o(1)\,.
\end{align*}
\end{Theorem}
\medskip

\noindent
\textbf{Other approaches to random multicurves.}
One more interesting aspect of geometry of random
multicurves is the lengths statistics of simple closed
hyperbolic geodesics associated to components of
multicurves of fixed topological type.
M.~Mirzakhani studied in~\cite{Mirzakhani:statistics}
random pants decompositions of a hyperbolic surface of
genus $g$. She considered the orbit $\Mod_g\gamma$ of a
multicurve $\gamma=\gamma_1+\dots+\gamma_{3g-3}$
corresponding to a fixed pants decomposition. Choosing
multicurves in this orbit of hyperbolic length at most $L$
she got a finite collection of multicurves. Letting
$L\to+\infty$ she defined a \textit{random pants
decomposition}. M.~Mirzakhani proved in Theorem~1.2
of~\cite{Mirzakhani:statistics} that under the
normalization $x_i=\frac{\ell(g\cdot\gamma_i)}{L}$
for $i=1,\dots,3g-3$,
the lengths
statistics of components of a random pair of pants has the
limiting density function $const\cdot x_1\dots x_{3g-3}$
with respect to the Lebesgue measure on the unit simplex.
F.~Arana-Herrera and
M.~Liu independently proved
in~\cite{Arana:Herrera:Equidistribution:of:horospheres},
\cite{Arana:Herrera:Counting:multi:geodesics} and
in~\cite{Mingkun} a generalization of this result to
arbitrary multicurves. In terms of square-tiled surfaces
the resulting hyperbolic lengths statistics coincides with
statistics of flat lengths of the waist curves of maximal
horizontal cylinders of the square-tiled surface (see
Section~1.9 in~\cite{DGZZ:volume}). It would be interesting
to study implications of these results to the large genus
limit.

In the regime where one considers simple closed curves of
lengths at most $L$ for any fixed $L>0$ and lets the genus
tend to $+\infty$, a very precise description of the
distribution of lengths was provided by M.~Mirzakhani and
B.~Petri in~\cite{Mirzakhani:Petri}.

It would be interesting to establish relations between
random multicurves and a general framework of random
partitions introduced by A.~M.~Vershik in~\cite{Vershik}.
\medskip

\noindent
\textbf{Random quadrangulations versus random square-tiled
surfaces.} In this article we are concerned with random
square-tiled surfaces, which are a particular case of random
quadrangulations, which are themselves a particular case of
random combinatorial maps (surfaces obtained from gluing
polygons). The two latter families have a much longer
mathematical history. The two important parameters are the
number of polygons $N$ and the genus $g$.

Surfaces obtained by random gluing of polygons have been
studied for a long time. Their enumeration can be traced
back to the works of W.~T.~Tutte \cite{Tutte} for $g=0$ and
of T.~R.~S.~Walsh and A.~B.~Lehman \cite{Walsh:Lehman} for
arbitrary $g$. In particular, their results allow to
compute the probability of getting a closed surface of
genus $g$ as a result of a random pairwise gluing of the
sides of a $2n$-gon. Somewhat later J.~Harer and D.~Zagier
\cite{Harer:Zagier} were able to enumerate genus $g$
gluings of a $2n$-gon in a more explicit and effective way.
This was a crucial ingredient in their computation of the
orbifold Euler characteristic of the moduli space
$\mathcal{M}_g$ of complex algebraic curves.

Surfaces obtained from randomly glued polygons have been
studied since a long time in physics in relation to string
theory and quantum gravity as in the paper of V.~Kazakov,
I.~Kostov, A.~Migdal~\cite{Kazakov:Kostov:Migdal}. In this
approach one often works with surfaces of genus zero
and with several perturbative terms corresponding to surfaces
of low genera.

In the case $g=0$ and $N \to+\infty$,
the Brownian map has been shown to be the scaling limit of
various models of combinatorial maps, see the surveys of
G.~Miermont~\cite{Miermont} and of J.-F.~Le~Gall~\cite{LeGall} and the references therein.
Combinatorial maps also
admit local limits, as proved, in particular, in the papers of O.~Angel and
O.~Schramm~\cite{Angel:Schramm}, of
M.~Krikun~\cite{Krikun}, of P.~Chassaing and
B.~Durhuus~\cite{Chassaing:Durhuus}, of
L.~M\'enard~\cite{Menard}. In higher but fixed genus, the
scaling limits giving rise to higher genera Brownian maps
have been investigated by J.~Bettinelli
in~\cite{Bettinelli10,Bettinelli12}.

Surfaces obtained by gluing polygons without restriction on
the genus have been studied by
R.~Brooks and
E.~Makover in~\cite{Brooks:Makover},
by \mbox{A.~Gamburd} in~\cite{Gamburd}, by S.~Chmutov and
B.~Pittel in~\cite{Chmutov:Pittel},
by A.~Alexeev and
P.~Zograf in~\cite{Alexeev:Zograf},
and by T.~Budzinski,
N.~Curien and
B.~Petri in~\cite{Budzinski:Curien:Petri:a,Budzinski:Curien:Petri:b}.
In this approach the genus $g$ of the resulting surface is
a random variable whose expectation is proportional to the
number of polygons $N$. See also the recent paper of
S.~Shresta~\cite{Shrestha} studying square-tiled surfaces
in a similar context.

Finally, in the regime $g = \theta N$ with $\theta \in
[0,\tfrac{1}{2})$ a local limit has been conjectured by
N.~Curien in~\cite{Curien} and recently proved by
T.~Budzinski and B.~Louf in~\cite{Budzinski:Louf}.

Note that our approach is different from all approaches
mentioned above. We fix the genus of the surface, and
consider square-tiled surfaces tiled with at most $N$
squares (or geodesic multicurves of length bounded by some
large number $L$). We define asymptotic frequencies of
square-tiled surfaces or of geodesic multicurves of a fixed
combinatorial type by passing to the limit when $N$
(respectively $L$) tends to infinity. Only when the
resulting limiting frequencies (probabilities) are already
defined in each individual genus we study their behavior in
the regime when the genus becomes very large. This approach
is natural in the context of dynamics of polygonal
billiards, dynamics of interval exchange transformations and
of translation surfaces, and in the context of geometry and
dynamics on the moduli space of quadratic differentials.

Note also that all but negligible part of our square-tiled
surfaces of genus $g$ have $4g-4$ vertices of valence $6$,
while all other vertices have valence $4$,
and the number of such vertices is incomparably larger
than $g$. This is one more
substantial difference between our random surface model and
the random quadrangulations considered in the probability
theory literature where, usually, there is no such degree
constraint imposed and vertices, typically, have arbitrary
degrees even if the resulting surface has genus $0$. As a
result, our square-tiled surfaces locally look like a
tiling of $\R^2$ by squares except around $4g-4$ conical
singularities with cone angle $3\pi$. This is not the case
for a random planar quadrangulation.

A regime similar to ours was used by H.~Masur, K.~Rafi and
A.~Randecker who studied in~\cite{Masur:Rafi:Randecker} the
covering radius of random translation surfaces
(corresponding to Abelian differentials)
and by M.~Mirzakhani, who studied in~\cite{Mirzakhani:random}
random hyperbolic surfaces of fixed large genus $g$.
\medskip

\noindent\textbf{Structure of the paper.~}
To make the current paper self-contained, we reproduce in
Section~\ref{s:Background:material} all necessary
background material. We start by recalling in
Section~\ref{ss:MV:volume} the definition of the
Masur--Veech volume of the moduli space of quadratic
differentials $\cQ_g$. We sketch in
Section~\ref{ss:Square:tiled:surfaces:and:associated:multicurves}
how Masur--Veech volumes are related to count of
square-tiled surfaces. In the same section we associate to
every square-tiled surface a multicurve and
we recall the notion of a stable graph, particularly
important in the framework of the current paper. We present
in Section~\ref{ss:intro:Masur:Veech:volumes} the formula
for the Masur--Veech volume $\Vol\cQ_g$ and a theorem of
A.~Aggarwal on the asymptotic value of this volume for
large genera $g$. The reader interested in more ample
information is addressed to the original
papers~\cite{DGZZ:volume}
and~\cite{Aggarwal:intersection:numbers} respectively.
In Section~\ref{ss:Frequencies:of:simple:closed:curves} we
recall Mirzakhani's
count~\cite{Mirzakhani:grouth:of:simple:geodesics} of
frequencies of multicurves. In
Section~\ref{ss:Frequencies:of:square:tiled:surfaces} we
explain why Questions~\ref{question:multicurves}
and~\ref{question:square:tiled} are equivalent and
demystify Theorem~\ref{th:same:distribution}. In
Section~\ref{ss:conjecture:on:correlators} we recall the
recent breakthrough results of
A.~Aggarwal~\cite{Aggarwal:intersection:numbers} on large
genus asymptotics of Witten correlators.

In Section~\ref{s:sum:over:single:vertex:graphs} we recall
general background from the works of H.~K.~Hwang~\cite{Hwang:PhD},
and of E.~Kowalski, P.-L.~M\'eliot, A.~Nikeghbali,
D.~Zeindler~\cite{KowalskiNikeghbali},
\cite{NikeghbaliZeindler}, \cite{FerayMeliotNikeghbali} on
random permutations and on mod-Poisson convergence and
apply this general technique to the probability
distribution $q_{3g-3,\infty,1/2}$. In particular, we prove
Theorems~\ref{thm:permutation:mod:poisson:introduction}
and~\ref{thm:permutation:asymptotics}.

We then introduce a probability distribution
$\ProbaCylsOne_g(k)$ of the random variable
$K_g(\gamma)=K_g(S)$ restricted to non-separating random
multicurves $\gamma$ on a surface of genus $g$
(equivalently restricted to random square-tiled surfaces of
genus $g$ having single horizontal critical level). Using
the results of
A.~Aggarwal~\cite{Aggarwal:intersection:numbers} on
asymptotics of Witten correlators we prove that the
distribution $q_{3g-3,\infty,1/2}$ very well-approximates the
distribution $\ProbaCylsOne_g$ (namely, that they share the
same mod-Poisson convergence but $\ProbaCylsOne_g$
has smaller radius of convergence). This allows us to extend
all the results obtained for random permutations
to these special random multicurves (special
random square-tiled surfaces).

It remains, however, to pass from the special
multi-curves (and square-tiled surfaces) to general ones.
The necessary estimates are prepared in
Section~\ref{s:disconnecting:multicurves}. In a sense, this
step was already performed by A.~Aggarwal
in~\cite{Aggarwal:intersection:numbers}, who proved a
generalization of our conjecture from~\cite{DGZZ:volume}
claiming that random multicurves (random square-tiled
surfaces) which do not contribute to the distribution
$\ProbaCylsOne_g$ become rare in large genera. This
justifies the fact that the distribution $\ProbaCylsOne_g$
well-approximates the distribution $\ProbaCyls_g$. However,
to prove this statement in a much stronger form stated in
the current paper we have to adjust certain estimates from
Sections~9 and~10 from the original
paper~\cite{Aggarwal:intersection:numbers} to our current
needs.

We recommend to readers interested in all details of
Section~\ref{s:disconnecting:multicurves} to read it in
parallel with Sections~9 and~10 of the original
paper~\cite{Aggarwal:intersection:numbers}. (Actually, we
recommend reading the entire
paper~\cite{Aggarwal:intersection:numbers} of A.~Aggarwal.
We have no doubt that the reader looking for a deep
understanding of the subject would appreciate beauty,
strength and originality of the proofs and ideas
in~\cite{Aggarwal:intersection:numbers} as we do.)

Having obtained all necessary estimates in
Section~\ref{s:disconnecting:multicurves} we prove in
Section~\ref{s:proofs} that the distribution
$\ProbaCylsOne_g$ well-approximates the distribution
$\ProbaCyls_g$. By transitivity this implies that the
distribution $q_{3g-3,\infty,1/2}$ well-approximates the
distribution $\ProbaCyls_g$. We show in
Section~\ref{s:proofs} how the
properties of $q_{3g-3,\infty,1/2}$ derived in
Section~\ref{s:sum:over:single:vertex:graphs} imply all our
main results.

In Section~\ref{s:numerics} we compare our theoretical
results with experimental and numerical data. We complete
by suggesting in Section~\ref{s:speculations} a conjectural
description of the combinatorial geometry of random Abelian
square-tiled surfaces of large genus and of random
square-tiled surfaces restricted to any non-hyperelliptic component of
any stratum in the moduli space of Abelian or quadratic
differentials of large genus.

This article is born from Appendices~D--F of the original
preprint~\cite{DGZZ:volume}. The latter contained several
conjectures and derived from them all other results as
``conditional theorems''. All these conjectures were proved
by A.~Aggarwal; see
Theorems~\ref{conj:Vol:Qg},
\ref{th:correlators:upper:bound},
\ref{th:asymptotics:of:correlators:lower},
\ref{th:asymptotics:of:correlators:upper},
and Corollary~\ref{conj:one:vertex:dominates:for:fixed:k}
in the current
paper or Theorems~1.7 and
Propositions~1.2, 4.1, 4.2, 10.7
respectively in the original
paper~\cite{Aggarwal:intersection:numbers}.
Moreover, most of the results are proved
in~\cite{Aggarwal:intersection:numbers} in a much stronger
form than we initially conjectured. Combining our initial
approach with these recent results of A.~Aggarwal and
elaborating close ties with random permutations allowed us
to radically strengthen the initial assertions
from~\cite{DGZZ:volume}.
\medskip

\noindent\textbf{Acknowledgements.}
We are very much indebted to A.~Aggarwal for transforming
our dreams into reality by proving all our conjectures
from~\cite{DGZZ:volume}. We also very much appreciate his
advices, including the indication on how to compute
multi-variate harmonic sums, which was crucial for making
correct predictions in~\cite{DGZZ:volume}. His numerous
precious comments on the preliminary versions of this paper
allowed us to correct a technical mistake and numerous
typos and improve the presentation.

Results of this paper were directly or indirectly
influenced by beautiful and deep ideas of
Maryam~Mirzakhani.

We thank S.~Schleimer, who was the first person to notice
that our experimental data on statistics of cylinder
decompositions of random Abelian square-tiled surfaces seems
to have resemblance with statistics of cycle decomposition
of random permutations.

We thank F.~Petrov for the reference to the
paper~\cite{Goncharov} in the context of cycle
decomposition of random permutations.

We thank M.~Bertola, A.~Borodin, G.~Borot, D.~Chen,
A.~Eskin, V.~Feray, M.~Kazarian, S.~Lando, M.~Liu,
H.~Masur, M.~M\"oller, B.~Petri, K.~Rafi, A.~Sauvaget,
J.~Souto, D.~Zagier and D.~Zvonkine for useful discussions.

We thank B.~Green for the talk at the conference ``CMI at
20'' and T.~Tao for his blog both of which were very
inspiring for us.

We thank D.~Calegari for kind permission to use a picture
from his book~\cite{Calegari} in Figure~\ref{fig:multicurve}.

We are grateful to MPIM in Bonn, where part of this work
was performed, to Chebyshev Laboratory in St.~Petersburg
State University, to MSRI in Berkeley and to MFO in
Oberwolfach for providing us with friendly and stimulating
environment.

\section{Background material}
\label{s:Background:material}

\subsection{Masur--Veech volume of the moduli space of
quadratic differentials}
\label{ss:MV:volume}
Consider the moduli space $\cM_{g,n}$ of complex curves of
genus $g$ with $n$ distinct labeled marked points. The
total space $\cQ_{g,n}$ of the cotangent bundle over
$\cM_{g,n}$ can be identified with the moduli space of
pairs $(C,q)$, where $C\in\cM_{g,n}$ is a smooth complex
curve with $n$ (labeled) marked points and $q$ is a meromorphic quadratic differential on
$C$ with at most simple poles at the marked points and no other
poles. In the case $n=0$ the quadratic differential $q$ is
holomorphic. Thus, the \textit{moduli space of quadratic
differentials} $\cQ_{g,n}$ is endowed with the canonical
symplectic structure. The induced volume element $\dVolMV$
on $\cQ_{g,n}$ is called the \textit{Masur--Veech volume
element}. (In the next Section
we provide alternative more common definition
of the Masur--Veech volume element.)

A non-zero quadratic differential $q$ in $\cQ_{g,n}$
defines a flat metric $|q|$ on the complex curve $C$. The
resulting metric has conical singularities at zeroes and
simple poles of $q$. The total area of $(C,q)$
$$
\Area(C,q)=\int_C |q|
$$
is positive and finite. For any real $a > 0$, consider the following
subset in $\cQ_{g,n}$:
$$
\cQ^{\Area\le a}_{g,n} := \left\{(C,q)\in\cQ_{g,n}\,|\, \Area(C,q) \le a\right\}\,.
$$
Since $\Area(C,q)$ is a norm in each fiber of the bundle
$\cQ_{g,n} \to \cM_{g,n}$, the set $\cQ^{\Area \le a}_{g,n}$
is a ball bundle over $\cM_{g,n}$. In particular, it is non-compact.
However, by the independent results of H.~Masur~\cite{Masur:82} and
W.~Veech~\cite{Veech:Gauss:measures}, the total mass of $\cQ^{\Area\le
a}_{g,n}$ with respect to the Masur--Veech volume element is finite.
Following a common convention we define the Masur--Veech volume $\Vol\cQ_{g,n}$
as
\begin{equation}
\label{eq:def:Vol:Q:g:n}
\Vol\cQ_{g,n}
=(12g-12+4n)\cdot\Vol \cQ^{\Area\le\frac{1}{2}}_{g,n}\,.
\end{equation}

\subsection{Square-tiled surfaces, simple
closed multicurves and stable graphs}
\label{ss:Square:tiled:surfaces:and:associated:multicurves}

We have already mentioned that a non-zero meromorphic
quadratic differential $q$ on a complex curve $C$ defines a
flat metric with conical singularities. One can construct a
discrete collection of quadratic differentials of this kind
by assembling together identical flat squares in the
following way. Take a finite set of copies of the oriented
$1/2 \times 1/2$-square for which two opposite sides are
chosen to be horizontal and the remaining two sides are
declared to be vertical. Identify pairs of sides of the
squares by isometries in such way that horizontal sides are
glued to horizontal sides and vertical sides to vertical.
We get a topological surface $S$ without boundary. We
consider only those surfaces obtained in this way which are
connected and oriented. The form $dz^2$ on each square is
compatible with the gluing and endows $S$ with a complex
structure and with a non-zero quadratic differential $q$
with at most simple poles. The total area $\Area(S,q)$ is
$\frac{1}{4}$ times the number of squares. We call such
surface a \textit{square-tiled surface}.

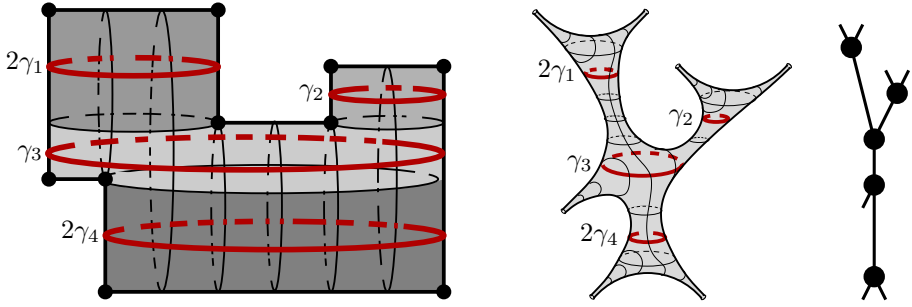
\begin{figure}[htb]
   %
   %
\includegraphics{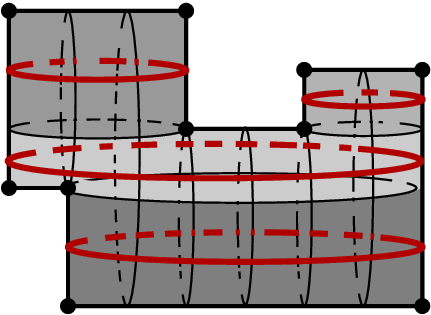}
\includegraphics{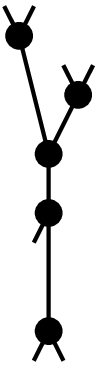}

\begin{picture}(0,0)(175,10)
\put(2,-13){$2\gamma_1$}
\put(113,-23){$\gamma_2$}
\put(1.5,-46){$\phantom{2}\gamma_3$}
\put(23,-77){$2\gamma_4$}
\end{picture}

\includegraphics{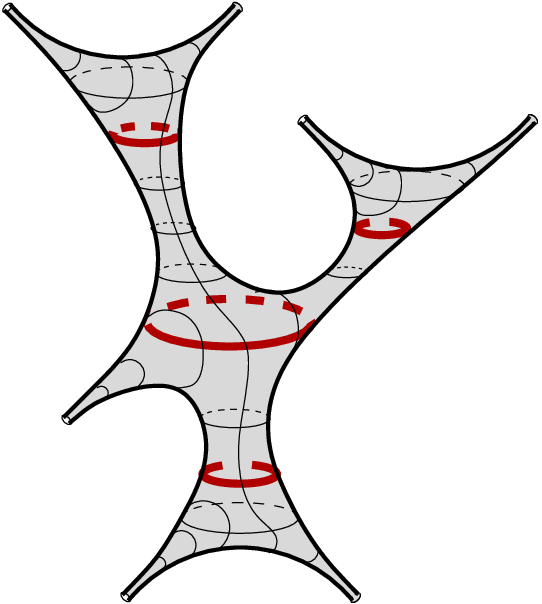}

\begin{picture}(0,0)(-26,-6)
\put(4,-20){$2\gamma_1$}
\put(53.5,-36.5){$\gamma_2$}
\put(15,-54){$\gamma_3$}
\put(20,-82){$2\gamma_4$}
\end{picture}
\vspace{95pt}
\caption{
\label{fig:square:tiled:surface:and:associated:multicurve}
Square-tiled surface in $\cQ_{0,7}$,
and associated multicurve and
stable graph
}
\end{figure}

Suppose that the resulting closed \textit{square-tiled
surface} has genus $g$ and $n$ conical singularities with
angle $\pi$, i.e. $n$ vertices adjacent to only two
squares. For example, the square-tiled surfaces in
Figure~\ref{fig:square:tiled:surface:and:associated:multicurve}
has genus $g=0$ and $n=7$ conical singularities with angle
$\pi$. Consider the complex coordinate $z$ in each square
and a quadratic differential $(dz)^2$. It is easy to check
that the resulting square-tiled surface inherits the
complex structure and globally defined meromorphic
quadratic differential $q$ having simple poles at $n$
conical singularities with angle $\pi$ and no other poles.
Thus, any square-tiled surface of genus $g$ having $n$
conical singularities with angle $\pi$ canonically defines
a point $(C,q)\in\cQ_{g,n}$. Fixing the size of the square
once and forever and considering all resulting square-tiled
surfaces in $\cQ_{g,n}$ we get a discrete subset $\cSTgn$
in $\cQ_{g,n}$.

Define $\cSTgn(N)\subset\cSTgn$ to be the subset of
square-tiled surfaces in $\cQ_{g,n}$
tiled with at most $N$ identical squares.
Square-tiled surfaces form a lattice in period coordinates
of $\cQ_{g,n}$, which justifies the following alternative
definition of the Masur--Veech volume:
\begin{equation}
\label{eq:Vol:sq:tiled}
\Vol\cQ_{g,n}
= 2(6g-6+2n)\cdot
\lim_{N\to+\infty}
\frac{\card(\cSTgn(2N))}{N^{d}}\,,
\end{equation}
where $d=6g-6+2n=\dim_{\C}\cQ_{g,n}$.
In this formula we assume that \textit{all} conical singularities
of square-tiled surfaces are labeled
(i.e., counting square-tiled surfaces we label not only $n$ simple poles but also all zeroes).
\smallskip

\noindent\textbf{Multicurve associated to a cylinder decomposition.}
Any square-tiled surface admits a decomposition into
maximal horizontal cylinders filled with isometric closed
regular flat geodesics. Every such maximal horizontal
cylinder has at least one conical singularity on each of
the two boundary components. The square-tiled surface in
Figure~\ref{fig:square:tiled:surface:and:associated:multicurve}
has four maximal horizontal cylinders which are represented
in the picture by different shades. For every maximal
horizontal cylinder choose the corresponding waist curve
$\gamma_i$.

By construction each resulting simple closed curve
$\gamma_i$ is non-periferal (i.e. it does not bound a
topological disk without punctures or with a single
puncture) and different $\gamma_i, \gamma_j$ are not freely
homotopic on the underlying $n$-punctured topological
surface. In other words, pinching simultaneously all waist
curves $\gamma_i$ we get a legal stable curve in the
Deligne--Mumford compactification $\overline{\cM}_{g,n}$.

We encode the number of circular horizontal bands of
squares contained in the corresponding maximal horizontal
cylinder by the integer weight $H_i$ associated to the
curve $\gamma_i$. The above observation implies that the
resulting formal linear combination $\gamma=\sum
H_i\gamma_i$ is a simple closed integral multicurve in the
space $\cML_{g,n}(\Z)$ of measured laminations. For
example, the simple closed multicurve associated to the
square-tiled surface as in
Figure~\ref{fig:square:tiled:surface:and:associated:multicurve}
has the form $2\gamma_1+\gamma_2+\gamma_3+2\gamma_4$.

Given a simple closed integral multicurve $\gamma$ in
$\cML_{g,n}(\Z)$ consider the subset
$\cSTgn(\gamma)\subset\cSTgn$ of those square-tiled
surfaces, for which the associated horizontal multicurve is
in the same $\Mod_{g,n}$-orbit as $\gamma$ (i.e. it is
homeomorphic to $\gamma$ by a homeomorphism sending $n$
marked points to $n$ marked points and preserving their
labeling). Denote by $\Vol(\gamma)$ the contribution to
$\Vol\cQ_{g,n}$ of square-tiled surfaces from the subset
$\cSTgn(\gamma)\subset\cSTgn$:
$$
\Vol(\gamma)
= 2(6g-6+2n)\cdot
\lim_{N\to+\infty}
\frac{\card(\cSTgn(2N)\cap\cSTgn(\gamma))}{N^{d}}\,.
$$
The results in~\cite{DGZZ:meanders:and:equidistribution}
imply that for any $\gamma$ in $\cML_{g,n}(\Z)$ the above
limit exists, is strictly positive, and that
\begin{equation}
\label{eq:Vol:Q:as:sum:of:Vol:gamma}
\Vol\cQ_{g,n}
=\sum_{[\gamma]\in\mathcal{O}} \Vol(\gamma)\,,
\end{equation}
where the
sum is taken over representatives $[\gamma]$ of all orbits
$\mathcal{O}$ of the mapping class group $\Mod_{g,n}$ in
$\cML_{g,n}(\Z)$.

\begin{Definition}
\label{def:asymptotic:probability}
Formula~\eqref{eq:Vol:Q:as:sum:of:Vol:gamma} allows to
interpret the ratio $\Vol(\gamma)/\Vol\cQ_{g,n}$ as the
\textit{asymptotic probability} to get a square-tiled surface in
$\cSTgn(\gamma)$ taking a random square-tiled surface in
$\cSTgn(N)$ as $N\to+\infty$. We will also call the same quantity
by the \textit{frequency} of square-tiled surfaces of
type $\cSTgn(\gamma)$ among all square-tiled surfaces.
\end{Definition}
\smallskip

\noindent\textbf{Stable graph associated to a multicurve.}
Following M.~Kontsevich~\cite{Kontsevich} we assign to any
multicurve $\gamma$ a \textit{stable graph}
$\Gamma(\gamma)=\Gamma(\gamma_{\mathit{reduced}})$.
The stable graph $\Gamma(\gamma)$ is a decorated graph dual
to $\gamma_{reduced}$. It consists of vertices, edges, and
``half-edges'' also called ``legs''. Vertices of
$\Gamma(\gamma)$ represent the connected components of the
complement $S_{g,n}\setminus\gamma_{\mathit{reduced}}$.
Each vertex is decorated with the integer number recording
the genus of the corresponding connected component of
$S_{g,n}\setminus\gamma_{reduced}$.
By convention, when this number is not explicitly indicated, it equals to zero.
Edges of
$\Gamma(\gamma)$ are in the natural bijective
correspondence with curves $\gamma_i$; an edge joins a
vertex to itself when on both sides of the corresponding
simple closed curve $\gamma_i$ we have the same connected
component of $S_{g,n}\setminus\gamma_{reduced}$. Finally,
the $n$ punctures are encoded by $n$ \textit{legs}. The
right picture in
Figure~\ref{fig:square:tiled:surface:and:associated:multicurve}
provides an example of the stable graph associated to the
multicurve $\gamma$.

Pinching a complex curve of genus $g$ with $n$ marked points by all components of a reduced multicurve $\gamma_{reduced}$ we get a stable complex curve representing a point in the Deligne--Mumford compactification $\overline{\mathcal{M}}_{g,n}$. In this way stable graphs
encode the boundary cycles of
$\overline{\mathcal{M}}_{g,n}$. In particular, the set
$\cG_{g,n}$ of all stable graphs is finite. It is in the
natural bijective correspondence with boundary cycles of
$\overline{\mathcal{M}}_{g,n}$ or, equivalently, with
$\Mod_{g,n}$-orbits of reduced multicurves in
$\cML_{g,n}(\Z)$.

\subsection{Formula for the Masur--Veech volumes}
\label{ss:intro:Masur:Veech:volumes}

In this section we introduce polynomials
$N_{g,n}(b_1, \ldots, b_n)$ that appear in different
contexts, in particular, in the formula
for the Masur--Veech volume.

Let $g$ be a non-negative integer and $n$ a positive integer. Let the
pair $(g,n)$ be different from $(0,1)$ and $(0,2)$. Let
$d_1,\dots,d_n$ be an ordered partition of $3g - 3 + n$ into a sum of
non-negative integers, $|d|=d_1+\dots+d_n=3g-3+n$, let $\boldsymbol{d}$
be a multiindex $(d_1,\dots,d_n)$ and let $b^{2\boldsymbol{d}}$
denote $b_1^{2d_1}\cdot\cdots\cdot b_n^{2d_n}$.

Define the following homogeneous polynomial
$N_{g,n}(b_1,\dots,b_n)$ of degree
$6g-6 + 2n$ in variables
$b_1,\dots,b_n$ in the following way.
\begin{equation}
\label{eq:N:g:n}
N_{g,n}(b_1,\dots,b_n)=
\sum_{|d|=3g-3+n}c_{\boldsymbol{d}} b^{2\boldsymbol{d}}\,,
\end{equation}
where
\begin{equation}
\label{eq:c:subscript:d}
c_{\boldsymbol{d}}=\frac{1}{2^{5g-6+2n}\, \boldsymbol{d}!}\,
\langle \tau_{d_1} \dots \tau_{d_n}\rangle_{g,n}
\end{equation}
\begin{equation}
\label{eq:correlator}
\langle \tau_{d_1} \dots \tau_{d_n}\rangle_{g,n}
=\int_{\overline{\cM}_{g,n}} \psi_1^{d_1}\dots\psi_n^{d_n}\,,
\end{equation}
and $\boldsymbol{d}!=d_1!\cdots d_n!$. Note that $N_{g,n}(b_1,\dots,b_n)$
contains only even powers of $b_i$, where $i=1,\dots,n$.

Following~\cite{AEZ:genus:0} we consider the following
linear operators $\cY(\boldsymbol{H})$ and $\cZ$ on the spaces of
polynomials in variables $b_1,b_2,\dots$, where $H_1, H_2, \dots$
are positive integers.
The operator $\cY(\boldsymbol{H})$ is defined on monomials as
\begin{equation}
\label{eq:cV}
\cY(\boldsymbol{H})\ :\quad
\prod_{i=1}^{k} b_i^{m_i} \longmapsto
\prod_{i=1}^{k} \frac{m_i!}{H_i^{m_i+1}}\,,
\end{equation}
and extended to arbitrary polynomials by linearity.
The operator $\cZ$ is defined on monomials as
\begin{equation}
\label{eq:cZ}
\cZ\ :\quad
\prod_{i=1}^{k} b_i^{m_i} \longmapsto
\prod_{i=1}^{k} \big(m_i!\cdot \zeta(m_i+1)\big)\,,
\end{equation}
and extended to arbitrary polynomials by linearity.

Given a stable graph $\Graph$ denote by $V(\Gamma)$ the set
of its vertices and by $E(\Gamma)$ the set of its edges. To
each stable graph $\Gamma\in\cG_{g,n}$ we associate the
following homogeneous polynomial $P_\Gamma$
of degree $6g-6+2n$. To
every edge $e\in E(\Gamma)$ we assign a formal variable
$b_e$. Given a vertex $v\in V(\Gamma)$ denote by $g_v$ the
integer number decorating $v$ and denote by $n_v$ the
valency of $v$, where the legs adjacent to $v$ are counted
towards the valency of $v$. Take a small neighborhood of
$v$ in $\Gamma$. We associate to each half-edge (``germ''
of edge) $e$ adjacent to $v$ the monomial $b_e$; we
associate $0$ to each leg. We denote by $\boldsymbol{b}_v$
the resulting collection of size $n_v$. If some edge $e$ is
a loop joining $v$ to itself, $b_e$ would be present in
$\boldsymbol{b}_v$ twice; if an edge $e$ joins $v$ to a
distinct vertex, $b_e$ would be present in
$\boldsymbol{b}_v$ once; all the other entries of
$\boldsymbol{b}_v$ correspond to legs; they are represented
by zeroes. To each vertex $v\in E(\Gamma)$ we associate the
polynomial $N_{g_v,n_v}(\boldsymbol{b}_v)$, where $N_{g,v}$
is defined in~\eqref{eq:N:g:n}. We associate to the stable
graph $\Gamma$ the polynomial obtained as the product
$\prod b_e$ over all edges $e\in E(\Graph)$ multiplied by
the product $\prod N_{g_v,n_v}(\boldsymbol{b}_v)$ over all
$v\in V(\Graph)$. We define $P_\Gamma$ as follows:
\begin{multline}
\label{eq:P:Gamma}
P_\Gamma(\boldsymbol{b})
=
\frac{2^{6g-5+2n} \cdot (4g-4+n)!}{(6g-7+2n)!}\cdot
\\
\frac{1}{2^{|V(\Graph)|-1}} \cdot
\frac{1}{|\operatorname{Aut}(\Graph)|}
\cdot
\prod_{e\in E(\Graph)}b_e\cdot
\prod_{v\in V(\Graph)}
N_{g_v,n_v}(\boldsymbol{b}_v)
\,.
\end{multline}

\begin{NNTheorem}[\cite{DGZZ:volume}]
   %
The Masur--Veech volume $\Vol \cQ_{g,n}$ of the stratum of quadratic differentials with $4g-4+n$
simple zeros and $n$ simple poles has the following value:
\begin{equation}
\label{eq:square:tiled:volume}
\Vol \cQ_{g,n}
= \sum_{\Graph \in \cG_{g,n}} \Vol(\Gamma)\,,
\end{equation}
where the contribution of an individual stable graph
$\Gamma$ has the form
\begin{equation}
\label{eq:volume:contribution:of:stable:graph}
\Vol(\Gamma)=\cZ(P_\Gamma)\,.
\end{equation}
\end{NNTheorem}

\begin{Remark}
\label{rk:volume:contribution}
The contribution~\eqref{eq:volume:contribution:of:stable:graph}
of any individual stable graph
has the following natural interpretation.
We have seen that stable graphs $\Gamma$ in $\cG_{g,n}$
are in natural bijective correspondence
with $\Mod_{g,n}$-orbits of \textit{reduced} multicurves
$\gamma_{\mathit{reduced}}=\gamma_1+\gamma_2+\dots$,
where simple closed curves $\gamma_i$ and $\gamma_j$
are not isotopic for any $i\neq j$.
Let $\Gamma\in\cG_{g,n}$, let $k=|V(\Gamma)|$, let
$\gamma_{\mathit{reduced}}=\gamma_1+\dots+\gamma_k$ be the reduced multicurve
associated to $\Gamma$. Let $\gamma_{\boldsymbol{H}}
=\gamma(\Gamma,\boldsymbol{H})=
H_1\gamma_1+\dots H_k\gamma_k$, where
$\boldsymbol{H}=(H_1,\dots,H_k)\in\N^k$.
We have
\begin{equation}
\label{eq:Vol:Gamma}
\Vol(\Gamma)=
\sum_{H\in\N^k}
\Vol\big(\Gamma,\boldsymbol{H}\big)
\end{equation}
where the contribution
$\Vol\big(\Gamma,\boldsymbol{H})$ of
square-tiled surfaces with the horizontal cylinder
decomposition of type $(\Gamma,\boldsymbol{H})$
to $\Vol\cQ_{g,n}$ is given
by the formula:
\begin{equation}
\label{eq:contribution:of:gamma:to:volume}
\Vol\big(\Gamma,\boldsymbol{H}\big)
=\cY(\boldsymbol{H})(P_\Gamma)\,.
\end{equation}
\end{Remark}

In other words, we can rearrange the sum in~\eqref{eq:Vol:Q:as:sum:of:Vol:gamma}
as
\begin{equation}
\label{eq:Vol:Q:g:n:as:sum}
\Vol\cQ_{g,n}
=\sum_{[\gamma]\in\mathcal{O}} \Vol(\gamma)
=\sum_{\Gamma\in\cG_{g,n}}
\sum_{\{[\gamma]\,| \Gamma(\gamma)=\Gamma\}} \Vol(\gamma)
\,,
\end{equation}
where
$$
\sum_{\{[\gamma]\,|\, \Gamma(\gamma)=\Gamma\}} \Vol(\gamma) = \Vol(\Gamma)\,.
$$
In this way we can extend
Definition~\ref{def:asymptotic:probability} and speak of
\textit{asymptotic probability} of getting a square-tiled
surface in $\cSTgn(\Gamma)=\cup_{\{[\gamma]\,|\,
\Gamma(\gamma)=\Gamma\}}\cSTgn(\gamma)$ taking a random
square-tiled surface in $\cSTgn(N)$ as $N\to+\infty$. In
the same way we define \textit{frequency} of square-tiled
surfaces having exactly $k$ maximal  horizontal cylinders
among all square-tiled surfaces of genus $g$.

In particular, we define the quantity
$\Proba\big(K_g(S)=k\big)$ from Equation~\eqref{eq:same:p:g}
as
\begin{equation}
\label{eq:Proba:K:g:S}
\Proba\big(K_g(S)=k\big)=
\frac{1}{\Vol\cQ_g}\cdot
\sum_{\substack{\Gamma \in \cG_g\\|V(\Gamma)| = k}}
\Vol(\Gamma)\,.
\end{equation}

We complete this section with the theorem which is one of
the two keystone results on which rely all further
asymptotic results of the current paper. Morally, it serves
to establish explicit normalization allowing to pass from a
finite measure with unspecified total mass to a specific
probability measure. This statement was conjectured
in~\cite{DGZZ:volume} and proved in~\cite[Theorem
1.7]{Aggarwal:intersection:numbers}.

\begin{Theorem}[A.~Aggarwal~\cite{Aggarwal:intersection:numbers}]
\label{conj:Vol:Qg}
The Masur--Veech volume of the moduli space of holomorphic
quadratic differentials has the following large genus
asymptotics:
\begin{equation}
\label{eq:Vol:Qg}
\Vol\cQ_g
=
\frac{4}{\pi}
\cdot\left(\frac{8}{3}\right)^{4g-4}\cdot\big(1+o(1)\big)
\quad\text{as }g\to+\infty\,.
\end{equation}
\end{Theorem}

\begin{Remark}
\label{rm:expansion:of:error:term}
The exact values of $\Vol\cQ_g$ for $g\le 250$ (and more)
can be obtained by combining results of D.~Chen,
M.~M\"oller, A.~Sauvaget~\cite{Chen:Moeller:Sauvaget} with
the results of M.Kazarian~\cite{Kazarian} or with the
results of D.~Yang, D.~Zagier and
Y.~Zhang~\cite{Yang:Zagier:Zhang}. Supported by serious
data analysis, the authots of~\cite{Yang:Zagier:Zhang}
conjecture that the error term in~\eqref{eq:Vol:Qg} admits
an asymptotic expansion in $g^{-1}$ with the leading term
$-\frac{\pi^2}{144}\cdot\frac{1}{g}$ and with explicit
coefficients for the terms $g^{-2}$ and $g^{-3}$. In
Theorem~\ref{thm:generating:series:vol}, using a refinement
of the estimates from~\cite{Aggarwal:intersection:numbers}
we prove that the error term $o(1)$ in~\ref{eq:Vol:Qg} can
be improved to a finer estimate $O(g^{-1/4})$.

Conjectural generalization of formula~\eqref{eq:Vol:Qg} to
all strata of meromorphic quadratic differentials and
numerical evidence beyond this conjecture are presented
in~\cite{ADGZZ:conjecture}. Actually, \cite[Theorem
1.7]{Aggarwal:intersection:numbers} proves the volume
asymptotics in the more general setting for $\Vol\cQ_{g,n}$
under assumption that the number $n$ of simple poles
satisfies the relation $20n<\log g$.
\end{Remark}

\subsection{Frequencies of multicurves (after M.~Mirzakhani)}
\label{ss:Frequencies:of:simple:closed:curves}

Recall that two integral multicurves on the same
smooth surface of genus $g$ with $n$ punctures
\textit{have the same topological type}
if they belong to the same orbit of the mapping class group
$\Mod_{g,n}$.

We change now flat setting to hyperbolic setting. Following
M.~Mirzakhani, given an integral multicurve $\gamma$ in
$\cML_{g,n}(\Z)$ and a hyperbolic surface $X\in\cT_{g,n}$
consider the function $s_X(L,\gamma)$ counting the number
of simple closed geodesic multicurves on $X$ of length at
most $L$ of the same topological type as $\gamma$.
M.~Mirzakhani proves
in~\cite{Mirzakhani:grouth:of:simple:geodesics} the
following Theorem.
\begin{NNTheorem}[M.~Mirzakhani]
For any rational multi-curve $\gamma$ and
any hyperbolic surface $X\in\cT_{g,n}$,
\begin{equation}
\label{eq:frequency:c}
s_X(L,\gamma)\sim B(X)\cdot\frac{c(\gamma)}{b_{g,n}}\cdot L^{6g-6+2n}\,,
\end{equation}
as $L\to+\infty$.
\end{NNTheorem}

The factor $B(X)$ in the above formula has the following geometric meaning. Consider the unit ball
$B_X=\{\gamma\in\cML_{g,n}\,|\,\ell_X(\gamma)\le 1\}$
defined by means of the length function $\ell_X$. The
factor $B(X)$ is the Thurston's measure of $B_X$:
$$
B(X)=\mu_{\mathrm{Th}}(B_X)\,.
$$

The factor $b_{g,n}$ is
defined as the average of $B(X)$ over $\cM_{g,n}$ viewed as
the moduli space of hyperbolic metrics, where the average
is taken with respect to the Weil--Petersson volume form on
$\cM_{g,n}$:
\begin{equation}
\label{eq:b:g:n}
b_{g,n}=\int_{\cM_{g,n}} B(X)\,dX\,.
\end{equation}

Mirzakhani showed that
\begin{equation}
\label{eq:b:g:n:as:sum:of:c:gamma}
b_{g,n}=\sum_{[\gamma]\in\mathcal{O}(g,n)} c(\gamma)\,,
\end{equation}
where the sum of $c(\gamma)$ taken with respect to
representatives $[\gamma]$ of all orbits $\mathcal{O}(g,n)$ of
the mapping class group $\Mod_{g,n}$ in $\cML_{g,n}(\Z)$.
This allows to interpret the ratio
$\tfrac{c(\gamma)}{b_{g,n}}$ as the probability to get a
multicurve of type $\gamma$ taking a ``large random''
multicurve (in the same sense as the probability that
coordinates of a ``random'' point in $\Z^2$ are coprime
equals $\tfrac{6}{\pi^2}$).

In particular, we define the quantity
$\Proba\big(K_g(\gamma)=k\big)$ from
Equation~\eqref{eq:same:p:g} as
\begin{equation}
\label{eq:Proba:K:g:gamma}
\Proba\big(K_g(\gamma)=k\big)=
\frac{1}{b_g}\cdot
\sum_{[\gamma]\in\mathcal{O}_k(g)}
c(\gamma)\,,
\end{equation}
where, $b_g=b_{g,0}$ and
$\mathcal{O}_k(g)\subset\mathcal{O}(g)=\mathcal{O}(g,0)$ is
the subcollection of orbits of those multicurves $\gamma$,
for which $\gamma_{\mathit{reduced}}$ has exactly $k$
connected components.

M.~Mirzakhani found an explicit expression for the
coefficient $c(\gamma)$ and for the global normalization
constant $b_{g,n}$ in terms of the intersection numbers of
$\psi$-classes.

\subsection{Frequencies of square-tiled surfaces of fixed
combinatorial type}
\label{ss:Frequencies:of:square:tiled:surfaces}

The following Theorem bridges flat and hyperbolic
count.

\begin{NNTheorem}[\cite{DGZZ:volume}]
For any integral multicurve $\gamma\in\cML_{g,n}(\Z)$,
the volume contribution $\Vol(\gamma)$
to the Masur--Veech volume $\Vol\cQ_{g,n}$
coincides with the Mirzakhani's
asymptotic frequency $c(\gamma)$ of simple
closed geodesic multicurves of topological type $\gamma$
up to the explicit factor $const_{g,n}$
depending only on $g$ and $n$:
\begin{equation}
\label{eq:Vol:gamma:c:gamma}
\Vol(\gamma)
=const_{g,n}\cdot c(\gamma)\,,
\end{equation}
where
\begin{equation}
\label{eq:const:g:n}
const_{g,n}
=2\cdot(6g-6+2n)\cdot
(4g-4+n)!\cdot 2^{4g-3+n}\cdot
\end{equation}
\end{NNTheorem}

\begin{proof}[Proof of Theorem~\ref{th:same:distribution}]
Definitions~\eqref{eq:Proba:K:g:S}
and~\eqref{eq:Proba:K:g:gamma} and
Formulae~\eqref{eq:Vol:Q:g:n:as:sum} and~\eqref{eq:b:g:n:as:sum:of:c:gamma}
combined with relation~\eqref{eq:Vol:gamma:c:gamma} imply
that $\Proba\big(K_g(\gamma)=k\big)
=\Proba\big(K_g(S)=k\big)$.
\end{proof}

\begin{NNCorollary}[\cite{DGZZ:volume}]
For any admissible pair of non-negative integers $(g,n)$
different from $(1,1)$ and $(2,0)$,
the Masur--Veech volume $\Vol\cQ_{g,n}$ and the average
Thurston measure of a unit ball $b_{g,n}$ are related as
follows:

\begin{equation}
\label{eq:Vol:g:n:b:g:n}
\Vol\cQ_{g,n}
=2\cdot(6g-6+2n)\cdot
(4g-4+n)!\cdot 2^{4g-3+n}\cdot
b_{g,n}\,.
\end{equation}
\end{NNCorollary}

\begin{Remark}
In Theorem~1.4 in~\cite{Mirzakhani:earthquake}
M.~Mirzakhani established the relation
$$
\Vol\cQ_{g}
=
const_{g}\cdot
b_{g}\,,
$$
where $b_{g}$ is computed in Theorem~5.3
in~\cite{Mirzakhani:grouth:of:simple:geodesics}. However,
Mirzakhani does not give any formula for the value of the
normalization constant $const_g$ presented
in~\eqref{eq:Vol:g:n:b:g:n}. This constant was recently
computed by F.~Arana--Herrera~\cite{Arana:Herrera} and by
L.~Monin and I.~Telpukhovskiy~\cite{Monin:Telpukhovskiy}
simultaneously and independently of us by different
methods. The same value of $const_{g,n}$ is obtained by
V.~Erlandsson and J.~Souto in~\cite{Erlandsson:Souto}
through an approach different from all the ones mentioned
above.
\end{Remark}

\subsection{Uniform large genus asymptotics of correlators (after A.~Aggarwal)}
\label{ss:conjecture:on:correlators}
We denote by $\Pi(m,n)$ the set of nonnegative compositions
of an integer $m$ as sum of $n$ non-negative integers. For
any nonnegative composition
$\boldsymbol{d}\in\Pi(3g-3+n,n)$ define
$\epsilon(\boldsymbol{d})$ through the following equation:
\begin{equation}
\label{eq:ansatz}
\langle \tau_{d_1} \dots \tau_{d_n}\rangle_{g,n}
=
\frac{(6g-5+2n)!!}{(2d_1+1)!!\cdots(2d_n+1)!!}
\cdot\frac{1}{g!\cdot 24^g}
\cdot\big(1+\epsilon(\boldsymbol{d})\big)\,.
\end{equation}
By construction, the intersection numbers
are nonnegative rational numbers, so
$\varepsilon\big(\boldsymbol{d}\big)\ge -1$ for any
$\boldsymbol{d}\in\Pi(3g-3+n,n)$.
We conjectured
in~\cite{DGZZ:volume} that $\epsilon(\boldsymbol{d})$ tends
to zero uniformly for all nonnegative compositions $\boldsymbol{d}\in
\Pi(3g-3+n,n)$ as soon as $n\le 2\log g$ and $g\to+\infty$.
This conjecture was proved in much stronger form in the
recent paper of
A.~Aggarwal~\cite{Aggarwal:intersection:numbers}.

The following Theorem corresponds
to~\cite[Proposition~1.2]{Aggarwal:intersection:numbers}.

\begin{Theorem}[A.~Aggarwal]
\label{th:correlators:upper:bound}
Let $n \in \mathbb{Z}_{\ge 1}$ and $\boldsymbol{d}
\in \mathbb{Z}_{\ge 0}^n$ satisfy $|\boldsymbol{d}|
= 3g + n - 3$, for some $g \in \mathbb{Z}_{\ge 0}$. Then,
\begin{equation}
\label{eq:Aggarwal:Prop:1:2}
1+\epsilon(\boldsymbol{d})
\le \left( \frac{3}{2} \right)^{n-1}\,.
\end{equation}
\end{Theorem}

The next Theorem corresponds
to~\cite[Proposition~4.1]{Aggarwal:intersection:numbers}.
\begin{Theorem}[A.~Aggarwal]
\label{th:asymptotics:of:correlators:lower}
Let $g > 2^{15}$ and $n \ge 1$ be integers such that $g >
30 n$, and let $\boldsymbol{d}\in \Pi(3g-3+n,n)$. Then we have
\begin{equation}
\label{eq:asymptotics:of:correlators:lower}
\epsilon(\boldsymbol{d})
\ge -20\cdot\frac{(n + 4\log g)}{g}\,.
\end{equation}
\end{Theorem}

Finally,
the following Theorem corresponds
to~\cite[Proposition~4.2]{Aggarwal:intersection:numbers}.
\begin{Theorem}[A.~Aggarwal]
\label{th:asymptotics:of:correlators:upper}
Let $g > 2^{30}$ and $n \ge 1$ be integers such that $g >
800 n^2$, and let $\boldsymbol{d}\in \Pi(3g-3+n,n)$. Then
we have
\begin{equation}
\label{eq:asymptotics:of:correlators:upper}
1+\epsilon(\boldsymbol{d})
\le \exp\left(625\cdot\frac{(n + 2\log g)^2}{g}\right)\,.
\end{equation}
\end{Theorem}

\begin{Remark}
We proved in~\cite{DGZZ:volume} explicit sharp upper and
lower bounds for $2$-correlators.
\end{Remark}

\section{Random non-separating multicurves and non-uniform random permutations}
\label{s:sum:over:single:vertex:graphs}
Consider the stable graph $\Petal_k(g)$ having a single
vertex, decorated with genus $g-k$, and having $k$ loops,
see the left picture in
Figure~\ref{fig:two:non:separating}. This stable graph
corresponds to multicurves on a closed surface of genus
$g$, for which the components $\gamma_1,\dots,\gamma_k$ of
the underlying reduced multicurve
$\gamma_{\mathit{reduced}}=\gamma_1+\dots+\gamma_k$
represent $k$ linearly independent homology cycles. The
square-tiled surfaces associated to this stable graph have
single horizontal singular layer and $k$ maximal horizontal
cylinders.

\begin{figure}[htb]
\includegraphics{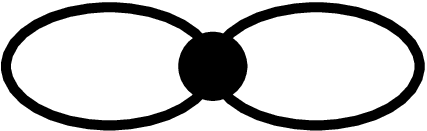}
\begin{picture}(0,0)(135,10) 
\put(44,-14){$g-k$}
\put(0,0){$\overbrace{\rule{70pt}{0pt}}^{k\text{ loops}}$}
\put(110,7){$\overbrace{\rule{55pt}{0pt}}^k$}
\put(100,-26){$\underbrace{\rule{175pt}{0pt}}_g$}
\end{picture}
\includegraphics{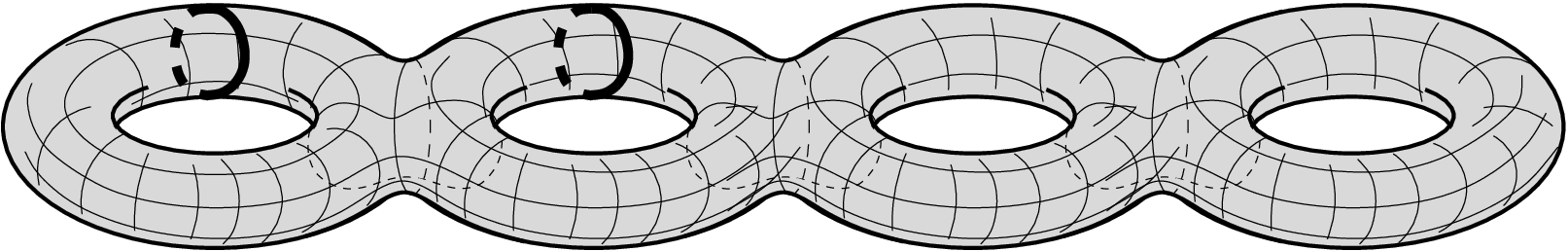}
\vspace*{40pt}

\caption{The stable graph $\Petal_k(g)$ (on the left) corresponds to
the reduced multicurve
$\gamma_{\mathit{reduced}}=\gamma_1+\dots+\gamma_k$
represented by $k$ linearly independent homology cycles on
a surface of genus $g$ (on the right).}
\label{fig:two:non:separating}
\end{figure}
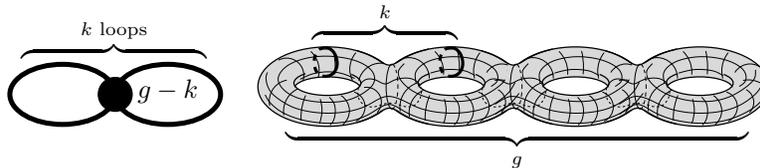

Recall from Section~\ref{ss:intro:Masur:Veech:volumes} that
by $\Vol(\Petal_k(g))$ we denote the
volume contribution from all square-tiled surfaces
corresponding to the stable graph $\Petal_k(g)$.
By
$\Vol\big(\Petal_k(g),(m_1, \ldots, m_k))$
we denote the volume contribution from those
square-tiled surfaces corresponding to the stable graph
$\Petal_k(g)$ for which one maximal horizontal
cylinder is filled with $m_1$ bands of squares, another
cylinder is filled with $m_2$ bands of
squares, and so on up to the $k$th maximal horizontal
cylinder, which is filled with
$m_k$ bands of squares. The corresponding multicurve has
the form $m_1\gamma_1+\dots+m_k\gamma_k$, where
$\gamma_1,\dots,\gamma_k$ are as described above.
By~\eqref{eq:Vol:Gamma} we have
\[
\Vol(\Petal_k(g)) =
\sum_{\substack{m_1, \ldots, m_k\\1\le m_i \le m\ \text{for }i=1,\dots,k}}
\Vol\big(\Petal_k(g),(m_1, \ldots, m_k)).
\]

In this section we prove the following result, which relies
on the uniform asymptotics of Witten correlators proved by
A.~Aggarwal (see
Theorems~\ref{th:correlators:upper:bound}--\ref{th:asymptotics:of:correlators:upper}
in the current paper or Propositions~1.2, 4.1, 4.2
respectively in the original
paper~\cite{Aggarwal:intersection:numbers} of A.~Aggarwal).

\begin{Theorem}
\label{thm:generating:series:vol:1}
Let $m \in \N \cup \{+\infty\}$. For any complex number
$t$,
in the disk $|t|<2$
we have as $g \to +\infty$
\begin{multline}
\label{eq:generating:series:vol:1}
\sum_{k=1}^{g}\quad
\sum_{\substack{m_1, \ldots, m_k\\1\le m_i \le m\ \text{for }i=1,\dots,k}}
\Vol(\Gamma_k(g), (m_1, \ldots, m_k)) \cdot t^k
=\\=
\frac{2\sqrt{2} \left(\frac{2m}{m+1}\right)^{t/2}}{\sqrt{\pi} \cdot \Gamma(\frac{t}{2})}
\cdot (3g-3)^{\frac{t-1}{2}} \cdot \left( \frac{8}{3} \right)^{4g-4}
\left(1 + O\left( \frac{(\log g)^2}{g}\right) \right)\,,
\end{multline}
where for every compact subset $U$
of the complex disk $|t|<2$
the error term is uniform
over $m \in \N \cup \{+\infty\}$ and $t\in U$.
In particular, for $m=+\infty$ and $t=1$ we obtain
\begin{equation}
\label{eq:sum:Vol:Gamma:k}
\sum_{k=1}^g \Vol(\Gamma_k(g))
=
\frac{ 4 }{ \pi} \left( \frac{8}{3} \right)^{4g-4}
\left(1 + O\left( \frac{(\log g)^2}{g}\right) \right)\,.
\end{equation}
\end{Theorem}

We prove Theorem~\ref{thm:generating:series:vol:1}
in Section~\ref{ss:from:q:to:p1}.
We note that asymptotics~\eqref{eq:sum:Vol:Gamma:k} was
first obtained by A.~Aggarwal
in~\cite[Proposition~8.3]{Aggarwal:intersection:numbers}.
Our refinement consists in the bound $O\left( \frac{(\log
g)^2}{g}\right)$ for the error term. Conjecturally, the
bound can be further improved to $O\left(
\frac{1}{g}\right)$; see
Remark~\ref{rm:expansion:of:error:term}.

\subsection{Volume contribution of stable graphs with a single vertex}
\label{ss:Volume:contribution:single:vertex}

In this section, we show how to express an approximate
value of the contribution $\Vol(\Petal_k(g))$ of
square-tiled surfaces corresponding to the stable graph
$\Petal_k(g)$ to the Masur--Veech volume $\Vol\cQ_g$ in
terms of the following
normalized weighted multi-variate harmonic sum.

\begin{Definition}
\label{def:hkzk}
Let $m \in \N \cup \{+\infty\}$ and
let $\alpha$ be a positive real number. For
integers $k, n$ such that $1 \leq k \leq n$, define
\begin{equation}
\label{eq:multiple:harmonic:sum:def}
\ContributionH_{n,m,\alpha}(k) = \frac{\alpha^k}{k!} \sum_{j_1+\dots+j_k=n}
 \frac{\zeta_m(2 j_1) \cdot \zeta_m(2 j_2) \cdots \zeta_m(2j_k)}{j_1 \cdot j_2 \cdots j_k}\,,
\end{equation}
where the sum is taken over all $k$-tuples
$(j_1, j_2, \ldots , j_k) \in \N^k$
of positive integers summing up to $n$ and
$$
\zeta_m(s) = 1 + \frac{1}{2^s} + \frac{1}{3^s} + \ldots + \frac{1}{m^s}
$$
is the partial zeta function.
\end{Definition}

\begin{Remark}
The particular cases of the above numbers, namely,
\begin{align*}
\MultiHarmonicH_k(n)=\sum_{j_1+\dots+j_k=n}&
\frac{1}{j_1\cdot j_2\cdots j_k}
= k! \cdot \ContributionH_{n,1,1}(k)
\\\
\MultiHarmonicZ_k(n)=\sum_{j_1+\dots+j_k=m}&
\frac{\zeta(2j_1)\cdots\zeta(2j_k)}
{j_1\cdot j_2\cdots j_k}
= k! \cdot \ContributionH_{n,\infty,1}(k)
\end{align*}
appeared in the preprint~\cite{DGZZ:volume}; the asymptotic
expansions for these quantities were obtained by A.~Aggarwal in
Sections~6 and~7 of~\cite{Aggarwal:intersection:numbers}.
The framework which we develop here allows to treat all
normalized weighted multi-variate harmonic
sums $\ContributionH_{n,m,\alpha}(k)$ in a unified way.
\end{Remark}

\begin{Theorem}
\label{th:bounds:for:Vol:Gamma:k:g}
There exists a constant $C_1$ such that for
sufficiently large $g\in\N$ the following property holds.
For any couple $m,k$, such that $m \in \N \cup
\{+\infty\}$, $k\in\N$, $800 k^2\le g$, we have
\begin{multline}
\label{eq:bounds:in:terms:of:H:k:gminus3}
\sum_{\substack{m_1, \ldots, m_k\\1\le m_i \le m\ \text{for }i=1,\dots,k}}
\Vol\big(\Petal_k(g), (m_1, \ldots, m_k)\big)
=\\=
\frac{2 \sqrt{2}}{\sqrt{\pi}} \cdot \sqrt{3g-3}
\cdot \left(\frac{8}{3}\right)^{4g-4}
\cdot\ContributionH_{3g-3,m,\frac{1}{2}}(k)
\cdot\big(1 + \varepsilon_1(g,k) \big)
\,,
\end{multline}
where $|\varepsilon_1(g,k)|\le
C_1\cdot\cfrac{(k+2\log g)^2}{g}$.

There exists a constant $C_2$ such that for
all triples $(g,k,m)$, where $g\in\N$, $g\ge 2$; $k\in\N$, $k\le g$;
$m \in \N \cup \{+\infty\}$, we have
\begin{multline}
\label{eq:Vol:Gamma:k:upper:bound}
\sum_{\substack{m_1, \ldots, m_k\\1\le m_i \le m\ \text{for }i=1,\dots,k}}
\Vol\big(\Petal_k(g), (m_1, \ldots, m_k)\big)\le
\\
\le C_2\cdot
\sqrt{g}\cdot
\left(\frac{8}{3}\right)^{4g-4}\cdot
\ContributionH_{3g-3,m,\frac{1}{2}}(k)\cdot
\left( \frac{9}{4} \right)^k
\,,
\end{multline}
where $\ContributionH_{3g-3,m,\frac{1}{2}}(k)$
is the normalized weighted multi-variate harmonic sum
defined in~\eqref{eq:multiple:harmonic:sum:def}.
\end{Theorem}

In order to prove Theorem~\ref{th:bounds:for:Vol:Gamma:k:g}
we first state and prove Lemma~\ref{lm:cgk:asymptotics} below.

Let $\textbf{D}=(D_1,\dots,D_k)\in\Pi(3g-3+2k,k)$.
Define $c_{g,k}(\boldsymbol{D})$ as
\begin{multline}
\label{eq:sum:of:normalized:correlators}
c_{g,k}(\boldsymbol{D}):=
\frac{g!\cdot(3g-3+2k)!}{(6g+4k-5)!}
\cdot\frac{3^g}{2^{3g-6+5k}}\cdot
\\
\cdot
\sum_{d_{1,1}+d_{1,2}=D_1}\dots\sum_{d_{k,1}+d_{k,2}=D_k}
\int_{\overline{\cM}_{g,2k}}
\psi_1^{d_{1,1}}\psi_2^{d_{1,2}}
\dots \psi_{2k-1}^{d_{k,1}}\psi_{2k}^{d_{k,2}}
\cdot\prod_{j=1}^k \cfrac{(2D_j+2)!}{d_{j,1}!\cdot d_{j,2}!}\,.
\end{multline}

The following result is a corollary of the uniform
asymptotics of Witten correlators proved by A.~Aggarwal
(see
Theorems~\ref{th:correlators:upper:bound}--\ref{th:asymptotics:of:correlators:upper}
in the current paper or Propositions~1.2, 4.1, 4.2
respectively in the original
paper~\cite{Aggarwal:intersection:numbers} of \mbox{A.~Aggarwal}).
\begin{Lemma}
\label{lm:cgk:asymptotics}
There exists a constant $C_3$ such that
for sufficiently large $g\in\N$
and for $k\in\N$ satisfying
$800 k^2\le g$ we have
\begin{equation}
\label{eq:cgk:at:most:logg}
|c_{g,k}(\boldsymbol{D}) - 1| \le
C_3 \cdot \frac{(k+2\log g)^2}{g}\,.
\end{equation}
For any positive integers $g,k\in\N$
satisfying $1\le k\le g$ and $g\ge 2$, we have
\begin{equation}
\label{eq:9:4:power:k}
c_{g,k}(\boldsymbol{D}) \le \left( \frac{9}{4} \right)^{k}.
\end{equation}
\end{Lemma}

\begin{proof}
Passing to double factorials and
applying definition~\eqref{eq:ansatz} of
$\epsilon(\boldsymbol{d})$ we get
\begin{multline*}
c_{g,k}(\boldsymbol{D})
=\frac{g!}{2^{3g-3+2k}\cdot(6g+4k-5)!!}
\cdot\frac{3^g}{2^{3g-6+5k}}\cdot
\\
\cdot
\sum_{d_{1,1}+d_{1,2}=D_1}\dots\sum_{d_{k,1}+d_{k,2}=D_k}
\langle\tau_{d_{1,1}}\tau_{d_{1,2}}
\dots \tau_{d_{k,1}}\tau_{d_{k,2}}\rangle_{g,2k}
\\
\prod_{j=1}^k\left(
\frac{(2d_{j,1}+1)!}{d_{j,1}!}\cdot\frac{(2d_{j,2}+1)!}{d_{j,2}!}
\cdot\binom{2D_j+2}{2d_{j,1}+1}\right)
=\\=
\frac{1}{2^{6g-6+5k}}
\sum_{d_{1,1}+d_{1,2}=D_1}\dots\sum_{d_{k,1}+d_{k,2}=D_k}
\left(\big(1+\epsilon(\boldsymbol{d})\big)
\cdot\prod_{j=1}^k
\binom{2D_j+2}{2d_{j,1}+1}\right)\,.
\end{multline*}
Applying the combinatorial identity
$$
\sum_{m=0}^{n-1}\binom{2n}{2m+1}=2^{2n-1}
$$
we get
\begin{multline*}
\sum_{d_{1,1}+d_{1,2}=D_1}\dots\sum_{d_{k,1}+d_{k,2}=D_k}
\prod_{j=1}^k
\binom{2D_j+2}{2d_{j,1}+1}
=\\=
\left(\prod_{j=1}^k \sum_{d_{j,1}=0}^{D_j}\binom{2D_j+2}{2d_{j,1}+1}\right)
=\left(\prod_{j=1}^k 2^{2D_j+1}\right)
=
2^{6g-6+5k}\,.
\end{multline*}
The claim that bound~\eqref{eq:cgk:at:most:logg} is valid
for sufficiently large $g$ now follows from combination of
bounds~\eqref{eq:asymptotics:of:correlators:lower}
and~\eqref{eq:asymptotics:of:correlators:upper} from
Theorems~\ref{th:asymptotics:of:correlators:lower}
and~\ref{th:asymptotics:of:correlators:upper} of
A.~Aggarwal (see Propositions~4.1, 4.2 respectively in the
original paper~\cite{Aggarwal:intersection:numbers}).

For $k\ge 2$ the universal bound~\eqref{eq:9:4:power:k}
follows from the universal bound~\eqref{eq:Aggarwal:Prop:1:2}
from Theorem~\ref{th:correlators:upper:bound}
of A.~Aggarwal (see Proposition~1.2 in the
original paper~\cite{Aggarwal:intersection:numbers}),
using the fact that for $k \ge 2$ we have $1 + (3/2)^{2k-1}
\le (3/2)^{2k}$.

We prove in Proposition~4.1 in~\cite{DGZZ:volume} that
$\epsilon(\boldsymbol{d})\le 0$ for any
$\boldsymbol{d}\in\Pi(3g-1,2)$. This implies
bound~\eqref{eq:9:4:power:k} for $k=1$, which completes the
proof of the Lemma.
\end{proof}

\begin{proof}[Proof of Theorem~\ref{th:bounds:for:Vol:Gamma:k:g}]
Let us denote
\begin{equation}
\label{eq:V:m:k}
V_{m,k}(g) =\!\!\!\!
\sum_{\substack{m_1, \ldots, m_k\\1\le m_i \le m\\ \text{for }i=1,\dots,k}}
\!\!\!\Vol(\Petal_k(g), (m_1, \ldots, m_k))
\quad\text{and}\quad
V_{m}(g)=\sum_{k=1}^g V_{m,k}(g)\,.
\end{equation}
The automorphism group $\operatorname{Aut}(\Petal_k(g))$
consists of all possible permutations of loops composed
with all possible flips of individual loops, so
$$
|\operatorname{Aut}(\Gamma)|=2^k\cdot k!\,.
$$
The graph $\Petal_k(g)$ has a single vertex, so
$|V(\Petal_k(g)|=1$.
Thus, applying~\eqref{eq:contribution:of:gamma:to:volume}
to $\Petal_k(g)$ we get
\begin{multline}
\label{eq:Vol:Gamma:k:g:init}
V_{m,k}(g)
=\frac{2^{6g-5} \cdot (4g-4)!}{(6g-7)!}
\cdot\\
1\cdot\frac{1}{2^k\cdot k!}
\cdot
\sum_{\substack{\boldsymbol{H} = (m_1, \ldots, m_k)\\m_1, \ldots, m_k \leq m}}
\cY\big(\boldsymbol{H},
b_1 b_2\dots b_k \cdot
N_{g-k,2k}(b_1,b_1,b_2,b_2,\dots,b_k,b_k)
\big)
=\\=
\frac{(4g-4)!}{(6g-7)!}
\cdot \frac{2^{6g-5}}{2^k\cdot k!}
\cdot\frac{1}{2^{5(g-k)-6+4k}}
\cdot\\
\sum_{\boldsymbol{d}\in\Pi(3g-3-k,2k)}
\frac{\langle\tau_{\mathbf{d}}\rangle_{g-k,2k}}{\mathbf{d}!}\,\cdot\,
\prod_{i=1}^k
\Big((2d_{2i-1}+2d_{2i}+1)!\,\cdot\,\zeta_m(2d_{2i-1}+2d_{2i}+2)\Big)\,.
\end{multline}

Rewrite the latter sum using notations
$\textbf{D}=(D_1,\dots,D_k)\in\Pi(3g-3-k,k)$
and $c_{g-k,k}(\boldsymbol{D})$ defined
by~\eqref{eq:sum:of:normalized:correlators}.
Adjusting
expression~\eqref{eq:sum:of:normalized:correlators}
given for genus $g$ to genus $g-k$
we get
\begin{multline*}
\sum_{\boldsymbol{d}\in\Pi(3g-3-k,2k)}
\frac{\langle\tau_{\mathbf{d}}\rangle_{g-k,2k}}{\mathbf{d}!}\,\cdot\,
\prod_{i=1}^k
\Big((2d_{2i-1}+2d_{2i}+1)!\,\cdot\,\zeta_m(2d_{2i-1}+2d_{2i}+2)\Big)
=\\=
\sum_{\boldsymbol{D}\in\Pi(3g-3-k,k)}
c_{g-k,k}(\boldsymbol{D})\cdot
\frac{(6(g-k)+4k-5)!}{(g-k)!\cdot(3(g-k)-3+2k)!}
\cdot\frac{2^{3(g-k)-6+5k}}{3^{(g-k)}}
\cdot\\
\cdot\prod_{j=1}^k \frac{\zeta_m(2D_j+2)}{2D_j+2}\,,
\end{multline*}
which allows to rewrite~\eqref{eq:Vol:Gamma:k:g:init} as
\begin{multline}
\label{eq:V:m:k:g}
V_{m,k}(g)
=\\=
\left(\frac{(4g-4)!}{(6g-7)!}
\cdot 2^{g+1}
\cdot\frac{1}{k!}\right)
\cdot
\left(\frac{(6g-2k-5)!}{(g-k)!\cdot(3g-3-k)!}
\cdot\frac{2^{3g-6+2k}}{3^{g-k}}\right)\cdot
\\
\sum_{\boldsymbol{D}\in\Pi(3g-3-k,2k)}
\frac{c_{g-k,k}(\boldsymbol{D})}{2^k}
\cdot\prod_{j=1}^k \frac{\zeta_m(2D_j+2)}{D_j+1}\,.
\end{multline}

Let us define
\[
c^{min}_{g-k,k} := \min_{\boldsymbol{D}} c_{g-k,k}(\boldsymbol{D})
\quad \text{and} \quad
c^{max}_{g-k,k} := \max_{\boldsymbol{D}} c_{g-k,k}(\boldsymbol{D})\,.
\]

Rearranging factors with factorials, collecting powers of
$2$ and $3$, and passing to notation
$\ContributionH_{m,3g-3,\frac{1}{2}}(k)$ for the
multivariate harmonic
sum~\eqref{eq:multiple:harmonic:sum:def} we get the
following bounds:
\begin{multline}
\label{eq:Vol:Gamma:k:g:intermed}
c^{min}_{g-k,k}\le V_{m,k}(g)\cdot
\\
\left(
\frac{(6g-2k-5)!}{(6g-7)!}
\cdot\frac{(4g-4)!}{(g-k)!\cdot(3g-3-k)!}
\cdot \frac{2^{4g-5+2k}}{3^{g-k}}
\cdot \ContributionH_{3g-3,m,\frac{1}{2}}(k)
\right)^{-1}
\\
\le c^{max}_{g-k,k}\,.
\end{multline}

We start by proving the first assertion of the theorem
represented by
relation~\eqref{eq:bounds:in:terms:of:H:k:gminus3}. We
rewrite the product of factorials
in~\eqref{eq:Vol:Gamma:k:g:intermed} as
\begin{multline}
\label{eq:factorials}
\frac{(6g-2k-5)! \cdot (4g-4)!}{(6g-7)! \cdot (g-k)! \cdot (3g-3-k)!}
=
(6g-6) \cdot
\binom{4g-4}{g-1} \cdot
\frac{\frac{(3g-3)!}{(3g-3-k)!} \cdot \frac{(g-1)!}{(g-k)!}}%
{\frac{(6g-6)!}{(6g-2k-5)!}} \\
= (6g-6) \cdot \binom{4g-4}{g-1} \cdot
\frac{(3g-3)^k \cdot (g-1)^{k-1}}{(6g-6)^{2k-1}}
\big(1 +  \varepsilon_3(g,k)\big)=  \\
= \sqrt{3g-3}  \cdot \sqrt{\frac{2}{\pi}} \cdot \frac{2^{8g-6-2k}}{3^{3g-4+k}}
\big(1 + \varepsilon_4(g,k) \big)\,.
\end{multline}
Note that there exist constants $C'_3$ and
$a_0$ such that for any integer $a$
satisfying
$a\ge a_0$ and for any $b\in\N$, satisfying $b\le \sqrt a$,
we have
$$
a^b \left( 1 - C'_3 \cdot \frac{b^2}{a} \right)
\le
\frac{a!}{(a-b)!}
\le
a^b \left( 1 + C'_3\cdot \frac{b^2}{a} \right)\,.
$$
This implies that there exists $g_0$ such that
for any $g\in\N$ satisfying $g\ge g_0$ and for any $k\in\N$
satisfying $800 k^2\le g$ we have the bound
$$
|\epsilon_3(g,k)|\le 1+C_3\cdot\frac{k^2}{g}
$$
for the error term in the second line of~\eqref{eq:factorials}.
Let
$$
\binom{4g-4}{g-1}
=\sqrt{\frac{2}{\pi (3g-3)}} \cdot \frac{2^{8g-8}}{3^{3g-3}}
\cdot\big(1+\varepsilon_5(g)\big)\,.
$$
There exist constants $C_5$ and $g_1$ such that for any
$g\in\N$ satisfying $g\ge g_1$ we have
$$
|\varepsilon_5(g)|\le C_5\cdot
\frac{1}{g}\,.
$$
The latter two bounds imply that
there exist a constant $C_4$
and a constant $g_2$ such that
for any $g\in\N$ satisfying $g\ge g_2$
we have the bound
$$
|\varepsilon_4(g)|\le C_4\cdot
\frac{k^2}{g}\,,
$$
for the error term on the right-hand side of the third line
of~\eqref{eq:factorials}.
Using the latter bound and
collecting powers of $2$, of $3$ and of $g$, we can
rewrite~\eqref{eq:Vol:Gamma:k:g:intermed} in the following
way:
\begin{multline}
\label{eq:to:serve:for:c:min:max:equals:1}
c^{min}_{g-k,k}
\left(1 - C_4\cdot \frac{k^2}{g}\right)
\leq \\
\le \cfrac{V_{m,k}(g)}
{\frac{2 \sqrt{2}}{\sqrt{\pi}}
\cdot\left(\frac{8}{3}\right)^{4g-4}
\cdot \sqrt{3g-3}
\cdot \ContributionH_{3g-3,m,\frac{1}{2}}(k)}
\le
c^{max}_{g-k,k} \left(1 + C_4\cdot \frac{k^2}{g}\right)\,.
\end{multline}
Now, using the bound~\eqref{eq:cgk:at:most:logg} from the
first part of Lemma~\ref{lm:cgk:asymptotics} we
get~\eqref{eq:bounds:in:terms:of:H:k:gminus3}.

The proof of the upper bound~\eqref{eq:Vol:Gamma:k:upper:bound}
is similar. For the product of factorials we use the bound
\begin{multline}
\label{eq:bound:for:factorials}
\frac{\frac{(3g-3)!}{(3g-3-k)!} \cdot \frac{(g-1)!}{(g-k)!}}%
{\frac{(6g-6)!}{(6g-2k-5)!}}
=
\prod_{i=0}^{k-1} \frac{3g-3-i}{6g-6-2i}
\prod_{i=1}^{k-1} \frac{g-i}{6g-5-2i}
=\\
=\frac{1}{12^k} \prod_{i=1}^{k-1} \frac{6g-6i}{6g-5-2i}
=\frac{1}{12^k} \left(1+\frac{1}{6g-7}\right) \prod_{i=2}^{k-1} \frac{6g-6i}{6g-5-2i}
\le \frac{6}{5}\cdot \frac{1}{12^k}\,.
\end{multline}
valid for any couple $(g,k)$ of positive integers
satisfying $g\ge 2$ and $k\le g$. Here we used the
inequality $6g-6i < 6g-5-2i$ valid for any integer $g,i$
such that $g\ge 2$ and $i \geq 2$. We also used the
inequality $1/(6g-7)\le 1/5$ valid for
any integer $g\ge 2$. The upper bound for $c^{max}_{g-k,k}$
was established in Equation~\eqref{eq:9:4:power:k} in the
second part of Lemma~\ref{lm:cgk:asymptotics}. Plugging
this bound for $c^{max}_{g-k,k}$
in~\eqref{eq:Vol:Gamma:k:g:intermed}
and the bound for the product of factorials
in~\eqref{eq:factorials} and proceeding as before
we obtain the result.
\end{proof}

Define
\begin{multline}
\label{eq:V:k:g:tilde}
\tilde V_{m,k}(g)
=\left(\frac{(4g-4)!}{(6g-7)!}
\cdot 2^{g+1}
\cdot\frac{1}{k!}\right)
\cdot
\left(\frac{(6g-2k-5)!}{(g-k)!\cdot(3g-3-k)!}
\cdot\frac{2^{3g-6+2k}}{3^{g-k}}\right)\cdot
\\
\sum_{\boldsymbol{D}\in\Pi(3g-3-k,2k)}
\frac{1}{2^k}
\cdot\prod_{j=1}^k \frac{\zeta(2D_j+2)}{D_j+1}\,.
\end{multline}
This expression is obtained by replacing
$c_{g-k,k}(\boldsymbol{D})$ with $1$ in~\eqref{eq:V:m:k:g}.
We have seen that this is equivalent to
replacing the Kontsevich--Witten correlators in
the right-hand side of
formula~\eqref{eq:Vol:Gamma:k:g:init} for $V_{m,k}(g)$
by the asymptotic
expression~\eqref{eq:ansatz} from
Section~\ref{ss:conjecture:on:correlators}.
We are now ready to give the formal definition of the
the approximating distribution $q_g(k)$ informally
described in Section~\ref{s:intro}.

Define the probability distribution $\ProbaGamma_g(k)$ as
\begin{equation}
\label{eq:def:q_g}
\ProbaGamma_g(k) := \frac{\tilde V_{\infty,k}(g)}{\tilde V_{\infty}(g)}\,,
\quad \text{where} \quad
\tilde V_{\infty}(g) := \sum_{k=1}^{g} \tilde V_{\infty,k}(g)\,.
\end{equation}

It follows from the proof of of Theorem~\ref{th:bounds:for:Vol:Gamma:k:g}
that for sufficiently large $g\in\N$
and for $k\in\N$ satisfying
$k^2\le g$ we have the bounds for $\tilde V_{m,k}(g)$
analogous to~\eqref{eq:to:serve:for:c:min:max:equals:1},
where $c^{max}_{g-k,k}$ and $c^{max}_{g-k,k}$ are replaced with $1$.
This implies that $\tilde V_{m,k}(g)$ satisfies the
lower bound
\begin{equation}
\label{eq:V:tilde:lower}
\tilde V_{m,k}(g)\ge
\frac{2 \sqrt{2}}{\sqrt{\pi}}
\cdot\left(\frac{8}{3}\right)^{4g-4}
\cdot \sqrt{3g-3}
\cdot \ContributionH_{3g-3,m,\frac{1}{2}}(k)
\cdot \left(1 + O\left(\frac{k^2}{g}\right)\right)\,,
\end{equation}
where the constant in the error term is uniform for
$k\in\N$ satisfying $k^2\le g$.

The upper bound for the expression in factorials
on the left-hand side of~\eqref{eq:bound:for:factorials}
can be expressed for large $g$ as
$\frac{1}{12^k}\cdot\left(1+O\left(g^{-1}\right)\right)$.
Thus, analog of~\eqref{eq:Vol:Gamma:k:g:intermed}
for $\tilde V_{m,k}(g)$,
where $c^{max}_{g-k,k}$ and $c^{max}_{g-k,k}$ are replaced with $1$
implies that that for sufficiently large $g\in\N$
and for any $k\in\N$ we have the upper bound
\begin{equation}
\label{eq:V:tilde:upper}
\tilde V_{m,k}(g)\le
\frac{2 \sqrt{2}}{\sqrt{\pi}}
\cdot\left(\frac{8}{3}\right)^{4g-4}
\cdot \sqrt{3g-3}
\cdot \ContributionH_{3g-3,m,\frac{1}{2}}(k)
\cdot \left(1 + O\left(\frac{1}{g}\right)\right)\,.
\end{equation}

\subsection{Multi-variate harmonic sums and random non-uniform permutations}
\label{ss:non:uniform:permutations}
In this section we analyze the
normalized weighted multi-variate harmonic sum
from Definition~\ref{def:hkzk} and
Theorem~\ref{th:bounds:for:Vol:Gamma:k:g}. We show how
these kind of sums naturally appear in the study
of random permutations in the
symmetric group.

Let us recall the settings from Section~\ref{s:intro}.
Let $\theta = \{\theta_k\}_{k \geq 1}$ be non-negative real
numbers. From now on we assume for simplicity that
$\theta_1 > 0$. Recall that given a permutation $\sigma \in
S_n$ with cycle type $(1^{\mu_1} 2^{\mu_2} \ldots
n^{\mu_n})$ we define its \textit{weight} as
\[
w_\theta(\sigma) := \theta_1^{\mu_1} \theta_2^{\mu_2} \cdots \theta_n^{\mu_n}.
\]
To every sequence $\theta = \{\theta_k\}_{k \geq 1}$ we associate a probability measure on the
symmetric group $S_n$ as in~\eqref{eq:proba:sym:group} by setting
$$
\Proba_{\theta,n}(\sigma) := \frac{w_{\theta}(\sigma)}{n!\cdot W_{\theta,n}}
\qquad \text{where} \qquad
W_{\theta,n} := \frac{1}{n!} \sum_{\sigma \in S_n} w_\theta(\sigma).
$$
Constant weights $\theta_i = 1$ correspond to the uniform measure on $S_n$. More
generally, the probability measures on $S_n$ obtained from constant weights
$\theta_i = \alpha$ are called \emph{Ewens measure}. The following lemma
identifies our
normalized weighted multi-variate harmonic sums from Definition~\ref{def:hkzk}
as total contribution of permutations
having exactly $k$ cycles to the sum $W_{\theta,n}$.
\begin{Lemma}
\label{lem:cycle:distribution:vs:harmonic:sum}
Let $\theta = \{\theta_k\}_{k \geq 1}$ be non-negative real numbers and consider the
associated probability measure $\Proba_{\theta,n}$ on the symmetric group
$S_n$ for some $n$. Then
\begin{equation}
\label{eq:def:weight:of:permutation}
\frac{1}{n!} \cdot
\sum_{\substack{\sigma \in S_n\\ \NumCycles_n(\sigma) = k}}
w_\theta(\sigma)
=
\frac{1}{k!} \cdot \sum_{i_1 + \cdots + i_k = n}
\frac{\theta_{i_1} \theta_{i_2} \cdots \theta_{i_k}}{i_1 \cdots i_k}\,,
\end{equation}
where $\NumCycles_n(\sigma)$ is the number of cycles in the cycle decomposition of $\sigma$
and the sum in the right hand-side is taken over positive integers $i_1, \ldots, i_k$.
In other words, we have the identity in the ring $\Q[[t,z]]$ of formal power series in $t$ and $z$
\begin{equation}
\label{eq:generating:series:multi:harmonic}
\sum_{n \geq 1} \sum_{\sigma \in S_n} w_\theta(\sigma) t^{\NumCycles_n(\sigma)} \frac{z^n}{n!}
=
\exp \left(t \sum_{k \geq 1} \theta_k \frac{z^k}{k} \right).
\end{equation}
\end{Lemma}

The first several terms of the expansion of~\eqref{eq:generating:series:multi:harmonic}
in $z$ have the following form:
\begin{align*}
\exp \left(t \sum_{k \geq 1} \theta_k \frac{z^k}{k} \right)
=
&1 + \\
&\left(\theta_1 t\right) z + \\
&\left(\theta_2 t + \theta_1^2 t^2\right) \frac{z^2}{2!} + \\
&\left(2 \theta_3 t + 3 \theta_1 \theta_2 t^2 + \theta_1^3 t^3\right) \frac{z^3}{3!} + \\
&\ldots
\end{align*}

\begin{proof}[Proof of
Lemma~\ref{lem:cycle:distribution:vs:harmonic:sum}] From
each permutation $\sigma$ in $S_n$ and a composition $(i_1,
\ldots, i_k)$ of $n$ we build the following permutation
$\widetilde{\sigma}$ with $k$ cycles (in cycle notation)
\[
\big(\sigma(1), \sigma(2), \ldots, \sigma(i_1)\big)
\big(\sigma(i_1+1), \ldots, \sigma(i_1+i_2)\big) \cdots
\big(\sigma(i_1 + \cdots + i_{k-1}+1), \ldots, \sigma(n)\big).
\]
Here the cycles of $\widetilde{\sigma}$ are ordered from $1$ to $k$ so that
the first cycle has length $i_1$, the second has length $i_2$, etc.
Since each cycle is defined up to cyclic ordering, for each fixed $(i_1, \ldots, i_k)$
we obtain the same permutation (with ordered cycles) $i_1 \cdot i_2 \cdots i_k$ times.
Hence the number
\[
n! \frac{\theta_{i_1} \theta_{i_2} \cdots \theta_{i_k}}{i_1 i_2 \cdots i_k}.
\]
is the weighted count of permutations with $k$ labelled cycles of lengths $i_1$,
\ldots, $i_k$. Now summing over all possible compositions $(i_1, \ldots, i_k)$
of $n$ and dividing by $k!$ gives the weighted sum of permutations having
exactly $k$ cycles.
\end{proof}

We see that the normalized weighted multi-variate harmonic
sums $\ContributionH_{n,m,\alpha}(k)$ defined
in~\eqref{eq:multiple:harmonic:sum:def} represent the total
weight of permutations having exactly $k$ disjoint cycles
in their cycle decomposition, where the weights
$w_\theta(\sigma)$ correspond to the sequence $\theta_k =
\alpha \zeta_m(2k)$, $k\in\N$. Thus,
Lemma~\ref{lem:cycle:distribution:vs:harmonic:sum}
implies the following relation for the
generalization of the quantities
$q_{n,\infty,\alpha}$ defined in~\eqref{eq:q:3g:minus:3:infty:1:2:as:proba}
for arbitrary $m\in\N\cup\{+\infty\}$:
\begin{equation}
\label{eq:q:n:m:alpha}
q_{n,m,\alpha}(k)
=\Proba_{n, m, \alpha}\big(\NumCycles_n(\sigma) = k\big)
=\frac{\ContributionH_{n,m,\alpha}(k)}{W_{n,m,\alpha}}\,,
\end{equation}
where
\begin{equation}
\label{eq:W:n:m:alpha}
W_{n,m,\alpha}
=\sum_{k=1}^n \ContributionH_{n,m,\alpha}(k)
=\sum_{k=1}^n
\frac{1}{n!} \cdot
\sum_{\substack{\sigma \in S_n\\ \NumCycles_n(\sigma) = k}}
w_\theta(\sigma)
\,.
\end{equation}

Theorem~\ref{th:bounds:for:Vol:Gamma:k:g} relates the
contributions $\Vol\big(\Petal_k(g)\big)$ of stable graphs
$\Petal_k(g)$ to the Masur--Veech volume $\Vol\cQ_g$ to the
total weight of permutations having exactly $k$ disjoint
cycles in their cycle decomposition, with the weights
$w_\theta(\sigma)$ corresponding to the sequence $\theta_k =
\tfrac{1}{2}\zeta(2k)$, $k\in\N$, that is to the normalized
weighted multi-variate harmonic sums with $m=+\infty$ and
$\alpha=\tfrac{1}{2}$.

The unsigned Stirling numbers of the first kind $s(n,k)$
corresponding to the uniform distribution on $S_n$ satisfy
$s(n,k) = n! \cdot \ContributionH_{n,1,1}(k)$.

\begin{Theorem}
\label{thm:multi:harmonic:sum:total:weight}
Let $t$ be a complex number and $m \in \N \cup
\{+\infty\}$. Then
\begin{equation}
\label{eq:multi:harmonic:sum:total:weight}
\sum_{k=1}^n \ContributionH_{n,m,\alpha}(k) t^k =
\frac{\left( \frac{2m}{m+1} \right)^{\alpha t} n^{\alpha t - 1}}{\Gamma(\alpha t)}
\left(1 + O\left(\frac{1}{n}\right)\right)\,,
\end{equation}
where the error term is uniform in $t$ over compact subsets of complex numbers.
\end{Theorem}
Here we use the convention
$$
\left.\frac{2m}{m+1}\right\vert_{m=+\infty}=2\,.
$$

A version of
Theorem~\ref{thm:multi:harmonic:sum:total:weight} stated
for the values $m=1$ and $m=+\infty$; $\alpha=\frac{1}{2}$;
$t=1$ of the parameters, which are particularly important
in the context of the current paper, was stated as a
conjecture in the preprint~\cite{DGZZ:volume} and was first
proved by A.~Aggarwal
in~\cite[Proposition~7.2]{Aggarwal:intersection:numbers}.
We suggest here a proof of
Theorem~\ref{thm:multi:harmonic:sum:total:weight} based on
technique of H.~Hwang~\cite{Hwang:PhD} applied to the generating function in
the right-hand side of
Equation~\eqref{eq:generating:series:multi:harmonic}. We
discovered this approach for ourselves after the
paper~\cite{Aggarwal:intersection:numbers} was available.

We will use the following elementary facts in the
proof Theorem~\ref{thm:multi:harmonic:sum:total:weight}.

\begin{Lemma}
\label{lem:g:m}
Let $m \in \N \cup \{+\infty\}$. The series
\[
g_m(z) = \sum_{k \geq 1} \zeta_m(2k) \frac{z^k}{k}
\]
converges in the unit disk $|z|<1$. Considered as as a
holomorphic function, it extends to  $\C \setminus [1,
+\infty)$. Moreover, as $z \to 1$ inside $\C \setminus [1,
+\infty)$ we have
\begin{equation}
\label{eq:g:m:z}
g_m(z) = - \log(1 - z) + \log\left( \frac{2m}{m+1} \right) + O(1 - z).
\end{equation}
\end{Lemma}
\begin{proof}
Expanding the definition of the partial zeta function
$\zeta_m$ and changing the order of summation
we find the alternative formula
\[
g_m(z) = - \sum_{n = 1}^m \log\left(1 - \frac{z}{n^2}\right)\,,
\]
which proves the first assertion of the Lemma.

Now, we have
\[
g_m(z) = - \log(1 - z) -
\sum_{n=2}^m \log\left(1 - \frac{1}{n^2}\right)
+O(z-1).
\]
For finite $m$, we can rewrite the constant term as
$$
 -\sum_{n = 2}^m \log (1 - \frac{1}{n^2})
= \sum_{n = 2}^m (2 \log(n) - \log(n-1) - \log(n+1))
= \log \left( \frac{2 m}{m + 1} \right).
$$
The case $m=+\infty$ is obtained by passing to the limit.
\end{proof}

\begin{proof}[Proof of Theorem~\ref{thm:multi:harmonic:sum:total:weight}]
Theorem~\ref{thm:multi:harmonic:sum:total:weight} can be
derived as a corollary of Theorem~12 of~\cite{Hwang:PhD}
(see also Lemma~2.13 in \cite{NikeghbaliZeindler}). To
make the proof tractable we provide here a complete
argument based on the asymptotic analysis performed in the
classical book by P.~Flajolet and
R.~Sedgewick~\cite{Flajolet:Sedgewick}.

By Lemma~\ref{lem:cycle:distribution:vs:harmonic:sum} we have
\begin{equation}
\label{eq:g:alpha:m}
\sum_{k=1}^n \ContributionH_{n,m,\alpha}(k) t^k = [z^n] \exp(t \alpha g_{m}(z))\,,
\end{equation}
where $g_m(z)$ is the function defined in Lemma~\ref{lem:g:m}.
Plugging the asymptotic expansion~\eqref{eq:g:m:z} into~\eqref{eq:g:alpha:m}
we obtain the following expansion
as $z \to 1$ inside
$\C \setminus [1, +\infty)$:
\begin{equation}
\label{eq:proof:of:Th;3:7}
\exp(\alpha t g_m(z))
= \left(\frac{1}{1-z}\right)^{\alpha t}
\cdot
\left(\frac{2 m}{m+1}\right)^{\alpha t}
\cdot
(1 + O(z - 1))\,,
\end{equation}
where the error term is uniform in $t$ over compact subsets of complex numbers.
Now by~\cite[Theorem~VI.I]{Flajolet:Sedgewick} we have
\[
[z^n] \left(\frac{1}{1-z}\right)^{\alpha t}
=
\frac{n^{\alpha t - 1}}{\Gamma(\alpha t)}
\left( 1 + O\left( \frac{1}{n} \right)\right)\,,
\]
where the error term is uniform in $t$ over compact subsets of complex numbers.
The term $\left(\frac{2 m}{m+1}\right)^{\alpha t}$
in~\eqref{eq:proof:of:Th;3:7} does not depend on $z$.
In order to bound the contribution of the error term
$\left(\frac{1}{1-z}\right)^{\alpha t}
\cdot(1 + O(z - 1))$ in~\eqref{eq:proof:of:Th;3:7}
we use the
following estimate~\cite[Theorem~VI.3]{Flajolet:Sedgewick}:
\[
[z^n] O\left(\left(\frac{1}{1-z}\right)^{s - 1}\right)
=
O\left(n^{s - 2}\right)\,.
\]
Hence
\[
[z^n] \exp(t g_{\alpha,m}(z)) = \frac{n^{\alpha t - 1}}{\Gamma(\alpha t)}
\cdot
\left(\frac{2 m}{m+1}\right)^{\alpha t}
\cdot
\left( 1 + O\left( \frac{1}{n} \right)\right)
\]
and the theorem is proved.
\end{proof}

\subsection{Mod-Poisson convergence}
\label{s:mod:poisson}
In this section we recall some facts about mod-Poisson
convergence of probability distributions. As a direct
corollary of
Theorem~\ref{thm:multi:harmonic:sum:total:weight} we
derive mod-Poisson convergence of
the probability distribution $q_{n, m, \alpha}=
\Proba_{n, m, \alpha}\big(\NumCycles_n(\sigma) = k\big)$
of the number of cycles
associated by
Lemma~\ref{lem:cycle:distribution:vs:harmonic:sum}
to the normalized weighted multi-variate
harmonic sums $\ContributionH_{n,m,\alpha}(k)$.
For details we refer to the monograph
of V.~F\'eray, P.-L.~M\'eliot and
A.~Nikeghbali~\cite{FerayMeliotNikeghbali} and, for the
particular case of uniformly distributed random
permutations, to the original article of A.~Nikehgbali and
D.~Zeindler~\cite{NikeghbaliZeindler}.

Given a probability distribution $p(k)$ of a random
variable $X$ taking values in non-negative integers, we
associate to it the \textit{generating series}
\begin{equation}
\label{eq:generating:series:F}
F_p(t) = \sum_{k =1}^{+\infty} p(k) t^k\,.
\end{equation}
The generating series of the Poisson distribution
defined in~\eqref{eq:Poisson:def} is $e^{\lambda (t-1)}$.

Recall that given two independent discrete random variables
with non-negative integer values $X$ and $Y$ with
distributions $p_X(k)$ and $p_Y(k)$ respectively, the
distribution of their sum $X + Y$ is the convolution
\[
p_{X+Y}(k) = \sum_{i+j=k} p_X(i)\cdot p_Y(j).
\]
The generating series of $p_{X+Y}$
is the product of the generating series
of $p_{X}$ and $p_{Y}$:
\begin{equation}
\label{eq:convolution:generating:series}
F_{X+Y} = F_X F_Y.
\end{equation}
We are particularly interested in the situations when we
have a sequence of distributions that are close to the
convolution of the Poisson distribution with a varying
parameter $\lambda_n$ which tends to $+\infty$ as
$n\to+\infty$ and an additional fixed distribution.
\begin{Definition}
\label{def:mod:poisson:convergence}
Let $p_n$ be a sequence of probability distributions on the
non-negative integers; let $\lambda_n$ be a sequence of
positive real numbers tending to $+\infty$ as
$n\to+\infty$; let $R \in (1, +\infty]$;
let $G(t)$ be a function on
the disk $|t|\le R$ in $\C$ and let $\epsilon_n$ be a sequence of
positive real numbers converging to zero. We say that $p_n$
\emph{converges mod-Poisson with parameters $\lambda_n$,
limiting function $G$, radius $R$ and speed $\epsilon_n$}
if for all $t\in\C$ such that $|t|<R$ we have
\begin{equation}
\label{eq:def:mod:poisson}
F_{p_n}(t) = e^{\lambda_n (t - 1)} \cdot G(t)
\cdot (1 + O(\epsilon_n))\,,
\end{equation}
where the error term $O(\epsilon_n)$ is uniform over $t$
varying in compact subsets of
the complex disk $|t|<R$.
\end{Definition}

We say that a sequence $X_n$ of random variables taking
values in non-negative integers \textit{converges mod-Poisson}
if the sequence of the associated probability distributions
$p_n$ converges mod-Poisson, where $p_n(k)=\Proba(X_n=k)$
for $k=0,1,\dots$.

The term $e^{\lambda_n (t - 1)} \cdot G(t)$ in the right
hand side of~\eqref{eq:def:mod:poisson} is the product of
the generating series of $\Poisson_{\lambda_n}$ with
$G(t)$. In other words, it looks
like~\eqref{eq:convolution:generating:series}. Note,
however, we emphasize that $G(t)$ is not necessarily the
generating series of a probability distribution.

Note that Equation~\eqref{eq:generating:series:F}
implies that for any $n$ we have $F_{p_n}(1)=1$.
Thus, condition~\eqref{eq:def:mod:poisson}
from the definition of mod-Poisson convergence
implies that
\begin{equation}
\label{eq:G:1:equals:1}
G(1)=1\,.
\end{equation}

\begin{Remark}
\label{rk:mod:poisson:diff}
Let us emphasize that our definition of mod-Poisson
convergence differs from~\cite{FerayMeliotNikeghbali} in
that we take generating series $\E(t^X)$ of random
variables instead of the moment generating function
$\E(e^{z X})$. One can pass from one to the other by
setting $t = e^z$. In particular, our assumption
that $G$ is analytic at $t=0$ is not
a requirement in the definition of~\cite{FerayMeliotNikeghbali}.
This extra assumption allows us to control the asymptotics of
$p_n(k)$ when $k$ is in the range $k \ll \log n$.
\end{Remark}

Let $\Proba_{n,m,\alpha}$ be the
discrete probability measure on the symmetric group $S_n$
corresponding to the weights $w_\theta(\sigma)$
associated to the sequence
$\theta_i=\alpha\cdot \zeta_m(2i)$ for $i=1,2\dots$ as defined
in~\eqref{eq:proba:sym:group}.
Recall that Lemma~~\ref{lem:cycle:distribution:vs:harmonic:sum}
and, more specifically, Equation~\eqref{eq:q:n:m:alpha}
expresses the probability distribution
$q_{n,m,\alpha}(k)
=\Proba_{n, m, \alpha}\big(\NumCycles_n(\sigma) = k\big)$
through multivariate harmonic sums
$\ContributionH_{n,m,\alpha}(k)$
defined in~\eqref{eq:multiple:harmonic:sum:def}.
The corollary below is a more general version of
Theorem~\ref{thm:permutation:mod:poisson:introduction} from
the introduction.
\begin{Corollary}
\label{cor:permutation:mod:poisson}
Let $\NumCycles_n(\sigma)$ be the number of cycles of a
permutation $\sigma$ in the symmetric group $S_n$.
Let $\E_{n, m, \alpha}$ be the expectation with respect
to the probability measure $\Proba_{n, m, \alpha}$ on
$S_n$ as in~\eqref{eq:q:n:m:alpha}.

For all $t \in \C$ we have as $n \to +\infty$
\begin{equation}
\label{eq:permutation:mod:poisson}
\E_{n, m, \alpha}\left(t^{\NumCycles_n}\right)
= \left(\frac{2m}{m+1} n\right)^{\alpha (t - 1)}
\cdot \frac{\Gamma(\alpha)}{\Gamma(\alpha t)}
\left(1 + O \left( \frac{1}{n} \right) \right)\,,
\end{equation}
Moreover, the convergence
in~\eqref{eq:permutation:mod:poisson} is uniform for $t$ in
any compact subset of $\C$.

In other words the sequence of random variables
$\NumCycles_n$ with respect to the probability measures
$\Proba_{n,m,\alpha}$ converges mod-Poisson with parameter
$\lambda_n=\alpha \log\left(\frac{2m}{m+1} n\right)$,
limiting function $\Gamma(\alpha) / \Gamma(\alpha t)$,
radius $R=+\infty$ and speed $1/n$.
\end{Corollary}

\begin{proof}[Proof of Corollary~\ref{cor:permutation:mod:poisson}]
Let us define
\[
G_n(t) := \sum_{k=1}^n \ContributionH_{n,m,\alpha}(k) t^k.
\]
Formula~\eqref{eq:generating:series:F} for an abstract
generating function combined with
Formula~\eqref{eq:q:n:m:alpha} for $\Proba_{n, m,
\alpha}\big(\NumCycles_n(\sigma) = k\big)$ give the
following expression for the generating series in the
left-hand side of~\eqref{eq:permutation:mod:poisson}:
\[
\E_{n,m,\alpha}\left(t^{\NumCycles_n(\sigma)}\right)
=
\frac{G_n(t)}{G_n(1)}
\]
and the corollary now
directly follows from
Formula~\eqref{eq:multi:harmonic:sum:total:weight} from
Theorem~\ref{thm:multi:harmonic:sum:total:weight}.
\end{proof}

Generalizing $u_{\lambda, 1/2}$ defined in~\eqref{eq:poisson:gamma:1:2} let
us define
\begin{equation}
\label{eq:poisson:gamma}
e^{\lambda (t-1)}\cdot
\frac{t \cdot \Gamma(1 + \alpha)}{\Gamma(1 + t \alpha)}
=
\sum_{k \geq 1} u_{\lambda, \alpha}(k) \cdot t^k.
\end{equation}

\begin{Corollary}
\label{cor:approximation:q:u}
Let $m\in\N\cup\{\infty\}$  and let $\alpha$ be a positive real.
Uniformly in $k \geq 1$ we have as $n \to +\infty$
\[
q_{n, m, \alpha}(k)
= u_{\lambda_n, \alpha}(k) + O \left(\frac{1}{n}\right)\,,
\]
where $\lambda_n=\alpha \log\left(\frac{2m}{m+1} n\right)$.
\end{Corollary}

\begin{proof}[Proof of Corollary~\ref{cor:approximation:q:u}]
Note that
$$
\frac{\Gamma(\alpha)}{\Gamma(\alpha t)}
=\frac{\alpha t\cdot\Gamma(\alpha)}{\alpha t\cdot\Gamma(\alpha t)}
=\frac{t \cdot \Gamma(1 + \alpha)}{\Gamma(1 + t \alpha)}\,.
$$
Let
\[
F_1(t) := \sum_{k \geq 1} q_{n,m,\alpha}(k) t^k
\quad
F_2(t) :=
e^{\lambda (t-1)} \cdot \frac{t \cdot \Gamma(1 + \alpha)}{\Gamma(1 + t \alpha)}
=
\sum_{k \geq 1} u_{\lambda_n, \alpha}(k) t^k.
\]
Both $F_1(t)$ and $F_2(t)$ are holomorphic in $\C$.
Since $F_2(t)$ does not vanish we have uniformly in $t \in D(0, 1+\epsilon)$
\[
F_1(t) - F_2(t) = O\left(\frac{1}{n}\right).
\]
Using the saddle-point bound~(18)
from~\cite[Proposition~IV.1]{Flajolet:Sedgewick} with radius $R=1$ we obtain
uniformly in $k \geq 1$
\[
q_{n,m,\alpha}(k) - u_{\lambda_n, \alpha}(k)
=
O\left(\frac{1}{n}\right).
\]
\end{proof}

\subsection{Large deviations and central limit theorem}
\label{ss:LD:and:CLT}
Having proved the mod-Poisson convergence in
Corollary~\ref{cor:permutation:mod:poisson}, we could derive
most of the following large deviation results by referring
to Theorem~3.2.2 from the monograph of V.~Feray,
\mbox{P.-L.~M\'eliot}, A.~Nikeghbali~\cite{FerayMeliotNikeghbali}
(see also Example~3.2.6 of the same monograph providing
more details in the case of uniform random permutations).
However, as we mentioned in
Remark~\ref{rk:mod:poisson:diff}, the
monograph~\cite{FerayMeliotNikeghbali} uses slightly
weaker definition of mod-Poisson convergence which does not
allow to study the probability distribution in the range of
values of the random variable of the order $o(\lambda_n)$.
To overcome this diffiulty we rely on Theorem~14
in~\cite{Hwang:PhD} and on Theorem~2 in~\cite{Hwang:distances}
due to H.~Hwang.
\begin{Theorem}[H.~Hwang \cite{Hwang:PhD}, \cite{Hwang:distances}]
\label{thm:large:deviations}
Let $\{X_n\}_n$ be a sequence of random variables taking
values in non-negative integers that converges mod-Poisson with
parameters $\lambda_n$, limiting function $G(t)$, radius $R$
and speed at least $\lambda_n^{-1}$. Assume furthermore
that $G(0) \not= 0$.

For any $x \in (0, R)$, uniformly in $0 \leq k \leq x \lambda_n$
we have as $n \to +\infty$
\begin{equation}
\label{eq:LD:fixed:k}
\Proba(X_n = k)
=
e^{-\lambda_n} \frac{\lambda_n^k}{k!}
\cdot (G(k / \lambda_n) + O((k+1) / (\lambda_n)^2))\,.
\end{equation}
For all $x \in (1,R)$ such that $x \lambda_n$ is an integer
\begin{equation}
\label{eq:LD:k:bigger:than}
\Proba(X_n > x \lambda_n)
=
\frac{e^{-\lambda_n (x \log x - x + 1)}}{\sqrt{2 \pi \lambda_n x}}
\cdot
\frac{x}{x-1}
\cdot
(G(x) + O(\lambda_n^{-1}))
\end{equation}
where the error term is uniform over $x$ in compact subsets of $(1,R)$.
Similarly, for all $x \in (0, 1)$ such that $x \lambda_n$ is an integer
\begin{equation}
\label{eq:LD:k:smaller:than}
\Proba(X_n \leq x \lambda_n)
=
\frac{e^{-\lambda_n (x \log x - x + 1)}}{\sqrt{2 \pi \lambda_n x}}
\cdot
\frac{x}{1-x}
\cdot
(G(x) + O(\lambda_n^{-1}))
\end{equation}
where the error term is uniform over $x$ in compact subsets of $(0,1)$.
\end{Theorem}

\begin{Remark}
Note that by Stirling formula, for $x = \frac{k}{\lambda_n}$ we have
\[
\frac{e^{- \lambda_n (x \log x - x + 1)}}{\sqrt{2\pi x\lambda_n}}
=
e^{-\lambda_n} \frac{(\lambda_n)^k}{k!}
(1 + O((\log n)^{-1}))\,.
\]

Note also that $x \log(x) - x + 1$ is convex and attains it minimum at
$x=1$ for which it has value zero. Hence both quantities in the right-hand
sides of~\eqref{eq:LD:k:bigger:than} and~\eqref{eq:LD:k:smaller:than}
are exponentially decreasing in $n$.
\end{Remark}

\begin{Remark}
\label{Remark:shift}
If the limiting function $G$ vanishes at $0$ we can apply
the following trick. Let $a\in\N$ be the order of the zero.
Then
the sequence of shifted variables $X_n - a$ converges
mod-Poisson with the same parameters and radius but with
the limiting function $t^{-a} G(t)$ which does not vanish
anymore at zero. We can then apply
Theorem~\ref{thm:large:deviations} to $X_n - a$.
\end{Remark}

\begin{Remark}
Since \cite[Theorem~3.2.2]{FerayMeliotNikeghbali} is stated
for the more general mod-$\phi$ convergence let us explain
how their notations translate in our context. Because we
use Poisson variables we have $\eta(t) = e^t - 1$ whose
Legendre-Fenchel transform is $F(x) = x \log x - x - 1$.
Because of this $h = \log x$. The limiting function is
$\phi(e^z) = G(z)$. This difference of notation for the
limiting functions is due to the fact that we used
generating series $\E(t^X)$ instead of moment generating
functions $\E(e^{z X})$.
\end{Remark}

The statement below is a generalization of
Theorem~\ref{thm:permutation:asymptotics} from
Section~\ref{s:intro} to arbitrary probability measure
$\Proba_{n,m,\alpha}$.

\begin{Corollary}
\label{cor:multi:harmonic:asymptotic:all:k}
Let $\alpha > 0$, $m \in\N\cup\{+\infty\}$ and let $\Proba_{n,m,\alpha}$ be the probability
measure as in~\eqref{eq:q:n:m:alpha}. Let
$\lambda_n := \alpha \log\left(\frac{2m}{m+1} n\right)$.
Let $x > 0$. Then, uniformly in $0 \leq k \leq x \lambda_n$, we have
\begin{multline}
\label{eq:LD1:q}
q_{n,m,\alpha}(k+1)=
\Proba_{n,m,\alpha}(\NumCycles_n = k+1)
=\\
=e^{-\lambda_n} \frac{(\lambda_n)^{k}}{k!}
\left( \frac{\Gamma(1+\alpha)}{\Gamma\left(1 + \alpha \frac{k}{\lambda_n}\right)}
+ O\left(\frac{k+1}{(\log n)^2} \right) \right)\,.
\end{multline}
For $x \in (1, +\infty)$ such that $x \lambda_n$ is an integer we have
\begin{multline}
\label{eq:LD2:q}
\sum_{k=x\lambda_n+1}^n q_{n,m,\alpha}(k+1)=
\Proba_{n,m,\alpha}\big(\NumCycles_n > x \lambda_n +1\big)
=\\
=\frac{e^{-\lambda_n (x \log x - x + 1)}}{\sqrt{2 \pi x \lambda_n}} \cdot
\frac{x}{x - 1}
\left( \frac{\Gamma(1+\alpha)}{\Gamma(1 + \alpha x)}
+ O\left(\frac{1}{\log n} \right) \right)\,.
\end{multline}
where the error term is uniform over $x$ in compact subsets of $(1, +\infty)$
and for $x \in (0, 1)$ such that $x \lambda_n$ is an integer we have
\begin{multline}
\label{eq:LD3:q}
\sum_{k=0}^{x\lambda_n} q_{n,m,\alpha}(k+1)=
\Proba_{n,m,\alpha}\big(\NumCycles_n \leq x \lambda_n +1\big)
=\\
=\frac{e^{-\lambda_n (x \log x - x + 1)}}{\sqrt{2 \pi x \lambda_n}} \cdot
\frac{x}{1 - x}
\left( \frac{\Gamma(1+\alpha)}{\Gamma(1 + \alpha x)}
+ O\left(\frac{1}{\log n} \right) \right)\,.
\end{multline}
where the error term is uniform over $x$ in compact subsets of $(0, 1)$
\end{Corollary}
\begin{proof}
By Corollary~\ref{cor:permutation:mod:poisson}, the
sequence of random variables $\NumCycles_n(\sigma)$ with
respect to the probability measures $\Proba_{n,m,\alpha}$
on the symmetric group $S_n$ converges mod-Poisson with
parameters $\lambda_n = \alpha \log\left(\frac{2m}{m+1}
n\right)$, limiting function $\Gamma(\alpha) /
\Gamma(\alpha t)$, radius $R=+\infty$ and speed $1/n$. The
limiting function $\Gamma(\alpha) / \Gamma(\alpha t)$ has
zero of the first order at $t=0$, so we have to apply the
trick described in Remark~\ref{Remark:shift}. The
sequence of random variables $\NumCycles_n-1$
converges mod-Poisson with the same radius $R=+\infty$ and
speed $1/n$ and has the limiting function
$\Gamma(\alpha) / (t\cdot\Gamma(\alpha t))$. Applying the
identity $\Gamma(z+1)=z\Gamma(z)$ we conclude that the new
limiting function
$$
G(t)
=\frac{\Gamma(\alpha)}{t\cdot\Gamma(\alpha t)}
=\frac{\alpha\Gamma(\alpha)}{\Gamma(1+\alpha t)}
=\frac{\Gamma(1+\alpha)}{\Gamma(1+\alpha t)}
$$
does not vanish at $t=0$ and Theorem~\ref{thm:large:deviations}
becomes applicable to the sequence of random variables $X_n-1$.
\end{proof}


\begin{Corollary}
\label{cor:multi:harmonic:asymptotic:all:k:bis}
Let $\alpha$ be a positive real number and let
$m\in\N\cup\{+\infty\}$. Let $\ContributionH_{n,m,\alpha}$
be the normalized weighted multi-variate harmonic
sum~\eqref{eq:multiple:harmonic:sum:def}.

Let $\{k_n\}_n$ be a sequence
of integers such that $k_n = O(\log n)$. Then as $n \to
+\infty$ we have
\begin{multline}
\label{eq:H:n:m:alpha}
\ContributionH_{n,m,\alpha}(k_n)=
\\
=\frac{\alpha^{k_n}}{n}
\frac{\left(\log n+\log\left(\tfrac{2m}{m+1}\right)\right)^{k_n-1}}
{(k_n-1)!}
\cdot\left(
\frac{1}
{\Gamma\left(1 + \alpha \frac{k_n-1}{\lambda_n}\right)}
+ O\left(\frac{k_n-1}{(\log n)^2} \right) \right)\,.
\end{multline}

If, moreover, $k_n = o(\log n)$, then as $n \to
+\infty$ we have
\begin{multline}
\label{eq:H:n:m:alpha:o}
\ContributionH_{n,m,\alpha}(k_n)
=\\
=\frac{\alpha^{k_n}}{n}
\frac{\left(\log n+\log\left(\tfrac{2m}{m+1}\right)\right)^{k_n-1}}
{(k_n-1)!}
\left(1+\frac{\gamma\cdot (k_n-1)}{\log n}
+O\left(\left(\frac{k_n}{\log n}\right)^2 \right)\right)\,.
\end{multline}
\end{Corollary}
\begin{proof}[Proof of Corollary~\ref{cor:multi:harmonic:asymptotic:all:k:bis}]
Applying~\eqref{eq:LD1:q} with
$\lambda_n = \alpha \log\left(\frac{2m}{m+1} n\right)$
we get
\begin{multline*}
q_{n,m,\alpha}(k_n)
=e^{-\lambda_n}
\frac{(\lambda_n)^{k_n-1}}{(k_n-1)!}
\left(
\frac{\Gamma(1+\alpha)}
{\Gamma\left(1 + \alpha \frac{k_n-1}{\lambda_n}\right)}
+ O\left(\frac{k_n-1}{(\log n)^2} \right) \right)
=\\
=\left(\frac{2m}{m+1}\right)^{-\alpha} n^{-\alpha}
\frac{\left(\alpha\log\left(\tfrac{2m}{m+1}n\right)\right)^{k_n-1}}
{(k_n-1)!}
\left(
\frac{\Gamma(1+\alpha)}
{\Gamma\left(1 + \alpha \frac{k_n-1}{\lambda_n}\right)}
+ O\left(\frac{k_n-1}{(\log n)^2} \right) \right)\,.
\end{multline*}
Applying
Equation~\eqref{eq:multi:harmonic:sum:total:weight} with
the value $t=1$, we get
$$
W_{n,m,\alpha}=
\sum_{k=1}^n \ContributionH_{n,m,\alpha}(k)  =
\frac{\alpha\left( \frac{2m}{m+1} \right)^{\alpha} n^{\alpha - 1}}
{\Gamma(1+\alpha)}
\left(1 + O\left(\frac{1}{n}\right)\right)\,,
$$
where we used the identity $\alpha\Gamma(\alpha)=\Gamma(1+\alpha)$.
By definition~\eqref{eq:q:n:m:alpha}
of $q_{n,m,\alpha}(k)$ we have
$$
\ContributionH_{n,m,\alpha}(k)
=q_{n,m,\alpha}(k)\cdot W_{n,m,\alpha}
$$
Multiplying the two expressions computed above
we get~\eqref{eq:H:n:m:alpha}.

To prove~\eqref{eq:H:n:m:alpha:o} we use
the asymptotic expansion
$$
\frac{1}{\Gamma(1+t)}=1+\gamma t + O(t^2)\quad
\text{as }\ t\to 0\,,
$$
where $\gamma = 0.5572\ldots$ denotes the Euler--Mascheroni
constant.
\end{proof}

Note that for the values of parameters $m=\alpha=1$ and for
the constant sequence $k_n=2$ for $n=1,2,\dots$, the
expansion~\eqref{eq:H:n:m:alpha:o} gives
$$
\ContributionH_{n,1,1}(k_n)
=\frac{1}{n}\log n
\left(1+\frac{\gamma}{\log n}
+O\left(\frac{1}{(\log n)^2}\right)\right)
$$
corresponding to the classical formula
$$
\frac{1}{2}\sum_{i=1}^n\frac{1}{j\cdot (n-j)}
=\frac{1}{n}\left(\log n+\gamma+O\left(\frac{1}{n}\right)\right).
$$

The following strong form of the central limit theorem
corresponds to Theorem~3.3.1
of~\cite{FerayMeliotNikeghbali}.
\begin{Theorem}[V.~F\'eray, P.-L.~M\'eliot, A.~Nikeghbali~\cite{FerayMeliotNikeghbali}]
\label{thm:CLT}
Let $\{X_n\}_n$ be a sequence of random variables on the non-negative
integers that converges mod-Poisson with parameters $\lambda_n$.
Let $x_n$ be a sequence of real numbers with $x_n = o( (\lambda_n)^{1/6})$. Then as $n \to +\infty$
\[
\Proba\left( \frac{X_n - \lambda_n}{\sqrt{\lambda_n}} \le x_n\right)
=
\left(\frac{1}{\sqrt{2\pi}}
\int_{-\infty}^{x_n} e^{-\tfrac{t^2}{2}} dt\right) (1 + o(1))
\]
\end{Theorem}
Note that, contrarily
to the large deviations, the radius $R$, the limiting
function $G$ and the speed $\epsilon_n$ of the mod-Poisson
convergence are irrelevant in the above Theorem.

\begin{Corollary}
Let $\alpha > 0$, $m\in\N\cup\{+\infty\}$
and let $\Proba_{n,m,\alpha}$ be the probability
distribution on the symmetric group defined in~\eqref{eq:q:n:m:alpha}. Let
$\lambda_n := \alpha \log\left(\frac{2m}{m+1} n\right)$ and
$x_n$ be a sequence of real numbers with $x_n = o( (\lambda_n)^{1/6})$. Then as $n \to +\infty$
\[
\Proba_{n,m,\alpha}
\left( \frac{\NumCycles_n - \lambda_n}{\sqrt{\lambda_n}}
\le x_n\right)
=
\left(\frac{1}{\sqrt{2\pi}}
\int_{-\infty}^{x_n} e^{-\tfrac{t^2}{2}} dt\right) (1 + o(1)).
\]
\end{Corollary}
\begin{proof}
By Corollary~\ref{cor:permutation:mod:poisson},
the sequence of random variables $\NumCycles_n$
converges mod-Poisson, so Theorem~\ref{thm:CLT}
is applicable to this sequence.
\end{proof}

\subsection{Moments of the Poisson distribution}
\label{ss:Moments:of:Poisson:distribution}
Recall that given a non-negative integer $n$ and a positive
real number $\lambda$, the $n$-th moment $P_n(\lambda)$ of
a random variable corresponding to the Poisson distribution
$\Poisson_\lambda$ with parameter $\lambda$ is defined as
\begin{equation}
\label{eq:def:poisson:moments}
P_n(\lambda) := e^{-\lambda} \cdot
\sum_{k = 0}^{+\infty} k^n \frac{\lambda^k}{k!}.
\end{equation}

Recall that given two integers $n,k$ satisfying $1\le
k\le n$, the \textit{Stirling number of the second kind},
denoted $S(n,k)$, is the number of ways to partition a set
of $n$ objects into $k$ non-empty subsets.
It is well-known that the Stirling number of the second kind
satisfy the following recurrence relation:
\begin{equation}
\label{eq:Stirling:2:reccurrence}
S(n+1,k)=k\cdot S(n,k)+S(n,k-1)\,,
\end{equation}
and are uniquely determined by the
initial conditions, where we set by convention:
$S(0,0)=1$ and $S(n,-1)=S(n,0)=S(0,n)=S(n,n+1)=0$ for
$n\in\N$.

Though the following statement is well-known, see, for
example, \cite{Riordan}, its proof is so short that we
present it for the sake of completeness.

\begin{Lemma}
\label{lem:poisson:moments}
For any $n\in\N$, the expression $P_n(\lambda)$
defined in~\eqref{eq:def:poisson:moments}
coincides with the following
monic polynomial in $\lambda$ of degree $n$:
\begin{equation}
\label{eq:P:n:lambda:Stirling}
P_n(\lambda) = \sum_{k=0}^n S(n,k) \lambda^k\,,
\end{equation}
where $S(n,k)$ are the Stirling numbers of the second kind.
\end{Lemma}

The polynomials $P_n(\lambda)$ are sometimes called
\textit{Touchard polynomials}, \textit{exponential
polynomials} or \textit{Bell polynomials}. For $n\le 4$ the
polynomials $P_n(\lambda)$ have the following explicit
form:
\begin{equation}
\label{eq:first:four:P:lambda}
\begin{aligned}
P_0(\lambda) &= 1\,, \\
P_1(\lambda) &= \lambda\,, \\
P_2(\lambda) &= \lambda^2 + \lambda\,, \\
P_3(\lambda) &= \lambda^3 + 3\lambda^2 + \lambda\,, \\
P_4(\lambda) &= \lambda^4 + 6\lambda^3 + 7\lambda^2 + \lambda\,.
\end{aligned}
\end{equation}

\begin{proof}[Proof of Lemma~\ref{lem:poisson:moments}]
Let $X$ be a random variable with distribution
$\Poisson_\lambda$ and let
$$
\phi(z) = \E(e^{z X}) = \sum_{n=0}^{+\infty} \E(X^n)
\frac{z^n}{n!}
$$
be its moment generating series. Then
\[
\phi(z) = \sum_{k \geq 0} e^{-\lambda} \frac{\lambda^k}{k!} e^{z k}
= e^{-\lambda} \sum_{k \geq 0} \frac{(\lambda e^z)^k}{k!}
= e^{\lambda (e^z - 1)}.
\]
By definition, $P_n(\lambda) = \frac{d^n}{dz^n} \phi(z)
|_{z=0}$. We claim that for any $n=0,1,\dots$
the following identity holds:
\begin{equation}
\label{eq:poisson:stirling}
\frac{d^n}{dz^n} \phi(z)
= \sum_{k=0}^n S(n,k)\cdot (\lambda e^z)^k\cdot \phi(z)\,.
\end{equation}
Indeed $\frac{d}{dz} \phi(z) = \lambda e^z \phi(z)$ and,
hence, the identity holds
for $n=0$ and $n=1$. Taking the derivative of the expression
in the right hand side of~\eqref{eq:poisson:stirling}
we obtain
\begin{align*}
\frac{d}{dz}
\sum_{k=0}^n S(n,k)\cdot \big(\lambda e^z\big)^k\cdot \phi(z)
&=
\sum_{k=0}^n S(n,k)\cdot
\Big(k\cdot\big(\lambda e^z\big)^k + \big(\lambda e^z\big)^{k+1}\Big)\cdot \phi(z) \\
&=\sum_{k=0}^{n+1}
\Big(S(n,k-1) + k\cdot S(n,k)\Big)\cdot (\lambda e^z)^k\cdot
\phi(z)\,.
\end{align*}
We recognize the recurrence
relations~\eqref{eq:Stirling:2:reccurrence} for Stirling
numbers of the second kind, which proves
identity~\eqref{eq:poisson:stirling}.
Taking $z=0$ in~\eqref{eq:poisson:stirling}
we obtain~\eqref{eq:P:n:lambda:Stirling}.
\end{proof}

\subsection{Moment expansion}
\label{ss:moment:expansion}
In this section we analyze the asymptotic expansions of
cumulants of probability distributions that satisfies
mod-Poisson convergence. We then apply it to the
probability distribution $q_{n,m,\alpha}(k)=
\frac{\ContributionH_{n,m,\alpha}(k)}{W_{n,m,\alpha}}$
(see Definition~\ref{def:hkzk} and~\eqref{eq:q:n:m:alpha}).

The \emph{cumulants} $\kappa_i(X)$ of a random variable $X$ are the
coefficients of the expansion
\[
\log \E(e^{t X}) = \sum_{i \geq 1} \kappa_i(x) \frac{t^i}{i!}.
\]
The first cumulant $\kappa_1(X) = \E(X)$ is the mean and the second cumulant
$\kappa_2(X) = \Var(X) = \E(X^2) - \E(X)^2$ is the variance.
The cumulants are combinations of moments, but contrarily to moments,
cumulants are additive: if $X$ and $Y$ are independent then $\kappa_i(X + Y) = \kappa_i(X) + \kappa_i(Y)$.

If $X$ is a Poisson random variable with parameter $\lambda$ then
\[
\log \E(e^{t X}) = \lambda (e^t - 1).
\]
This implies that all cumulants of a Poisson random
variable are equal to $\lambda$. The Theorem below proves
that when a sequence of random variables converges
mod-Poisson, the main contribution to the cumulants comes
from the Poisson part while an explicit correction comes
from the logarithmic derivative of the limiting function.

\begin{Theorem}
\label{thm:cumulants}
Let $X_n$ be a sequence of probability distributions that
converges mod-Poisson with parameters $\lambda_n$ limiting
function $G$ and speed $\epsilon_n$ as $n\to+\infty$. Then for all $i \geq
1$, as $n \to +\infty$ we have the following asymptotic
equivalence of the $i$-th cumulant
\begin{equation}
\label{eq:cumulants:abstract}
\kappa_i(X_n) = \lambda_n + \sum_{k=1}^i S(i,k)
\cdot \delta_k + O(\epsilon_n)\,,
\end{equation}
where $S(i,k)$ are the Stirling numbers of the second kind
and $\delta_k = \frac{d^k}{dt^k} \log G(t)|_{t=1}$ are the values of the
logarithmic derivatives of the limiting function $G$ at $t=1$.
\end{Theorem}
We warn the reader that the error term $O(\epsilon_n)$
in~\eqref{eq:cumulants:abstract} is uniform in $n$ but not
in $i$.

\begin{Remark}
\label{rm:cumulants:do:not:depend:on:R}
Note that the Theorem above is valid for any radius of
convergence $R$ as soon as $R>1$. The latter inequality
makes part of Definition~\ref{def:mod:poisson:convergence}
of mod-Poisson convergence.
\end{Remark}

\begin{proof}
By Definition~\ref{def:mod:poisson:convergence}
of the mod-Poisson convergence we have
\[
\E(e^{z X_n}) = e^{\lambda_n (e^z - 1)} G(e^z) (1 + O(\epsilon_n))\,,
\]
see~\eqref{eq:def:mod:poisson}.
We have seen in~\eqref{eq:G:1:equals:1} that our definition
of mod-Poisson convergence implies that
$G(1) = 1$. Hence, there exists a radius $R'$ such that
for $|z| <R'$ we have $G(e^z) \not\in [-\infty,0)$. On
the disk $|z| <R'$ we can take the principal determination of the logarithm to obtain
$$
\log \E(e^{z X_n}) = \lambda_n (e^z - 1) + \log G(e^z) + O(\epsilon_n).
$$
which can be rewritten as
\begin{equation}
\label{eq:moment}
\sum_{i \geq 1} (\kappa_i(\NumCycles_n) - \lambda_n - \Delta_i) \cdot \frac{z^i}{i!} = O(\epsilon_n)
\end{equation}
where $\Delta_i := \frac{d^i}{dz^i} \log G(e^z)|_{z=0}$.
Let $g^{(i)}(t) = \frac{d^i}{dt^i} \log G(t)$.
The rest of the proof is similar to the proof of of Lemma~\ref{lem:poisson:moments}.
Namely, we claim that for $i \geq 1$ we have
\begin{equation}
\label{eq:cumulant:stirling2}
\frac{d^i}{dz^i} \log G(e^z) = \sum_{k=1}^i S(i,k) g^{(k)}(e^z) e^{kz}.
\end{equation}
It is indeed the case for $i=1$ and when differentiating once the formula
in the right hand side we obtain
\begin{align*}
\frac{d}{dz} \sum_{k=1}^i S(i,k) g^{(k)}(e^z) e^{kz} &=
\sum_{k=1}^i (S(i,k) g^{(k+1)}(e^z) e^{(k+1)z} + k S(i,k) g^{(k)}(e^z) e^{kz}) \\
&= \sum_{k=1}^{i+1} (S(i,k-1) + k S(i,k)) g^{(k)}(e^z) e^{k z}.
\end{align*}
We recognize the recurrence relation~\eqref{eq:Stirling:2:reccurrence}
for the
unsigned Stirling numbers of the second kind. This proves the claim.
Now let $\delta_i = g^{(i)}(1)$. Specializing~\eqref{eq:cumulant:stirling2} at $z=0$
we obtain $\Delta_i = \sum_{k=1}^i S(i,k) \delta_k$.

Now, since the radius of convergence in~\eqref{eq:moment} is positive, we obtain
\[
\kappa_i(\NumCycles_n) - \lambda_n - \delta_i(\alpha) = O(\epsilon_n)
\]
(where the error term depends on $i$). This concludes the proof.
\end{proof}

Recall that for $m \geq 0$, the $m$-th polygama function is
defined as
\begin{equation}
\label{eq:polygamma}
\psi^{(m)}(z) = \frac{d^{m+1}}{dz^{m+1}} \log \Gamma(z).
\end{equation}

\begin{Corollary}
\label{cor:cumulants:q}
Let $m\in\N\cup\{+\infty\}$, $\alpha \geq 0$ and
let $\NumCycles_n$ be the random variable
corresponding to the probability law
$q_{n,m,\alpha}$
defined in~\eqref{eq:q:n:m:alpha}.
Then for any $i\in\N$ we have the
following asymptotic equivalence
for the $i$-th cumulant of $\NumCycles_n$
as $n \to +\infty$:
\begin{equation}
\label{eq:cumulant:q:n:m:alpha}
\kappa_i(\NumCycles_n) = \alpha \log\left(\frac{2m}{m+1} \cdot n\right)
- \sum_{k=1}^i S(i,k) \cdot \psi^{(k-1)}(\alpha) \cdot \alpha^k
+ O\left(\frac{1}{n}\right)\,,
\end{equation}
where $S(i,k)$ is the Stirling number of the second kind
and $\psi^{(j)}$ is the polygamma function.
\end{Corollary}
\begin{proof}
By Corollary~\ref{cor:permutation:mod:poisson}, the random
variables $\NumCycles_n$ with respect to
$\Proba_{n,m,\alpha}$ converges mod-Poisson with parameters
$\lambda_n = \alpha \log\left(\frac{2m}{m+1} n\right)$,
limiting function $G(t) = \Gamma(\alpha) / \Gamma(\alpha
t)$ and rate $O(1/n)$. The logarithmic derivatives of the
limiting function are expressed in terms of the polygamma
function by the following relation:
$$
\frac{d^k}{dt^k} \log \frac{\Gamma(\alpha)}{\Gamma(\alpha t)}
=-\frac{d^k}{dt^k}\log\Gamma(\alpha t)
=- \Gamma(\alpha)
\cdot\alpha^k
\cdot\psi^{(k-1)}(\alpha t)\,.
$$
Applying Equation~\eqref{eq:cumulants:abstract}
from Theorem~\ref{thm:cumulants} we obtain
the desired relation~\eqref{eq:cumulant:q:n:m:alpha}.
\end{proof}

The derivatives of the polygamma functions at rational
points have explicit expressions in terms of
$\zeta$-values. The following lemma provides the values of
these derivatives at $1$ and at $1/2$ relevant for
the purposes of the current paper.
These formulae reproduce
Formulae~6.4.2 and~6.4.4, at page~260
of~\cite{Abramowitz:Stegun}. The proofs can be found,
for example, in the paper~\cite{Choi:Cvijovic}
of J. Choi and D. Cvijovi\'c.

\begin{Lemma}
We have
\[
\psi^{(0)}(1) = -\gamma
\qquad
\psi^{(0)}(1/2) = -\gamma - 2 \log 2
\]
and for $m \geq 1$
\begin{align*}
\psi^{(m)}(1) &= (-1)^{m+1} \, \zeta(m+1) \, m! \\
\psi^{(m)}(1/2) &= (-1)^{m+1} \, \zeta(m+1) \, m! \, \left(2^{m+1} - 1 \right)\,.
\end{align*}
\end{Lemma}

\begin{Remark}
In the special case $m=1$ and $\alpha=1$ which corresponds to the uniform distribution
on $S_n$ we obtain
\begin{align*}
\kappa_1(q_{n,1,1}) &= \log n + \gamma + O(1/n) \\
\kappa_2(q_{n,1,1}) &= \log n + \gamma - \zeta(2) + O(1/n) \\
\kappa_3(q_{n,1,1}) &= \log n + \gamma - 3 \zeta(2) + 2 \zeta(3) + O(1/n) \\
\kappa_4(q_{n,1,1}) &= \log n + \gamma - 7 \zeta(2) + 12 \zeta(3) - 4 \zeta(4) + O(1/n).
\end{align*}
We recover the expression~\eqref{eq:mean:variance:perm}
obtained by V.~L.~Goncharov~\cite{Goncharov} for the expectation and variance of
the cycle count of random uniform permutations in $S_n$.
\end{Remark}

%
%
%
%
%

\subsection{From $q_{3g-3,\infty,1/2}$ to $\ProbaCylsOne_g$}
\label{ss:from:q:to:p1}
Recall that
$\Vol\big(\Petal_k(g),(m_1, \ldots, m_k))$
denotes the volume contribution from those
square-tiled surfaces corresponding to the stable graph
$\Petal_k(g)$ for which the first maximal horizontal
cylinder is filled with $m_1$ bands of squares, the second
cylinder is filled with $m_2$ bands of
squares, and so on up to the $k$-th maximal horizontal
cylinder, which is filled with
$m_k$ bands of squares. Recall also that
for any $m\in\N\cup\{\infty\}$
we defined in~\eqref{eq:V:m:k} the quantities
$$
V_{m,k}(g) =\!\!\!\!
\sum_{\substack{m_1, \ldots, m_k\\1\le m_i \le m\ \text{for }i=1,\dots,k}}
\!\!\!\Vol(\Petal_k(g), (m_1, \ldots, m_k))
\quad\text{and}\quad
V_{m}(g)=\sum_{k=1}^g V_{m,k}(g)\,.
$$
We define the probability distribution
$\ProbaCylsOne_{g,m}(k)$ for $k=1,\dots,g$
as
\begin{equation}
\label{eq:ProbaCylsOne}
\ProbaCylsOne_{g,m}(k):=\frac{V_{m,k}(g)}{V_m(g)}\,.
\end{equation}
We will sometimes denote $\ProbaCylsOne_{g,\infty}$
just by $\ProbaCylsOne_g$.
In this section we use estimates~\eqref{eq:bounds:in:terms:of:H:k:gminus3}
and~\eqref{eq:Vol:Gamma:k:upper:bound}
obtained in
Theorem~\ref{th:bounds:for:Vol:Gamma:k:g}
for $V_{m,k}(g)$ and our study of
the normalized weighted harmonic sums
$\ContributionH_{n,m,\alpha}$ performed in the previous
sections to deduce properties of the probability
distribution $\ProbaCylsOne_{g,m}$. We now state and prove a
lemma that we will use in the proof of Theorem~\ref{thm:generating:series:vol:1}.

\begin{Proposition}
\label{prop:1:2:versus:9:8}
Let $m$ be in $\N\cup\{+\infty\}$.
For any $t\in\C$ satisfying $|t|<2$
we have the following estimates
as $n \to +\infty$
\begin{align}
\label{eq:tail:parameter:change}
\sum_{k = \lceil 22 \cdot \log(n) \rceil+1}^{n}
\ContributionH_{n,m,9/8}(k) |t|^k
&=
\left|
\sum_{k=1}^n
 \ContributionH_{n,m,1/2}(k) t^k\right|
\cdot o\left(n^{-1}\right)\!\!\,,
\\
\label{eq:tail:with:t}
\sum_{k = \lceil 22 \cdot \log(n) \rceil+1}^{n}
\ContributionH_{n,m,1/2}(k) |t|^k
&=
\left|
\sum_{k=1}^n
 \ContributionH_{n,m,1/2}(k) t^k\right|
\cdot o\left(n^{-1}\right)\!\!\,,
\end{align}
where the error term is uniform over $t$ varying in compact
subsets of the complex disk $|t|<2$.
\end{Proposition}
\begin{proof}
It follows from
definition~\eqref{eq:multiple:harmonic:sum:def} of the
weighted multi-variate harmonic sum
$\ContributionH_{n,m,\alpha}(k)$ that
for any $n,m,k$ we have
$\ContributionH_{n,m,1/2}(k)
\le \ContributionH_{n,m,9/8}(k)$.
Thus, estimate~\eqref{eq:tail:parameter:change}
implies estimate~\eqref{eq:tail:with:t} and
it is sufficient to prove~\eqref{eq:tail:parameter:change}.

We consider separately the cases $|t| \leq 1/e$ and
$1/e\le |t|<2$. We start with the case $|t| \leq 1/e$. Using the fact
that we have a generating series of a probability
distribution, and that for $|t|\in[0,1/e]$ and positive $k$
the power $|t|^k$ is bounded
from above by $e^{-k}$, we get the following estimate valid
for any $|t|\in[0,1/e]$ and any $m\in\N\cup\{+\infty\}$:
\begin{multline}
\label{eq:estimate:t:small}
\frac{1}{W_{n,m,9/8}}
\sum_{k = \lceil 22 \cdot \log(n) \rceil+1}^{n}
\ContributionH_{n,m,9/8}(k)\ |t|^k
\le\\
\le\frac{1}{W_{n,m,9/8}}
\sum_{k = \lceil 22 \log n\rceil+1}^n
\ContributionH_{n,m,9/8}(k)\ |t|^{1+22\log n}
\le\\
\le \left(\frac{1}{W_{n,m,9/8}}
\sum_{k=1}^n \ContributionH_{n,m,9/8}(k)\right)\cdot
|t|\cdot e^{-22\log n}
= |t|\cdot n^{-22}\,,
\end{multline}
where $W_{n,m,9/8}$ is the sum over $k$ of
$\ContributionH_{n,m,9/8}(k)$ as defined
in~\eqref{eq:W:n:m:alpha}. On the other hand, using the
identity $z\Gamma(z)=\Gamma(1+z)$ and applying
Equation~\eqref{eq:multi:harmonic:sum:total:weight} from
Theorem~\ref{thm:multi:harmonic:sum:total:weight} for
respectively $\alpha = 1/2$ and $\alpha=9/8$ we have as $n
\to +\infty$
\begin{align}
\label{eq:H:1:2:t}
\sum_{k=1}^n \ContributionH_{n,m,1/2}(k)\ t^k &=
t\cdot n^{-t/2}\cdot
\frac{\left( \frac{2m}{m+1} \right)^{t/2}}{2\Gamma(1+t/2)}
\left(1 + O\left(\frac{1}{n}\right)\right)\,,
\\
\notag
\sum_{k=1}^n \ContributionH_{n,m,9/8}(k)\ t^k &=
t\cdot n^{t/8}\cdot
\frac{9\left(\frac{2m}{m+1} \right)^{9t/8}}
{8\Gamma(1+9t/8)}
\left(1 + O\left(\frac{1}{n}\right)\right)\,,
\end{align}
where the error terms are uniform over the compact
complex disk $|t|\le 2$.
In particular, for $\alpha=9/8$ and $t=1$ we get
\[
W_{n,m,9/8} \sim \frac{\left(\frac{2m}{m+1}\right)^{9/8} n^{1/8}}{\Gamma(9/8)}.
\]
The latter asymptotic equivalence
combined with~\eqref{eq:estimate:t:small}
imply that uniformly for $t\in[0,1/e]$
we have:
\begin{equation}
\label{eq:o:for:t:le:1:e}
\sum_{k = \lceil 22 \cdot \log(n) \rceil+1}^{n}
\ContributionH_{n,m,9/8}(k)\ |t|^k
=O(n^{1/8})\cdot |t|\cdot n^{-22}\,.
\end{equation}
Recall that $1/\Gamma(1+z)$ is an entire function having zeroes
at negative integers and at no other points. Thus,
for all $m\in\N\cup\{\infty\}$ and
for all $t$ satisfying $|t| \le 1/e$ we have
$$
\min_{|t|\le\frac{1}{e}}
\left|\frac{\left( \frac{2m}{m+1} \right)^{t/2}}{2\Gamma(1+t/2)}\right|
=C>0\,.
$$
Together with~\eqref{eq:H:1:2:t} the latter bound implies that
for $|t|\le 1/e$ and for sufficiently large $n$ we have
$$
\left|\sum_{k=1}^n \ContributionH_{n,m,1/2}(k)\ t^k\right|
\ge 2C\cdot |t|\cdot n^{|t|/2}
\ge 2C\cdot |t|\cdot n^{1/(2e)}\,.
$$
The latter inequality combined with asymptotic
estimate~\eqref{eq:o:for:t:le:1:e} implies the desired
relation~\eqref{eq:tail:parameter:change}
for $|t|\le 1/e$.

We now prove~\eqref{eq:tail:parameter:change} for $|t|\ge
1/e$. In this regime the proof is based on
Corollary~\ref{cor:multi:harmonic:asymptotic:all:k}.

Choose any real parameter $R$ satisfying $1/e<R<2$. From
now on we assume that the complex variable $t$ belongs to
the closed annulus $1/e\le |t| \le R$. All the estimates
below are uniform for $t\in[1/e,R]$, but the constants
might depend on the choice of $R$.

We start with several preparatory remarks. We consider the
function
$$
f(y)=\frac{1}{2}(y \log y - y + 1)\,,
$$
where $y>0$. We note that the function $f$ is strictly
convex with a minimum at $y=1$, where $f(1)=0$. We will
need the following inequalities for $f(44/9)$:
\begin{equation}
\label{eq:f:of:10}
f(44/9) > 1\,,
\end{equation}
and
\begin{equation}
\label{eq:f:for:1:e}
\max_{1/e\le|t|\le 2}
\left(\frac{13}{8}|t|
-\frac{9}{4} f(44/9)\cdot |t| \right)
=
\left(\frac{13}{8}|t|
-\frac{9}{4} f(44/9)\cdot |t| \right)\!\Bigg\vert_{|t|=1/e}
 < -1\,.
\end{equation}
We denote
$
\lambda_{m,n}
= \frac{\log\left(\frac{2m}{m+1} \cdot n \right)}{2}
$.
For $n\ge 3$ and any $m\in\N\cup\{\infty\}$ we have
\begin{equation}
\label{eq:log:2:lambda:log}
\frac{\log(n)}{2}
\le \frac{\log\left(\frac{2m}{m+1} \cdot n \right)}{2}
< \log n\,,
\end{equation}
which implies, in particular, that for real positive $y$
we have
\begin{equation}
\label{eq:o:n:power:minus:1}
e^{- \lambda_{m,n} (y \log y - y + 1)}
\le e^{- \tfrac{\log n}{2} (y \log y - y + 1)}
=n^{-f(y)}\,.
\end{equation}
The next remark is particularly important for the proof. It
follows directly from
definition~\eqref{eq:multiple:harmonic:sum:def} of the
weighted multi-variate harmonic sum
$\ContributionH_{n,m,\alpha}(k)$ that
\begin{equation}
\label{eq:multi:variate:harmonic:rescale}
\ContributionH_{n,m,\alpha}(k)\, t^k
= \ContributionH_{n,m,\alpha t}(k)\,.
\end{equation}
We also get
$$
W_{n,m,\alpha t}
=\sum_{k=1}^n\ContributionH_{n,m,\alpha t}(k)
=\sum_{k=1}^n\ContributionH_{n,m,\alpha}(k)\, t^k\,.
$$
Using this remark we can rewrite the asymptotic
estimate~\eqref{eq:tail:parameter:change}
(which we aim to prove for $|t|\in[1/e,R]$)
in the following equivalent form:
\begin{equation}
\label{eq:H:9:8:t:W:1:2:t}
\frac{1}{|W_{n,m,t/2}|}
\sum_{k = \lceil 22 \log(n) \rceil+1}^{n}
\ContributionH_{n,m,|9t/8|}(k) \stackrel{?}{=} o(n^{-1})\,.
\end{equation}

Now everything is ready for the proof of
Proposition~\ref{prop:1:2:versus:9:8} for $|t|\in[1/e,R]$.
By Theorem~\ref{thm:multi:harmonic:sum:total:weight}
for $\alpha=1/2$ and $\alpha=9/8$
we have uniformly in $t$ in the annulus $1/e\le|t|\le R$
\begin{equation}
\label{eq:W:with:t}
W_{n,m,\alpha t}
=\frac{\alpha t\left( \frac{2m}{m+1} \right)^{\alpha t}
n^{\alpha t - 1}}{\Gamma(1+\alpha t)}
\left(1 + O\left(\frac{1}{n}\right)\right)
=O\left(n^{\alpha t-1}\right)
\quad\text{as}\ n \to +\infty\,.
\end{equation}
Recall definition~\eqref{eq:q:n:m:alpha} of
$q_{n,m,\alpha}(k)$ and apply estimate~\eqref{eq:LD2:q}
from Corollary~\eqref{cor:multi:harmonic:asymptotic:all:k},
where we let $\alpha=9|t|/8$. Under such choice of
$\alpha$, the variable $\lambda_n$, present
in~\eqref{eq:LD2:q}, takes the following value:
$\lambda_n = \frac{9|t|}{8} \log\left(\frac{2m}{m+1} n\right)=
(9|t|/4)\lambda_{m,n}$.
For any $y>1$ we have
\begin{multline}
\label{eq:tail:q:9:8:t}
\frac{1}{W_{n,m,9|t|/8}}
\sum_{k=\lceil y\lambda_{n}\rceil+1}^{n}
\ContributionH_{n,m,9|t|/8}(k)
=\sum_{k=\lceil y\cdot (9|t|/4)\cdot\lambda_{m,n}\rceil+1}^{n}
q_{n,m,9|t|/8}(k)
=\\
=\frac{e^{-(9|t|/4)\cdot\lambda_{m,n} (y \log y - y + 1)}}
{\sqrt{2\pi\cdot y\cdot (9|t|/4)\lambda_{m,n}}}
\frac{y}{(y - 1)}
\cdot O(1)
= o\left(n^{-\tfrac{9|t|}{4}\cdot f(y)} \right)\,,
\end{multline}
where we used~\eqref{eq:o:n:power:minus:1} for the
rightmost equality.

Note that $|t|\le R< 2$. This implies that
$$
\min_{1/e\le|t|\le R}
\left|\frac{t\left( \frac{2m}{m+1} \right)^{\alpha t}}
{2\Gamma(1+t/2)}\right|=C'(R)>0\,.
$$
This observation combined with~\eqref{eq:W:with:t}
imply that
$$
\frac{W_{n,m,9|t|/8}}{|W_{n,m,t/2}|}
=O\left(n^{\tfrac{9|t|}{8}+\tfrac{|t|}{2}}\right)
=O\left(n^{\tfrac{13|t|}{8}}\right)\,.
$$
uniformly in $1/e\le|t|\le R$.
Combining the latter estimate
with~\eqref{eq:tail:q:9:8:t} we obtain
\begin{multline*}
\frac{1}{|W_{n,m,t/2}|}
\sum_{k=\lceil y\lambda_{n}\rceil+1}^{n}
\ContributionH_{n,m,9|t|/8}(k)
=\\
=\frac{W_{n,m,9|t|/8}}{|W_{n,m,t/2}|}
\cdot
\frac{1}{W_{n,m,9|t|/8}}
\sum_{k=\lceil y\lambda_{n}\rceil+1}^{n}
\ContributionH_{n,m,9|t|/8}(k)
=O\!\left(n^{\tfrac{13|t|}{8}}\right)
o\!\left(n^{-\tfrac{9|t|}{4}\cdot f(y)}\right).
\end{multline*}
We now choose $y=44/9$. Applying~\eqref{eq:f:for:1:e}, we
conclude that for $1/e\le t\le R$ we have uniformly in $t$
\begin{equation}
\label{eq:larger:sum:9:8}
\frac{1}{|W_{n,m,t/2}|}
\sum_{k=\left\lceil\tfrac{44}{9}\lambda_{n}\right\rceil+1}^{n}
\ContributionH_{n,m,9|t|/8}(k)
=o(n^{-1})\,.
\end{equation}
It remains to note that
for $|t|\le R<2$ and for
$n\ge 3$ we have
$$
\frac{44}{9}\lambda_n
=\frac{44}{9}\cdot\frac{9}{4}\cdot|t| \cdot\frac{\log\left(\frac{2m}{m+1} \cdot n \right)}{2}
< 11\cdot |t|\cdot\log n
< 22\cdot\log n\,.
$$
This implies that the sum on the left-hand side
of~\eqref{eq:H:9:8:t:W:1:2:t} is contained in the sum on
the left-hand side of~\eqref{eq:larger:sum:9:8}. Thus,
\eqref{eq:larger:sum:9:8}
implies~\eqref{eq:H:9:8:t:W:1:2:t} and, hence, it implies
the equivalent estimate~\eqref{eq:tail:parameter:change} in
the regime $1/e\le |t|\le R<2$.
\end{proof}

Now everything is ready to prove Theorem~\ref{thm:generating:series:vol:1}.

\begin{proof}[Proof of Theorem~\ref{thm:generating:series:vol:1}]
The main ingredients on the proof are the asymptotic
equivalence~\eqref{eq:bounds:in:terms:of:H:k:gminus3}
and the upper bound~\eqref{eq:Vol:Gamma:k:upper:bound}
from Theorem~\ref{th:bounds:for:Vol:Gamma:k:g} combined with
Proposition~\ref{prop:1:2:versus:9:8}.
We use abbreviation~\eqref{eq:V:m:k}.
We split the sum~\eqref{eq:generating:series:vol:1} into two
parts $\sum_{k=1}^g V_{m,k}(g) \cdot t^k=\Sigma_1+\Sigma_2$,
where
$$
\Sigma_1=\sum_{k=1}^{\lceil 22 \cdot \log(3g-3) \rceil}
V_{m,k}(g) \cdot t^k\,,
\qquad\qquad
\Sigma_2=\sum_{k=\lceil 22 \cdot \log(3g-3) \rceil+1}^g
V_{m,k}(g) \cdot t^k
$$
and evaluate the two sums separately.
Using~\eqref{eq:bounds:in:terms:of:H:k:gminus3}
from Theorem~\ref{th:bounds:for:Vol:Gamma:k:g}
we get
\begin{multline}
\label{eq:head:contribution}
\Sigma_1=\frac{2\sqrt{2}}{\sqrt{\pi}} \cdot \sqrt{3g-3}
\left(\frac{8}{3}\right)^{4g-4}
\cdot\\
\cdot\left(
\sum_{k=1}^{\lceil 22 \cdot \log(3g-3) \rceil}
\ContributionH_{3g-3,m,1/2}(k) t^k \right)
\left(1 + O\left( \frac{(\log g)^2}{g}\right) \right)\,.
\end{multline}
Applying~\eqref{eq:tail:with:t}
from Proposition~\ref{prop:1:2:versus:9:8}
combined with~\eqref{eq:multi:harmonic:sum:total:weight}
from Theorem~\ref{thm:multi:harmonic:sum:total:weight},
where we let  $\alpha=1/2$ and $n=3g-3$, we get
\begin{multline}
\label{eq:tmp:tmp}
\sum_{k=1}^{\lceil 22 \cdot \log(3g-3) \rceil}
\ContributionH_{3g-3,m,1/2}(k) t^k
=\\
=\sum_{k=1}^{3g-3}
\ContributionH_{3g-3,m,1/2}(k) t^k
-
\sum_{k=\lceil 22 \cdot \log(3g-3) \rceil+1}^{3g-3}
\ContributionH_{3g-3,m,1/2}(k) t^k
=\\
=\left(
\sum_{k=1}^{3g-3}
\ContributionH_{3g-3,m,1/2}(k) t^k \right)
\left(1 - O\left(g^{-1}\right) \right)
=\\
=\frac{\left( \frac{2m}{m+1} \right)^{t/2}
(3g-3)^{t/2 - 1}}{\Gamma(t/2)}
\cdot\left(1 + O\left(g^{-1}\right) \right)
\,.
\end{multline}
Plugging the latter expression in~\eqref{eq:head:contribution}
we get
\begin{multline}
\label{eq:nose}
\Sigma_1=\sum_{k=1}^{\lceil 22 \cdot \log(3g-3) \rceil}
V_{m,k}(g) \cdot t^k
=\\=
\frac{2\sqrt{2} \left(\frac{2m}{m+1}\right)^{t/2}}{\sqrt{\pi} \cdot \Gamma(\frac{t}{2})}
\cdot (3g-3)^{\frac{t-1}{2}} \cdot \left( \frac{8}{3} \right)^{4g-4}
\left(1 + O\left( \frac{(\log g)^2}{g}\right) \right)\,,
\end{multline}
where for every compact subset $U$
of the complex disk $|t|<2$
the error term is uniform
over $m \in \N \cup \{+\infty\}$ and $t\in U$.
Note that the expression on the right-hand side of
the latter equation coincides with the right-hand side
of~\eqref{eq:generating:series:vol:1}
from Theorem~\ref{thm:generating:series:vol:1}.

Using~\eqref{eq:Vol:Gamma:k:upper:bound},
from Theorem~\ref{th:bounds:for:Vol:Gamma:k:g},
we get the following bound for the second sum:
\begin{equation}
\label{eq:tail:contribution}
|\Sigma_2|
\le C_2\cdot \sqrt{g}
\cdot\left(\frac{8}{3}\right)^{4g-4}
\sum_{k=\lceil 22 \cdot \log(3g-3) \rceil+1}^g
\ContributionH_{3g-3,m,9/8}(k)
\cdot |t|^k\,.
\end{equation}
Combining~\eqref{eq:tail:parameter:change} from
Proposition~\ref{prop:1:2:versus:9:8}
with~\eqref{eq:tail:contribution} and comparing the resulting
expressions for $\Sigma_1$
from~\eqref{eq:nose}
and for $|\Sigma_2|$ from~\eqref{eq:tail:contribution} we conclude that
$\Sigma_2=\Sigma_1\cdot o\left(g^{-1}\right)$ uniformly
over $m\in\N\cup\{\infty\}$ and over $t$ in any compact
subset $U$
of the complex disk $|t|<2$.
This completes the proof of
Theorem~\ref{thm:generating:series:vol:1}.
\end{proof}

We show now that Theorem~\ref{thm:generating:series:vol:1}
implies the following result.

\begin{Corollary}
\label{cor:mod:poisson:for:vol:1}
For any $m\in\N\cup\{+\infty\}$
the family of probability distributions
$\{\ProbaCylsOne_{g,m}\}_{g \geq 2}$
defined in~\eqref{eq:ProbaCylsOne}
converges mod-Poisson with radius $R=2$, parameters
$\lambda_{3g-3}
= \tfrac{\log\left( \tfrac{2m}{m+1} \cdot (3g-3)\right)}{2}$,
limiting function
$\frac{\Gamma\left(\frac{1}{2}\right)}{\Gamma \left( \frac{x}{2}\right)}$
and speed $O \left( \frac{(\log g)^2}{g}
\right)$.
\end{Corollary}
\begin{proof}
Let
$$
\Phi_g(t) = \sum_{k=1}^{g} V_{m,k}(g) t^k\,.
$$
The above sum coincides with expression~\eqref{eq:generating:series:vol:1}
from Theorem~\eqref{thm:generating:series:vol:1}.
By definition~\eqref{eq:generating:series:F}, the generating series
$F(t)$ of the probability
distribution $\ProbaCylsOne_{g,m}$
is $\Phi_g(t) / \Phi_g(1)$. Applying
Equation~\eqref{eq:generating:series:vol:1} we get
$$
F(t)=\frac{\Phi_g(t)}{\Phi_g(1)}
=
\left(\frac{2m}{m+1} \cdot (3g-3) \right)^{\tfrac{t-1}{2}}
\frac{\Gamma(1/2)}{\Gamma(t/2)} \cdot
\left(1 + O\left( \frac{(\log g)^2}{g}\right) \right)\,.
$$
We conclude that the generating series satisfies all conditions
of Definition~\ref{def:mod:poisson:convergence}
of mod-Posson convergence,
with parameters $\lambda_{3g-3}=\log(\frac{2m}{m+1} \cdot (3g-3))/2$,
limiting function $\frac{\Gamma(1/2)}{\Gamma(t/2)}$,
radius $R=2$ and speed $\frac{(\log g)^2}{g}$.
\end{proof}

We complete this Section with a proof of the assertion
stated in Section~\ref{s:intro} claiming that \textit{the
probability distribution $q_{3g-3, \infty, 1/2}$
well-approximates the probability distribution
$\ProbaGamma_g(k)$}. We admit that we would not use this
statement in this particular form, so we provide this
justification just for the sake of completeness.

Consider the asymptotic relation~\eqref{eq:tail:with:t}
from Proposition~\ref{prop:1:2:versus:9:8}
in which we let $t=1$, $n=3g-3$ and $m=+\infty$. We get
$$
\sum_{k = \lceil 10 \log(n) \rceil+1}^{3g-3}
\ContributionH_{3g-3,\infty,1/2}(k)
=
\left(
\sum_{k=1}^{3g-3}
 \ContributionH_{3g-3,\infty,1/2}(k)\right)
\cdot o\left(n^{-1}\right)\!\!\,.
$$
The latter relation
combined with~\eqref{eq:V:tilde:lower}
and~\eqref{eq:V:tilde:upper}
imply the following asymptotic relations for the probability
distribution $\ProbaGamma_g$ defined in~\eqref{eq:def:q_g}.
For $k\in\N$, $k^2\le g$, we have
$$
\ProbaGamma_g(k)=
q_{3g-3,\infty,\frac{1}{2}}(k)
\cdot
\left(1 + O\left( \frac{k^2}{g}\right) \right)\,.
$$
The following asymptotic bound is valid as $g\to+\infty$:
$$
\sum_{k=\lceil 10\log g\rceil+1}^g
\ProbaGamma_g(k)
=O\left(\frac{(\log g)^3}{g}\right)\,.
$$

Analogous considerations imply that
for $k\in\N$, $800 k^2\le g$, we have
\begin{equation}
\label{eq:ProbaCylsOne:k}
\ProbaCylsOne_{g,\infty}(k)
=q_{3g-3,\infty,\frac{1}{2}}(k)
\cdot
\left(1 + O\left( \frac{(k+2\log g)^2}{g}\right) \right)\,,
\end{equation}
and
$$
\sum_{k=\lceil 10\log g\rceil+1}^g
\ProbaCylsOne_{g,\infty}(k)
=O\left(\frac{(\log g)^3}{g}\right)\,.
$$

\section{Contribution of separating multicurves}
\label{s:disconnecting:multicurves}
In Section~\ref{s:sum:over:single:vertex:graphs} we studied
the volume contributions $\Vol(\Gamma_E(g))$ of stable
graphs $\Gamma_E(g)$ with a single vertex and with $E$ loops. In particular,
Theorem~\ref{thm:generating:series:vol:1} provides precise
asymptotics for the generating series $\sum_{E \geq 1}
\Vol \Gamma_E(g) t^E$ as $g \to +\infty$. In this section we
study the volume contribution of the remaining stable graphs.

In Section~\ref{ss:generic:tail:bound} we provide some simple estimates of
tails of certain series related to Poisson distribution. In
Sections~\ref{ss:V:2} and~\ref{ss:V:ge:3} we prove necessary minor refinements
of estimates from~\cite{Aggarwal:intersection:numbers} to bound
respectively the volume contributions of stable graphs with 2 vertices
and for the volume contribution coming from stable graphs with at least
3 vertices. We emphasize that the main asymptotic analysis of volume
contributions of various stable graphs was already
performed by A.~Aggarwal
in~\cite{Aggarwal:intersection:numbers}. Our presentation
in Sections~\ref{ss:V:2} and~\ref{ss:V:ge:3} closely
follows original presentation in
respectively Sections~9 and~10
of~\cite{Aggarwal:intersection:numbers}, where
we perform more or less straightforward
adjustment of the original bounds for the sums to
bounds for the associated generating series
which we need in the context of the current paper.

Following A.~Aggarwal let us
introduce the following notations for the contributions of
stable graphs with a given number of vertices.

\begin{Definition}
\label{def:volume:contributions}
Let $g$ be an integer satisfying $g \geq 2$. Given a stable
graph $\Gamma \in \cG_g$ we denote by $V(\Gamma)$ and
$E(\Gamma)$ respectively the set of vertices and the set of
edges of $\Gamma$. We denote by $|V(\Gamma)|$ and $|E(\Gamma)|$
the cardinalities of these sets. We define
\begin{equation}
\Upsilon_g^{(V)} :=
\sum_{\substack{\Gamma \in \cG_g\\|V(\Gamma)| = V}} \Vol(\Gamma)\,;
\qquad
\Upsilon_g^{(V; E)} :=
\sum_{\substack{\Gamma \in \cG_g\\|V(\Gamma)|
= V\\|E(\Gamma)| = E}} \Vol(\Gamma)\,,
\end{equation}
where $\Vol(\Gamma)$ is the contribution of the stable
graph $\Gamma$ to the Masur-Veech volume $\Vol \cQ_g$ as
given in~\eqref{eq:volume:contribution:of:stable:graph}.
\end{Definition}
Note that by~\eqref{eq:square:tiled:volume} we have
\[
\Vol \cQ_g = \sum_{V=1}^{2g-2} \Upsilon_g^{(V)}
\quad \text{and} \quad
\Upsilon_g^{(V)} = \sum_{E = 1}^{3g-3} \Upsilon_g^{(V; E)}\,.
\]
We also have $\Upsilon_g^{(1; E)} = \Vol \Gamma_E(g)$.

The following propositions are
refinements of Propositions~8.4 and~8.5 respectively from
the original paper~\cite{Aggarwal:intersection:numbers} of
A.~Aggarwal.
\begin{Proposition}
\label{prop:up:bound:V2}
There exists constants $B_2$ and $g_2$
such that for any couple $g, t$, satisfying
$g\in\N$, $g\ge g_2$, and
$0\le t\le \frac{44}{19}$,
we have
\begin{equation}
\label{eq:up:bound:V2}
\left(\frac{8}{3}\right)^{-4g}
\cdot
\sum_{E = 1}^{3g-3} \Upsilon_g^{(2; E)} t^E
\le B_2 \cdot
t \cdot
(\log g)^{14} \cdot
g^{\frac{2t - 3}{2}}\,.
\end{equation}
\end{Proposition}
We prove Proposition~\ref{prop:up:bound:V2}
in Section~\ref{ss:V:2}.

\begin{Proposition}
\label{prop:up:bound:V3}
There exists constants $B_3$ and $g_3$
such that for any triple $V,g, t$, satisfying
$V\in\N$, $V\ge 3$, $g\in\N$, $g\ge g_3$, and
$0 \le t\le \frac{44}{19}$,
we have
\begin{equation}
\label{eq:up:bound:V3}
\left(\frac{8}{3}\right)^{-4g} \cdot
\sum_{V = 3}^{2g-2} \sum_{E=1}^{3g-3}
\Upsilon_g^{(V; E)} t^E
\le B_3\cdot
t \cdot
(\log g)^{24} \cdot
g^{\tfrac{9 t - 10}{4}}\,.
\end{equation}
\end{Proposition}
We prove Proposition~\ref{prop:up:bound:V3}
in Section~\ref{ss:V:ge:3}.

\subsection{Bound for tail contribution to moments of
Poisson distribution}
\label{ss:generic:tail:bound}
Recall that given a non-negative integer $n$ and a positive
real number $\lambda$, the $n$-th moment $P_n(\lambda)$ of
a random variable corresponding to the Poisson distribution
$\Poisson_\lambda$ with parameter $\lambda$ is defined as
\begin{equation*}
P_n(\lambda) := e^{-\lambda} \cdot
\sum_{k = 0}^{+\infty} k^n \frac{\lambda^k}{k!}.
\end{equation*}
In the next two sections we will use several times the
following upper bound for the tail of the
above expression.

\begin{Lemma}
\label{lem:exp:tail:bound}
Let $\lambda$ and $x$ be strictly positive real numbers and
let $n$ be an integer satisfying $n \geq 0$.
Then
\begin{equation}
\label{eq:tail:upper:bound}
\sum_{k = \lceil x \lambda \rceil}^{+\infty}
\frac{k^n \cdot \lambda^k}{k!} \leq
P_n(x \lambda) \cdot \exp\big(-\lambda (x \log x - x)\big)\,.
\end{equation}
\end{Lemma}
We are interested in the case where $x$ is fixed and
$\lambda$ tends to infinity. Note that the term $x \log x
- x$ is positive for $x > e = 2.712\ldots$, so
for $x>e$ we prove an exponential decay in
$\lambda$ of the expression in the left hand side
of~\eqref{eq:tail:upper:bound}.
\begin{proof}[Proof of Lemma~\ref{lem:exp:tail:bound}]
Let $\theta \geq 0$. Then
for $k \geq \lceil \lambda x \rceil$ we have
$\exp (\theta (k - \lambda x)) \geq 1$. Hence
\[
\sum_{k = \lceil x \lambda \rceil}^{+\infty} \frac{k^n \cdot \lambda^k}{k!}
\leq
\sum_{k = 0}^{+\infty} \frac{\exp(\theta(k - \lambda x)) \cdot k^n \cdot \lambda^k}{k!}
=
e^{-\theta \lambda x} \sum_{k=0}^{+\infty} \frac{k^n \cdot (e^\theta \cdot \lambda)^k}{k!}\,.
\]
We get a sum as in~\eqref{eq:def:poisson:moments}.
By Lemma~\ref{lem:poisson:moments} we have
\[
\sum_{k \geq 0} \frac{k^n \cdot (e^\theta \lambda)^k}{k!}
= \exp(e^\theta \lambda) \cdot P_n(e^\theta \lambda).
\]
The tail bound is obtained by taking $\theta = \log x$.
\end{proof}

We note that analogous Lemma~2.4 of A.~Aggarwal~\cite{Aggarwal:intersection:numbers}
provides a similar
upper bound for the case $n=0$ given
as
\begin{equation}
\label{eq:Lemma:2::Amol}
\sum_{k = \lfloor (1 + 2\delta) R \rfloor}^{+\infty}
\frac{R^k}{k!}
\leq
\exp \left(- R \left( \delta \log(1+\delta) + \frac{\log(\delta)}{R} - 1\right) \right)\,.
\end{equation}
We will need a slightly stronger estimate.
Bound~\eqref{eq:Lemma:2::Amol} above
provides exponential decay as long as
$\log(1 + \delta) > 1/\delta$ which corresponds to $\delta > 1.23997\ldots$
In comparison, bound~\eqref{eq:tail:upper:bound} reads as
\[
\sum_{k = \lceil (1 + 2\delta) R \rceil}^{+\infty}
\frac{R^k}{k!}
\leq
\exp \big(-R ((1+2\delta) \log (1 + 2 \delta) - (1 + 2 \delta) \big)\,,
\]
which provides exponential decay as long as $1 + 2 \delta > e$ which corresponds
to $\delta > 0.8591409$.


  %

\subsection{Volume contribution of stable graphs with $2$ vertices}
\label{ss:V:2}
Following the original approach of
A.~Aggarwal~\cite{Aggarwal:intersection:numbers}, we
consider the following refinement of the quantity
$\Upsilon_g^{(V;E)}$ introduced in
Definition~\ref{def:volume:contributions}.

Given a stable graph $\Gamma \in \cG_g$ denote by
$S(\Gamma)$ the number of self-edges of $\Gamma$ (edges
with their two ends on the same vertex). Denote by
$T(\Gamma)$ the set of simple edges of $\Gamma$ (edges
with their two ends at distinct vertices). The set
$E(\Gamma)$ decomposes as the disjoint union $E(\Gamma) =
S(\Gamma) \sqcup T(\Gamma)$.
Following~\cite[Definition~8.6]{Aggarwal:intersection:numbers}
let
$$
\Upsilon_g^{(V; S, T)} :=
\sum_{\substack{\Gamma \in \cG_g\\
|V(\Gamma)| = V\\
|S(\Gamma)| = S\\|T(\Gamma)| = T}}
\Vol(\Gamma)\,.
$$
The number $\Upsilon^{(V; S, T)}_g$ is non-zero only when
the following three conditions are simultaneously
satisfied: $V-1 \leq T$ (connectedness of the graph), $S +
T - V + 1 \leq g$ (bound on the genus) and $V \leq 2g-2$
(``stability'' of the graph).
In particular, the number $\Upsilon^{(V; S, T)}_g$
is non-zero only when
the following bounds are simultaneously
satisfied: $0\le S\le g$ and $V-1\le T\le 3g-3$.

Following~\cite[Lemma 9.5]{Aggarwal:intersection:numbers}
we split the stable graphs with $2$ vertices into three
groups corresponding to the following ranges of parameters
$S$ and $T$. We have $S\ge g-1$ for the first collection of
stable graphs. We have $T>13$ and $S\le g-2$ for the second
collection. We have $T\le 13$ and $S\le g-2$ for the third
collection. Lemmas~\ref{lm:V:2:large:S},
\ref{lm:V:2:large:T}, \ref{lm:V:2:small:T} provide upper
bounds for the respective contributions to the
sum~\eqref{eq:up:bound:V2} in
Proposition~\ref{prop:up:bound:V2}. As we already
mentioned, our proofs of Lemmas~\ref{lm:V:2:large:S},
\ref{lm:V:2:large:T}, \ref{lm:V:2:small:T} are obtained by
elementary adjustment of bounds elaborated by A.~Aggarwal
in~\cite[Section~9]{Aggarwal:intersection:numbers}.

\begin{Lemma}
\label{lm:V:2:large:S}
For any non-negative real number $t$ and for any
integer $g$ satisfying $g\ge 2$ the following bound
is valid
\begin{equation}
\label{eq:V:2:large:S}
\left(\frac{8}{3}\right)^{-4g}
\sum_{\substack{T\ge 1\\S\ge g-1 }}
\Upsilon^{(2;S,T)}_g
\cdot t^{S+T}
\le 2^{11}\cdot \big(t^g+t^{g+1}\big)
\cdot g^{3/2}
\cdot\left(\frac{9}{8}\right)^{g}
\frac{(\log g+7)^g}{(g-1)!}
\,.
\end{equation}
\end{Lemma}

\begin{proof}
All possible stable graphs with 2 vertices and with $S \ge
g-1$ split into the following three types:
\begin{enumerate}
\item[(I)] $1+[(g-1)/2]$ stable graphs with $S=g-1$, $T=2$
and genera (decorations) $g_1=g_2=0$ at the two vertices;
\item[(II)] $g-1$ graphs with $S=g-1$, $T=1$ and $g_1=1$,
$g_2=0$;
\item[(III)] $1+[(g-1)/2]$ graphs with $S=g$, $T=1$ and $g_1=g_2=0$.
\end{enumerate}
By Equation~(9.14) in~\cite{Aggarwal:intersection:numbers}
for any of these graphs $\Gamma$ we have
\begin{equation}
\label{eq:Z:of:Gamma}
\left(\frac{8}{3}\right)^{-4g}\Vol(\Gamma)\le
2^{10}(S^2+T^2)g^{-3/2}\xi_1\xi_2
\frac{(\log g+7)^{S+T-1}}{2^S S! T!}\,.
\end{equation}
Here $\xi_i$, $i=1,2$, are defined in
Equation~(9.1) in~\cite{Aggarwal:intersection:numbers}
as
$$
\xi_i=\max_{\boldsymbol{d}\in\Pi(3g_i+2s_i+T+3,2s_i+T)}
\big(1+\epsilon(\boldsymbol{d})\big)\,,
$$
where $g_i$ and $s_i$ are respectively the genera and the
number of self edges at the $i$-th vertex, for $i=1,2$
and $\epsilon(\boldsymbol{d})$ is defined in~\eqref{eq:ansatz}.
The bound~\eqref{eq:Aggarwal:Prop:1:2} from
Theorem~\ref{th:correlators:upper:bound} of A.~Aggarwal
(see Proposition~1.2 in the original
paper~\cite{Aggarwal:intersection:numbers})
implies that for the stable graphs with $V=2$ we have
$$
\xi_1\xi_2
\le\left(\frac{3}{2}\right)^{2E-2}
\le\left(\frac{3}{2}\right)^{2g}
$$
and~\eqref{eq:Z:of:Gamma} provides the following bound for
any stable graph with $V=2$ and $S\ge g-1$:
\begin{equation}
\label{eq:Z:of:Gamma:bis}
\left(\frac{8}{3}\right)^{-4g}\Vol(\Gamma)
\le
2^{10}(g^2+1)g^{-3/2}
\cdot\left(\frac{3}{2}\right)^{2g}
\frac{(\log g+7)^g}{2^g (g-1)!}
\,.
\end{equation}
We have seen that when $V=2$ and $S\ge g-1$
there are $g-1$ stable graphs of type (II)
for which $S+T=g$ and there are at most $g+1$
stable graphs of types (I) and (III) counted together
for which $S+T=g+1$.
Thus we get
\begin{multline*}
\left(\frac{8}{3}\right)^{-4g}\cdot
\sum_{\substack{T\ge 1\\ S \ge g-1}}
\Upsilon^{(2;S,T)}_g \cdot t^{S + T}
\le\\
\le\Big((g-1)t^g+(g+1)t^{g+1}\Big)\cdot
2^{10}\cdot(g^2+1)\cdot g^{-3/2}
\cdot\left(\frac{3}{2}\right)^{2g}
\frac{(\log g+7)^g}{2^g (g-1)!}
\le\\
\le 2^{11}\cdot\big(t^g+t^{g+1}\big)
\cdot g^{3/2}
\cdot\left(\frac{9}{8}\right)^{g}
\frac{(\log g+7)^g}{(g-1)!}
\,.
\end{multline*}
\end{proof}

\begin{Lemma}
\label{lm:V:2:large:T}
There exists a constant $C_7$ such that
for any non-negative real number $t$ and for any
integer $g$ satisfying $g\ge 2$ we have
\begin{multline}
\label{eq:V2:2}
\left(\frac{8}{3}\right)^{-4g}
\sum_{\substack{14 \leq T \leq 3g-3\\0\leq S \leq g-2}}
\Upsilon^{(2;S,T)}_g
\cdot t^{S+T}
\le\\
\le C_7\cdot t^{14}\cdot \exp(189t/8)
\cdot(\log g +7)^{13}
\cdot g^{(27t - 56)/8}\,.
\end{multline}
\end{Lemma}
\begin{proof}
By Equation~(9.17) from Lemma~9.5
in~\cite{Aggarwal:intersection:numbers}, in the case $T>13$
and $S\le g-2$ we have for $g$ large enough
$$
\left(\frac{8}{3}\right)^{-4g}
\Upsilon^{(2;S,T)}_g
\le 2^{55} g^{-7}\left(\frac{9}{4}\right)^E
\frac{(\log g +7)^{E-1}}{2^S\, S!\, T!}\,.
$$
By the binomial expansion we have
$$
\sum_{\substack{S+T=E\\S,T \geq 0}}
\frac{1}{2^S\, S!\, T!}=
\frac{1}{E!}\left(\frac{1}{2}+1\right)^E=
\frac{1}{E!}\left(\frac{3}{2}\right)^E\,.
$$
Note that the set $\cG_g$ of stable graphs of any fixed
genus $g$ is finite.
Hence, there exists a constant $C'_7$ such that for any
$g\ge 2$ and for any fixed $E$ we have
\begin{multline*}
\left(\frac{8}{3}\right)^{-4g}
\cdot
\sum_{\substack{14 \le T \le 3g-3\\0 \le S \le g-2\\S+T=E}}
\Upsilon^{(2;S,T)}_g
\le\\
\le C'_7\cdot g^{-7}
\cdot(\log g +7)^{-1}
\cdot\left(\frac{9}{4}\right)^E
\cdot(\log g +7)^{E}
\sum_{\substack{S+T=E\\S,T \geq 0}}
\frac{1}{2^S\, S!\, T!}
=\\=
C'_7\cdot g^{-7}
\cdot(\log g +7)^{-1}
\cdot\left(\frac{27}{8}\right)^E
\cdot(\log g +7)^{E}
\cdot\frac{1}{E!}\,.
\end{multline*}
Multiplying each term by $t^E=t^{S+T}$
and taking the sum with respect to
$E$ we obtain
\begin{multline*}
\left(\frac{8}{3}\right)^{-4g}
\cdot
\sum_{\substack{14 \le T \le 3g-3\\0 \le S \le g-2}}
\Upsilon^{(2;S,T)}_g
\cdot t^{S+T}
\le\\
\le
C'_7\cdot g^{-7}
\cdot(\log g +7)^{-1}
\sum_{E=14}^{+\infty}
\left(\frac{27}{8}\right)^E
\cdot(\log g +7)^{E}\cdot t^E
\cdot\frac{1}{E!}
\le\\
\le C'_7\cdot g^{-7}
\cdot(\log g +7)^{-1}
\cdot
\left(
\frac{27t}{8}
\cdot
(\log g +7)
\right)^{14}
\cdot\exp \left(
\frac{27t}{8}
\cdot
(\log g +7)
\right)
=\\=
C_7\cdot t^{14}
\cdot \exp(189t/8)
\cdot(\log g +7)^{13}
\cdot g^{(27t - 56)/8}\,,
\end{multline*}
where we used the inequality
$\sum_{k=n}^{+\infty} \frac{x^k}{k!} \le x^n\exp(x)$,
which is valid for any non-negative $x$,
and where we let $C_7=C'_7\cdot\left(\frac{27}{8}\right)^{14}$.
\end{proof}

\begin{Lemma}
\label{lm:V:2:small:T}
There exists a constant $C_8$ such that
for any non-negative real number $t$
and for any integer $g$ satisfying $g\ge 2$ we have
\begin{multline}
\label{eq:V2:3}
\left(\frac{8}{3}\right)^{-4g}\cdot
\sum_{\substack{1\le T\le 13\\0\le S\le g-2}}
\Upsilon^{(2;S,T)}_g
\cdot t^{S + T}
\le C_8\cdot t (1+t)^{14}
\cdot \exp(63t/4)
\cdot \\
\cdot (\log g+7)^{14} \cdot \Big(g^{(2t - 3)/2} + g^{(9t-28)/4}\Big)\,.
\end{multline}
\end{Lemma}

\begin{proof}
It follows from Equation~(9.18) from Lemma~9.5
from~\cite{Aggarwal:intersection:numbers} that
there exists a constant $C_8$ such that
in the case $T\le
13$ and $S\le g-2$ the following bound is valid
for any integer $g$ satisfying $g\ge 2$:
\begin{equation*}
\left(\frac{8}{3}\right)^{-4g}
\Upsilon^{(2;S,T)}_g
\le C_8\cdot\frac{(\log g+7)^{S+12}}{S!}
\left(g^{-3/2}(S^2+1)+g^{-7}\left(\frac{9}{4}\right)^S\right)\,.
\end{equation*}
Multiplying by $t^E=t^{S+T}$ and taking the sum
over $1 \le T \le 13$ and over $0 \le S \le g-2$ we obtain
the following bound:
\begin{multline}
\label{eq:bound:T:13}
\left(\frac{8}{3}\right)^{-4g}\cdot
\sum_{\substack{1\le T\le 13\\0\le S\le g-2}}
\Upsilon^{(2;S,T)}_g
\cdot t^{S + T}
\le\\
\le C_8 \cdot\Big(t+t^2+\dots+t^{13}\Big)
\cdot(\log g+7)^{12}
\cdot \Big(g^{-3/2} \Sigma_1 + g^{-7} \Sigma_2\Big)
\le\\
\le C_8\cdot t \cdot(1+t)^{12}
\cdot(\log g+7)^{12}
\cdot \Big(g^{-3/2} \Sigma_1 + g^{-7} \Sigma_2\Big)\,,
\end{multline}
where
\begin{align*}
\Sigma_1 &=
\sum_{S=0}^{+\infty} (S^2+1) \frac{(t \cdot (\log g+7))^S}{S!}\,,
\\
\Sigma_2 &=
\sum_{S=0}^{+\infty}
\frac{\big(\frac{9t}{4} (\log g+7)\big)^S}{S!}
=\exp(63t/4)\cdot g^{9t/4}\,.
\end{align*}
The sum $\Sigma_1$ can be decomposed
into two sums of the form~\eqref{eq:def:poisson:moments},
where we take
$n=2$ and $n=0$ respectively
and where we let $\lambda=t \cdot (\log g+7)$.
Applying Lemma~\ref{lem:poisson:moments}
and taking into consideration that
$P_2(\lambda)=\lambda^2+\lambda$
by~\eqref{eq:first:four:P:lambda},
we get
\begin{multline*}
\Sigma_1 =
\sum_{S=0}^{+\infty} (S^2+1)\cdot
\frac{(t \cdot (\log g+7))^S}{S!}
=e^\lambda\cdot\big(P_2(\lambda)+1\big)
=\\
=\exp(7t)\cdot g^t\cdot
\Big(1+t\cdot(\log g+7)+t^2\cdot(\log g+7)^2\Big)
\le\\
\le
\exp(7t)\cdot g^t\cdot
(1+t)^2 (\log g + 7)^2\,.
\end{multline*}
Plugging the resulting bounds for the sums $\Sigma_1$ and
$\Sigma_2$
into~\eqref{eq:bound:T:13}
we obtain the bound
\begin{multline*}
\left(\frac{8}{3}\right)^{-4g}\cdot
\sum_{\substack{1\le T\le 13\\0\le S\le g-2}}
\Upsilon^{(2;S,T)}_g
\cdot t^{S + T}
\le C_8\cdot t \cdot (1+t)^{14}
\cdot(\log g+7)^{14}
\cdot\\
\Big(g^{-3/2}\cdot
\exp(7t)\cdot g^t
+
g^{-7}\cdot
\exp(63t/4)\cdot g^{9t/4}
\Big)
=\\
=C_8\cdot t (1+t)^{14}
\cdot(\log g+7)^{14}\cdot
\\
\Big(\exp(7t) \cdot g^{(2t - 3)/2}
+\exp(63t/4)\cdot g^{(9t-28)/4}
\Big)\,.
\end{multline*}
Now it remains to notice that $7t \le 63t/4$ to get the desired bound.
\end{proof}

\begin{proof}[Proof of Proposition~\ref{prop:up:bound:V2}]
By taking the sum of
the bounds~\eqref{eq:V:2:large:S}, \eqref{lm:V:2:large:T}
and~\eqref{lm:V:2:small:T} from respectively
Lemmas~\ref{lm:V:2:large:S}, \ref{lm:V:2:large:T}
and~\ref{lm:V:2:small:T} covering all possible combinations
of $S$ and $T$ we obtain
\begin{multline}
\label{eq:up:bound:V2:long}
\left(\frac{8}{3}\right)^{-4g}
\cdot
\sum_{E = 1}^{3g-3} \Upsilon_g^{(2; E)} t^E
\le 2^{11}\cdot\big(t^g+t^{g+1}\big)
\cdot g^{3/2}
\cdot\left(\frac{9}{8}\right)^{g}
\frac{(\log g+7)^g}{(g-1)!}
+\\+
C_7\cdot t^{14}\cdot \exp(189t/8)
\cdot(\log g +7)^{13}
\cdot g^{(27t - 56)/8}
+\\+
C_8\cdot t \cdot (1+t)^{14} \cdot \exp(63t/4)
\cdot(\log g+7)^{14}\cdot \big(g^{(2t - 3)/2} + g^{(9t-28)/4}\big)
\,.
\end{multline}
Now note that
$$
\begin{aligned}
\frac{9t-28}{4} \le \frac{2t - 3}{2}
&\qquad\text{for }t \le \frac{22}{5}=4.4\quad\text{and}
\\
\frac{27t - 56}{8} \le \frac{2t-3}{2}
&\qquad\text{for }t \le \frac{44}{19}\approx 2.3\,.
\end{aligned}
$$
Note also that for the particular value $t=\frac{44}{19}$
of $t$ for which the powers of $g$ in the second and
in the third term in~\eqref{eq:up:bound:V2:long}
coincide, the power of $(\log g +7)$ is larger
in the third term. Since by construction $C_8>0$, there exists a constant $g_0$ such that for any
$g\ge g_0$ and any $t$ in the interval
$\left[0,\frac{44}{19}\right]$ the expression
$$
C_8\cdot t \cdot (1+t)^{14}
\cdot(\log g+7)^{14}
\cdot\exp(63t/4)
\cdot g^{(2t - 3)/2}
$$
dominates the sum~\eqref{eq:up:bound:V2:long}. This
completes the proof of
Proposition~\eqref{prop:up:bound:V2}.
\end{proof}

\subsection{Volume contribution of stable graphs with $3$ and more vertices}
\label{ss:V:ge:3}
In this section we adjust the bound for the sum of
contributions of stable graphs with $3$ and more vertices
to the Masur--Veech volume $\Vol\cQ_g$ elaborated by
A.~Aggarwal
in~\cite[Section~10]{Aggarwal:intersection:numbers} to
bounds for the associated generating series
in a variable $t$.

The Lemma below is based on Proposition~10.4
in~\cite{Aggarwal:intersection:numbers} and is parallel to
\cite[Lemma~10.5]{Aggarwal:intersection:numbers}.

\begin{Lemma}
\label{lem:upgraded:Aggarwal:10.5}
For any couple of integers $g$ and $V$ satisfying
$g \geq 2$, $V \geq 2$, and for any non-negative real $t$
we have
\begin{multline}
\label{eq:upgraded:Aggarwal:10.5}
\left(\frac{8}{3}\right)^{-4g} \sum_{S=0}^{g}
\Upsilon_g^{(V; S, T)} \cdot t^{S+T}
\le T^{2V+Y+1}\cdot\frac{A_{g,t}^T}{T!}\cdot
\\
\cdot 2^{12}\cdot2^{23V}
\cdot\frac{1}{V^{3V}}
\cdot(\log g + 7)^{-1}
\cdot g^{1/2-V}\cdot
\\
\cdot\Big(
1+2^V \big(A_{g,t}+A_{g,t}^V\big)
+A_{g,t}^V (2^V + A_{g,t}) \exp(A_{g,t})
\Big)\,.
\end{multline}
where we use the notations $Y = \min(2T, 3V)$
and $\displaystyle A_{g,t} =
\frac{9t}{8}\cdot \left(\log g + 7\right)$.
\end{Lemma}

\begin{proof}
It follows from Proposition~10.4
in~\cite{Aggarwal:intersection:numbers} that
\begin{equation}
\label{eq:sum:over:S}
\left(\frac{8}{3}\right)^{-4g}
\sum_{S=0}^{g} \Upsilon_g^{(V; S, T)} \cdot t^{S+T}
\le
2^{11} \cdot B_{T,V} \cdot
\left( \Sigma_1 + \Sigma_2 \right)\,,
\end{equation}
where
\begin{equation*}
B_{T,V} =
2T\cdot g^{1/2-V}\cdot
2^{20V}\cdot \left( \frac{9t}{4} \right)^T
\left( \frac{T}{V} \right)^{2V}
\frac{ (\log g + 7)^{T-1}\cdot (2T-1)!!}{V^V (2T-Y)!}\,,
\end{equation*}
and
\begin{equation*}
\Sigma_1 = 1+\sum_{S=1}^{V-1} S\cdot A_{g,t}^S \\
\quad \text{and} \quad
\Sigma_2 = \sum_{S=V}^{g} S \frac{A_{g,t}^S}{(S - V)!}\,.
\end{equation*}
Transforming the bound
in~\cite[Proposition~10.4]{Aggarwal:intersection:numbers}
to the above form we used the following trivial
observations. Since $V\ge 2$ we have $T\ge 1$. The case
$S=0$ corresponds to the constant summand ``$1$'' in
$\Sigma_1$. In the remaining cases, that is when $S\ge 1$,
we used the bound $S+T\le 2TS$ for the factor $(S+T)$
present
in~\cite[Proposition~10.4]{Aggarwal:intersection:numbers}
which is valid for all $S,T\in\N$.

Using the trivial inequality $(V-1)^2<2^V$, valid for any
$V\in\N$, we obtain the following bound on the first sum:
$$
\Sigma_1 \leq  1+(V-1)^2 \big(A_{g,t}+A_{g,t}^V\big)
\le 1+2^V \big(A_{g,t}+A_{g,t}^V\big)\,.
$$
The sum $\Sigma_2$ is a part of an infinite sum of the
form~\eqref{eq:def:poisson:moments} taken with parameters
$n=1$ and $\lambda=A_{g,t}$. Applying
Lemma~\ref{lem:poisson:moments} and using the fact that
$P_1(\lambda)=\lambda$ by~\eqref{eq:first:four:P:lambda},
we get the following bound:
$$
\Sigma_2 \leq \sum_{S=V}^{+\infty} S \frac{ A_{g,t}^S }{(S-V)!}
=
A_{g,t}^V \sum_{S=0}^{+\infty} (V+S) \frac{A_{g,t}^S}{S!}
=A_{g,t}^V (V + A_{g,t}) \exp(A_{g,t})\,.
$$
Applying the trivial bound $V\le 2^V$
and taking the sum we obtain
\begin{equation}
\label{eq:V3:Sigma:Sigma2:up:bound}
\Sigma_1 + \Sigma_2 \le
1+2^V \big(A_{g,t}+A_{g,t}^V\big)
+A_{g,t}^V (2^V + A_{g,t}) \exp(A_{g,t})\,.
\end{equation}

Using $(2T)^Y (2T-Y)! \ge (2T)! = 2^T T! (2T-1)!!$ and $Y \le 3V$
we obtain the following bound for $B_{T,V}$,
\begin{multline}
\label{eq:BTV:up:bound}
B_{T,V} =
2T\cdot g^{1/2-V}\cdot
2^{20V}\cdot \left( \frac{9t}{4} \right)^T
\frac{T^{2V}}{V^{2V}}
\cdot\frac{ (\log g + 7)^{T-1}}{V^V}
\cdot \frac{(2T-1)!!}{(2T-Y)!}
\le\\
\le 2^{1+20V+Y}
\cdot(\log g + 7)^{-1}
\cdot g^{1/2-V}
\cdot T^{2V+Y+1}
\cdot \frac{A_{g,t}^T}{V^{3V} T!}
\le\\
2\cdot 2^{23V}
\cdot(\log g + 7)^{-1}
\cdot g^{1/2-V}
\cdot\frac{T^{2V+Y+1}}{V^{3V}}
\cdot\frac{A_{g,t}^T}{T!}
\,.
\end{multline}
Putting together~\eqref{eq:V3:Sigma:Sigma2:up:bound}
and~\eqref{eq:BTV:up:bound}
into~\eqref{eq:sum:over:S}
we obtain~\eqref{eq:upgraded:Aggarwal:10.5}.
\end{proof}

We perform now summation over the variable $T$.
The following statement is an adjustment
of~\cite[Lemma~10.6]{Aggarwal:intersection:numbers}
to generating series in $t$.

\begin{Lemma}
\label{lem:upgraded:Aggarwal:10.6}
There exist constants $C_9$ and $C_{10}$ such that
for any couple of integers $g$ and $V$ satisfying
$g \geq 2$, $V \geq 2$, and for any non-negative real $t$
we have
\begin{multline}
\label{eq:upgraded:Aggarwal:10.6}
\left( \frac{8}{3} \right)^{-4g}
\sum_{E} \Upsilon_g^{(V; E)} t^E
\le t\cdot C_9
\cdot \exp(63 t / 4)
\cdot g^{\tfrac{9t+2}{4}}
\cdot\\
\cdot\left(
\left(
\frac{C_{10}\cdot V^{1/2}}{g} \right)^V
+
\left(
\frac{C_{10}
\cdot t^8
\cdot V^{1/2}\cdot (\log g + 7)^8}{g} \right)^V
\right)
\,.
\end{multline}
\end{Lemma}

\begin{proof}
First note that the
second and the third lines
of expression~\eqref{eq:upgraded:Aggarwal:10.5}
do not depend on the variable $T$.
To bound the sum
$$
\sum_{T=V-1}^{3g-3}
T^{2V+Y+1}\cdot\frac{A_{g,t}^T}{T!}\,,
$$
where $Y = \min(2T, 3V)$,
we bound separately the following three partial sums:
\begin{equation*}
\begin{aligned}
\Sigma_1 &=
\sum_{T=V-1}^{\lfloor 3V/2 \rfloor}
T^{2V+2T+1} \frac{A_{g,t}^T}{T!}
\\
\Sigma_2 &=
\sum_{T=\lceil 3V / 2 \rceil}^{6V+1}
T^{5V+1} \frac{A_{g,t}^T}{T!}
\\
\Sigma_3&=\sum_{T=6V+2}^{+\infty}
T^{5V+1} \frac{A_{g,t}^T}{T!}\,,
\end{aligned}
\end{equation*}
where we use the same notation
$A_{g,t} = \frac{9 t}{8}(\log g + 7)$
as in Lemma~\ref{lem:upgraded:Aggarwal:10.5}.

To bound $\Sigma_1$ and $\Sigma_2$ we use
twice the inequality $T! \geq e^{-T} T^T$ to obtain
\begin{multline*}
\Sigma_1 \le
\sum_{T=1}^{\lfloor 3V/2 \rfloor}
(e\cdot A_{g,t})^T
\cdot T^{2V+T+1}
\le\left(\frac{3V}{2} \right)^{2V+1+3V/2}
\cdot\sum_{T=1}^{\lfloor 3V/2 \rfloor}
(e\cdot A_{g,t})^T
\\
\le\left(\frac{3V}{2} \right)^{7V/2+1}
\left(\frac{3V}{2} \right)
\Big(e A_{g,t}+(e A_{g,t})^{3V/2}\Big)
=\\=
\Big(e A_{g,t}+(e A_{g,t})^{3V/2}\Big)
\left(\frac{3V}{2} \right)^{7V/2+2}\,.
\end{multline*}
Similarly,
\begin{multline*}
\Sigma_2
\le
\sum_{T=\lceil 3V / 2 \rceil}^{6V+1}
T^{5V+1-T} (e\cdot A_{g,t})^T
\le
(6V+1)^{5V+1-3V/2}
\cdot \sum_{T=\lceil 3V / 2 \rceil}^{6V+1}
(e\cdot A_{g,t})^T
\le\\
\le
(6V+1)^{7V/2+1}
\cdot(6V+1)\cdot
\Big((e\cdot A_{g,t})^{3V/2}+(e\cdot A_{g,t})^{6V+1}\Big)
\le\\ \le
\Big((e\cdot A_{g,t})^{3V/2}+(e\cdot A_{g,t})^{6V+1}\Big)
\cdot(7V)^{7V/2+2}
\,.
\end{multline*}
To bound the third sum, we use the inequality
$$
\frac{T^{5V+1}}{T!}\le \frac{6^{5V+1}}{(T - 5V - 1)!}\,,
$$
valid for $T > 6V+1$,
and the inequality
$\sum_{k=n}^{+\infty} \frac{x^k}{k!} \le x^n\exp(x)$,
valid for any non-negative $x$:
\begin{multline*}
\Sigma_3=\sum_{T=6V+2}^{+\infty}
T^{5V+1} \frac{A_{g,t}^T}{T!}
 \le 6^{5V+1} \sum_{T=6V+2}^{+\infty}
 \frac{A_{g,t}^T}{(T - 5V - 1)!}= \\
 = (6 A_{g,t})^{5V+1} \sum_{T=V+1}^{+\infty} \frac{A_{g,t}^T}{T!}
\le (6 A_{g,t})^{5V+1} \sum_{T=V}^{+\infty} \frac{A_{g,t}^T}{T!}
\le\\
 \le (6 A_{g,t})^{5V+1} A_{g,t}^{V}\exp(A_{g,t})
=(6 A_{g,t})^{6V+1}
 \cdot \exp(A_{g,t})\,.
\end{multline*}
Collecting the terms and applying the bounds
$V^2\le 4^V$
and $A_{g,t}^{3V/2}\le A_{g,t}+A_{g,t}^{7V/2}$
we get
\begin{multline*}
\Sigma_1+\Sigma_2+\Sigma_3
\le
\Big(e A_{g,t}+(e A_{g,t})^{3V/2}\Big)\cdot
\left(\frac{3V}{2} \right)^{7V/2+2}
+\\+
\Big((e\cdot A_{g,t})^{3V/2}+(e\cdot A_{g,t})^{6V+1}\Big)
\cdot(7V)^{7V/2+2}
+
(6 A_{g,t})^{6V+1}
 \cdot \exp(A_{g,t})
\le\\
\le
3e\cdot A_{g,t}\cdot\Big(\big(1+(e\cdot A_{g,t})^{6V}\big)
\cdot(7V)^{7V/2+2}
+
(6 A_{g,t})^{6V}
 \cdot \exp(A_{g,t})\Big)
\le\\
\le 10 t\cdot (\log g + 7)
\cdot\Big(49\cdot 4^V\big(1+(e\cdot A_{g,t})^{6V}\big)
\cdot(7V)^{7V/2}
+
(6 A_{g,t})^{6V}\cdot \exp(A_{g,t})\Big)\,.
\end{multline*}
Combining the resulting bound with~\eqref{eq:upgraded:Aggarwal:10.5}
we get
\begin{multline*}
\left( \frac{8}{3} \right)^{-4g}
\sum_{E} \Upsilon_g^{(V; E)} t^E
\le t\cdot g^{1/2-V}\cdot
\\
\cdot 10\cdot 2^{12}
\cdot\left(\frac{2^{23}}{V^3}\right)^V
\cdot\Big(49\cdot 4^V\big(1+(e\cdot A_{g,t})^{6V}\big)
\cdot(7V)^{7V/2}
+
(6 A_{g,t})^{6V}
\exp(A_{g,t})\Big)
\cdot
\\
\cdot\Big(
1+2^V \big(A_{g,t}+A_{g,t}^V\big)
+A_{g,t}^V (2^V + A_{g,t}) \exp(A_{g,t})
\Big)\,.
\end{multline*}
Expanding the product of the terms located in the second
and in the third lines of the expression above
we get a sum of $15$ non-negative terms, where
every term has the form
$$
a\cdot b^V\cdot (A_{g,t})^{cV+d}\cdot V^{\alpha V}\cdot\exp(k A_{g,t})\,
$$
with  constants $a,b,c,d,\alpha,k$ specific for each summand,
but satisfying, however, the following common conditions.
We always have $a,b>0$,
$c\in\{0,1,6,7\}$, $d\in\{0,1\}$,
$\alpha\le\frac{1}{2}$,
and $k\in\{0,1,2\}$. It remains to note that since
$A_{g,t}\ge 0$, we have $\exp(2A_{g,t})\ge \exp(A_{g,t})\ge \exp(0)$.
Note also, that since $V\ge 2$, our restrictions
on $c$ and $d$ imply that $A_{g,t}^{cV+d}\le\max(1,A_{g,t}^{8V})$.
These observations imply that each of the terms
can be bounded from above by the expression
$$
a\cdot b^V\cdot \big(1+A_{g,t}^{8V}\big)\cdot V^{V/2}\cdot\exp(2 A_{g,t})\,.
$$
Recall that
$A_{g,t} = \frac{9 t}{8}(\log g + 7)$, so
$\exp(2A_{g,t})=g^{\tfrac{9}{4}t}\exp(63t/4)$.
The above observations imply that
letting
$C_9=15a'$, where $a'$ is the maximal value of
the parameter $a$ over $15$ terms, and letting
$C_{10}=\left(\frac{9}{8}\right)^8 b'$, where $b'$ is the maximal value
of the parameter $b$ over $15$ terms,
we complete the proof of~\eqref{eq:upgraded:Aggarwal:10.6}.
\end{proof}

\begin{proof}[Proof of Proposition~\ref{prop:up:bound:V3}]
By Lemma~\ref{lem:upgraded:Aggarwal:10.6}, we have that for
all non-negative real $t$ we have
\begin{multline*}
\left( \frac{8}{3} \right)^{-4g}
\sum_{V = 3}^{2g-2} \sum_{E \geq 1}
\Upsilon_g^{(V; E)} t^E
\le
t\cdot C_9
\cdot \exp(63 t / 4)
\cdot g^{\tfrac{9t+2}{4}}\cdot
\\
\cdot\left(
\sum_{V = 3}^{2g-2}
\left(
\frac{C_{10}
\cdot V^{1/2}}{g} \right)^V
+
\sum_{V = 3}^{2g-2}
\left(
\frac{C_{10}
\cdot t^8
\cdot V^{1/2}\cdot (\log g + 7)^8}{g} \right)^V
\right)
\,.
\end{multline*}
Let us denote by
\[
a_V :=
\left(\frac{C_{10}
\cdot t^8
\cdot V^{1/2}\cdot (\log g + 7)^8}{g}\right)^V
\]
the term in the second sum. Then
\begin{multline*}
\frac{a_{V+1}}{a_V}
=
\frac{C_{10}
\cdot t^8
\cdot (1+1/V)^{V/2} \cdot (V+1)^{1/2} \cdot (\log g + 7)^8}{g}
\leq\\
\le\frac{C_{10}
\cdot t^8
\cdot e^{1/2} \cdot (2g-1)^{1/2} \cdot (\log g + 7)^8}{g}.
\end{multline*}
In particular, since $t$ is bounded, there exists $g_3$ such that
for $g\ge g_3$ we have $\frac{a_{V+1}}{a_V} \leq 1/2$ for all $V$.
Hence
$$
\sum_{V = 3}^{2g-2}
\left(
\frac{C_{10}
\cdot t^8
\cdot V^{1/2}\cdot (\log g + 7)^8}{g} \right)^V\le
2
\left(
\frac{C_{10}
\cdot t^8
\cdot 3^{1/2}\cdot (\log g + 7)^8}{g} \right)^3\,.
$$
Applying analogous bound for the first sum and
collecting the estimates we get
\begin{multline*}
\left( \frac{8}{3} \right)^{-4g}
\sum_{V = 3}^{2g-2} \sum_{E \geq 1}
\Upsilon_g^{(V; E)} t^E
\le\\
\\
\le t\cdot C_9
\cdot \exp(63 t / 4)
\cdot g^{\tfrac{9t+2}{4}}
\cdot g^{-3}\cdot
2\cdot
\left(C_{10}
\cdot 3^{1/2}\right)^3
\left(1+t^{24}\cdot(\log g + 7)^{24}\right)\\
\le  B_3\cdot t \cdot g^{\tfrac{9t - 10}{4}}
\cdot(\log g + 7)^{24}\,,
\end{multline*}
where
$$
B_3
=C_9 \cdot \exp\left(\frac{63}{4}\cdot\frac{44}{19}\right)
\cdot 2
\cdot \left(C_{10}\cdot 3^{1/2}\right)^3
\cdot 2\cdot \left(\frac{44}{19}\right)^{24}\,.
$$
\end{proof}

\section{Proofs}
\label{s:proofs}

We proved in Section~\ref{ss:from:q:to:p1} mod-Poisson
convergence of the distribution $\ProbaCylsOne_g$
corresponding to volume contributions of stable graphs with
a single vertex. In Section~\ref{ss:from:p1:to:p} we apply
the results collected in
Section~\ref{s:disconnecting:multicurves} on volume
contributions of stable graphs with two and more vertices
to prove that the distribution $\ProbaCylsOne_g$
well-approximates the distribution $\ProbaCyls_g$. In
Section~\ref{ss:Remaining:proofs} we present the remaining
proofs of Theorems stated in Section~\ref{s:intro}.

\subsection{From $\ProbaCylsOne_g$ to $\ProbaCyls_g$}
\label{ss:from:p1:to:p}
For any $k$ denote
\begin{equation}
\label{eq:vol:k:cyl}
\Vol_{k\textit{-cyl}}\cQ_g
:= \sum_{\substack{\Graph \in \cG_{g}\\|E(\Gamma)|=k}}
\Vol(\Gamma)\,,
\end{equation}
the contribution to $\Vol\cQ_g$ of stable graphs with
exactly $k$ edges. Using this notation, the probability distribution $\ProbaCyls_g(k)$ defined
in Theorem~\ref{th:same:distribution} can be rewritten as
\[
\ProbaCyls_g(k) = \frac{\Vol_{k\textit{-cyl}}\cQ_g}{\Vol \cQ_g}.
\]
Recall that we also have a probability distribution
$q_{3g-3, \infty, 1/2}(k)$ defined
in~\eqref{eq:q:3g:minus:3:infty:1:2:as:proba}
and evaluated in~\eqref{eq:q:n:m:alpha}
that corresponds to the number of cycles in a
random permutation of $3g-3$ elements according to the probability
distribution $\Proba_{3g-3, \infty, 1/2}$ on $S_{3g-3}$, see
Lemma~\ref{lem:cycle:distribution:vs:harmonic:sum}.

We gather the results from
Sections~\ref{s:sum:over:single:vertex:graphs}
and~\ref{s:disconnecting:multicurves} in the following two
statements.

\begin{Theorem}
\label{thm:generating:series:vol}
For $t\in\C$ satisfying $|t|\le 4/5$ we have as $g \to
+\infty$
$$
\sum_{k=1}^{3g-3} \Vol_{k\textit{-cyl}}\cQ_g\ t^k
=\frac{2t}{\sqrt{\pi}\, \Gamma\!\left(1+\frac{t}{2}\right)}
(6g-6)^{\tfrac{t-1}{2}}  \left( \frac{8}{3} \right)^{4g-4}
\!\left(1 + O\left(g^{\tfrac{t}{2}-1} (\log g)^{24} \right)\! \right),
$$
where the error term is uniform for $t$ in the disk $|t|\le 4/5$.

For $t\in\C$ satisfying $4/5\le |t| <8/7$ we have as $g \to
+\infty$
$$
\sum_{k=1}^{3g-3} \Vol_{k\textit{-cyl}}\cQ_g\ t^k
=\frac{2t}{\sqrt{\pi} \, \Gamma\!\left(1+\frac{t}{2}\right)}
(6g-6)^{\tfrac{t-1}{2}}  \left( \frac{8}{3} \right)^{4g-4}
\!\left(1 + O\left(g^{\tfrac{7 t}{4} - 2} (\log g)^{24} \right)\! \right)\,,
$$
where the constant in the error term is uniform for $t$
in any compact subset of the annulus $4/5\le |t| <8/7$.
In particular, for $t=1$ we get
\begin{equation}
\label{eq:Vol:Q:with:error:term}
\Vol(\cQ_g) =
\sum_{k=1}^{3g-3} \Vol_{k\textit{-cyl}}\cQ_g\
=\frac{4}{\pi}
\cdot \left( \frac{8}{3} \right)^{4g-4}
\cdot \left(1 + O\left(g^{-1/4} \cdot (\log g)^{24}\right)\right)\,.
\end{equation}
\end{Theorem}

We note that the asymptotic formula for
$\Vol(\cQ_g)$ without explicit
error term as in~\eqref{eq:Vol:Q:with:error:term}
was conjectured
in~\cite{DGZZ:volume} and proved in~\cite[Theorem
1.7]{Aggarwal:intersection:numbers}.
See also Remark~\ref{rm:expansion:of:error:term}
for the discussion of the expected optimal error term.

\begin{Theorem}
\label{thm:p_g:individual:estimates}
For any $k$ satisfying
$k\le \frac{\log g}{\log\frac{9}{4}}$
we have
$$
\Vol_{\textrm{k-cyl}} \cQ_g
=
\Vol\Gamma_k(g) \left(1 +
O\left((\log g)^{25} \cdot
g^{-1+\tfrac{k \log 2}{\log g}}\right)\right)\,.
$$
For any $x$ satisfying
$x<\frac{2}{\log\frac{9}{2}}$
and for all $k$ satisfying
$\frac{\log g}{\log\frac{9}{4}}\le k \le x\log g$
we have
$$
\Vol_{\textrm{k-cyl}} \cQ_g
=
\Vol\Gamma_k(g) \left(1 +
O\left((\log g)^{25} \cdot
g^{-2+\tfrac{k \log\frac{9}{2}}{\log g}}\right)\right)\,.
$$
For $k\le \frac{3}{4\log 2} \log g$ we have
\begin{equation}
\label{eq:p:through:q:minus:1:4}
p_g(k)
=q_{3g-3,\infty,\frac{1}{2}}(k)
\cdot \left(1 + O\left(g^{-1/4} \cdot (\log g)^{24}\right)\right)\,.
\end{equation}
For $k$ in the range
$\frac{3\log g}{4\log 2}\le
k\le \frac{\log g}{\log\frac{9}{4}}$
we have
$$
\begin{aligned}
p_g(k)
&=q_{3g-3,\infty,\frac{1}{2}}(k)
\cdot \left(1 + O\left(g^{-1}\cdot 2^k\cdot (\log g)^{25}\right)\right)
\\
&=q_{3g-3,\infty,\frac{1}{2}}(k) \left(1 +
O\left((\log g)^{25} \cdot
g^{-1+\tfrac{k \log 2}{\log g}}\right)\right)\,.
\end{aligned}
$$
For $k$ in the range
$\frac{\log g}{\log\frac{9}{4}}\le k \le x\log g$
we have
$$
\begin{aligned}
p_g(k)
&=q_{3g-3,\infty,\frac{1}{2}}(k)
\cdot \left(1 + O\left(g^{-1}
\cdot\left(\frac{9}{2}\right)^k
\cdot (\log g)^{25}\right)\right)
\\
&=q_{3g-3,\infty,\frac{1}{2}}(k) \left(1 +
O\left((\log g)^{25} \cdot
g^{-2+\tfrac{k \log\frac{9}{2}}{\log g}}\right)\right)\,,
\end{aligned}
$$
where all above estimates are uniform in the corresponding
ranges of $k$.
\end{Theorem}

\begin{proof}[Proof of Theorem~\ref{thm:generating:series:vol}]
Using the notations from Definition~\ref{def:volume:contributions} we decompose
\[
\Vol_{k\textit{-cyl}}\cQ_g = \Upsilon_g^{(1; k)} + \Upsilon_g^{(2;k)} + \sum_{V \geq 3} \Upsilon_g^{(V; k)}.
\]
Here $\Upsilon_g^{(1; k)} = \Vol \Gamma_k(g)$.
Note that $8/7<2$, so
Theorem~\ref{thm:generating:series:vol:1}
gives the uniform asymptotic
equivalence for the first term. Applying the identity
$\frac{t}{2}\,\Gamma\!\left(\frac{t}{2}\right)=\Gamma\!\left(1+\frac{t}{2}\right)$
we set $m=+\infty$ and rewrite~\eqref{eq:generating:series:vol:1}
as
\begin{equation}
\label{eq:sum:Upsilon:1:t:k}
\sum_{k \geq 1} \Upsilon_g^{(1; k)} t^k
=
\frac{2t}{\sqrt{\pi} \, \Gamma\left(1+\frac{t}{2}\right)}
(6g-6)^{\tfrac{t-1}{2}} \left( \frac{8}{3} \right)^{4g-4}
\left(1 + O\left( \frac{(\log g)^2}{g}\right) \right)\,.
\end{equation}
The bounds for the contributions of the second and third terms are provided
by Propositions~\ref{prop:up:bound:V2} and~\ref{prop:up:bound:V3}
respectively.
We have for $|t|\in[0,44/19]$ and, hence, for
$|t|\in[0,8/7]$:
\begin{align*}
\left( \frac{8}{3} \right)^{-4g}
\sum_{k \geq 1} \Upsilon_g^{(2; k)} |t|^k
&\le g^{\tfrac{|t|-1}{2}}
\cdot O\left(g^{\tfrac{|t|-2}{2}} (\log g)^{14} \right)
\\
\left( \frac{8}{3} \right)^{-4g}
\sum_{k \geq 1} \sum_{V \geq 3} \Upsilon_g^{(V; k)} |t|^k
&\le g^{\tfrac{|t|-1}{2}}
\cdot O\left(g^{\tfrac{7|t|-8}{4}} (\log g)^{24} \right)\,.
\end{align*}
Hence
\begin{multline}
\label{eq:two:error:terms}
\sum_{k \geq 1} \Vol_{k\textit{-cyl}} \cQ_g t^k
=\\
= \left(\sum_{k \geq 1} \Upsilon_g^{(1; k)} t^k\right)
\cdot\left(1 + O\left((\log g)^{14} \cdot g^{(|t|-2)/2}\right)
+ O\left((\log g)^{24} \cdot g^{(7|t|-8)/4}\right)\right)\,.
\end{multline}
Note that $\frac{|t|-2}{2}\ge -1$, so
the error term
$O\left((\log g)^{14} \cdot g^{(|t|-2)/2}\right)$
dominates the error term
$O\left((\log g)^2\cdot g^{-1}\right)$
coming from~\eqref{eq:sum:Upsilon:1:t:k}.
Note also that
$$
\frac{|t|-2}{2} \ge \frac{7|t|-8}{4}
\quad\text{for}\quad |t|\le \frac{4}{5}
\qquad\text{and}\qquad
\frac{|t|-2}{2} \le \frac{7|t|-8}{4}
\quad\text{for}\quad |t|\ge \frac{4}{5}\,.
$$
This shows which of the two error terms in~\eqref{eq:two:error:terms}
dominates on which interval of the values $|t|$.
Plugging~\eqref{eq:sum:Upsilon:1:t:k}
into~\eqref{eq:two:error:terms},
taking into consideration the observation concerning
the domination of the error terms,
and taking the maximum of powers $14$ and $24$
of logarithms to cover the case $|t|=\frac{4}{5}$,
we complete
the proof of Theorem~\ref{thm:generating:series:vol}.
Note that passing to~\eqref{eq:Vol:Q:with:error:term}
we used that $\Gamma(3/2)=\sqrt{\pi}/2$.
\end{proof}

In the proof of Theorem~\ref{thm:p_g:individual:estimates}
we use the following saddle point bound which corresponds to
Equation~(18) from~\cite[Proposition~IV.1]{Flajolet:Sedgewick}.

\begin{NNProposition}[\cite{Flajolet:Sedgewick}, Proposition~IV.1]
Let $f(z)$ be analytic in the disk $|z|<R$ with
$0<R\le\infty$. Define $M(f;r)$ for $r\in(0,R)$ by
$M(f;r):=\sup_{|z|=r}|f(z)|$. Then, one has for any $r$ in
$(0,R)$, the family of saddle-point upper bounds
\begin{equation}
\label{eq:FS:saddle:point:bound}
\left[z^n\right]f(z)\le \frac{M(f;r)}{r^n}
\qquad\text{implying}\qquad
\left[z^n\right]f(z)\le \inf_{r\in(0,R)} \frac{M(f;r)}{r^n}\,.
\end{equation}
\end{NNProposition}

\begin{proof}[Proof of Theorem~\ref{thm:p_g:individual:estimates}]
Let $\delta=\frac{44}{19}$. From
Propositions~\ref{prop:up:bound:V2}
and~\ref{prop:up:bound:V3} we have for all $t$ in the
interval $[0,\delta)$ the bounds
\begin{align*}
\sum_{k \geq 1} \Upsilon_g^{(2; k)} t^k
&\leq
B_2 \cdot \left( \frac{8}{3} \right)^{4g}
\cdot t \cdot g^{\tfrac{2t - 3}{2}} \cdot (\log g)^{14}
\\
\sum_{k \geq 1} \sum_{V \geq 3} \Upsilon_g^{(V; k)} t^k
&\leq
B_3 \cdot \left( \frac{8}{3} \right)^{4g}
\cdot t \cdot g^{\tfrac{9 t - 10}{4}} \cdot (\log g)^{24}\,.
\end{align*}
Combining these bounds with~\eqref{eq:FS:saddle:point:bound}
we obtain for all non-negative integer $k$
\begin{align*}
\Upsilon_g^{(2; k)}
&\leq
B_2 \cdot \left( \frac{8}{3} \right)^{4g} \cdot (\log g)^{14}
\cdot g^{-3/2} \cdot \inf_{t \in (0, \delta)}
\left(t^{1-k} g^t\right)\,,
\\
\sum_{V \geq 3} \Upsilon_g^{(V; k)}
&\leq
B_3 \cdot \left( \frac{8}{3} \right)^{4g} \cdot (\log g)^{24}
\cdot g^{-5/2} \cdot \inf_{t \in (0, \delta)}
\left(t^{1-k} g^\frac{9t}{4}\right)\,.
\end{align*}
For the rest of the proof we assume that $t$ is real and
is contained in the interval $[0,\delta)$.
The minima of $t^{1-k} g^t$ and $t^{1-k} g^{\frac{9t}{4}}$
on $[0,+\infty)$ are reached at $t =
\frac{k-1}{\log g}$ and at $t = \frac{k-1}{\frac{9}{4}
\log g}$ respectively. Hence, for $k-1 \leq \delta\cdot\log g$,
we obtain the following bounds:
\begin{align*}
\Upsilon_g^{(2;k)}
&\leq
B_2 \cdot \left( \frac{8}{3} \right)^{4g} \cdot (\log g)^{14}
\cdot g^{-3/2} \cdot (\log g)^{k-1} \cdot
\left(\frac{e}{k-1}\right)^{k-1}\,, \\
\sum_{V \geq 3} \Upsilon_g^{(V; k)}
&\leq
B_3 \cdot \left( \frac{8}{3} \right)^{4g} \cdot (\log g)^{24}
\cdot g^{-5/2} \cdot \left( \frac{9}{4} \right)^{k-1}
(\log g)^{k-1} \left( \frac{e}{k-1} \right)^{k-1}\,.
\end{align*}
Now, for $g$ large enough we have
$\sqrt{2 \pi \delta \log g} \leq \log g$.
Hence, by Stirling formula, for $g$ large
enough and for all $k-1 \leq \delta\cdot\log g$ we have
\begin{align*}
\Upsilon_g^{(2; k)}
&\leq
B_2 \cdot \left( \frac{8}{3} \right)^{4g} \cdot (\log g)^{15}
\cdot g^{-3/2} \cdot \frac{(\log g)^{k-1}}{(k-1)!}\,,
\\
\sum_{V \geq 3} \Upsilon_g^{(V; k)}
&\leq
B_3 \cdot \left( \frac{8}{3} \right)^{4g} \cdot (\log g)^{25}
\cdot g^{-5/2}
\cdot \left( \frac{9}{4}\right)^{k-1}
\cdot \frac{(\log g)^{k-1}}{(k-1)!}\,.
\end{align*}

Combining expression~\eqref{eq:bounds:in:terms:of:H:k:gminus3}
(in which we set $m=+\infty$)
for $\Upsilon_g^{(1; k)}= \Vol \Gamma_k(g)$ from
Theorem~\ref{th:bounds:for:Vol:Gamma:k:g}
with expression~\eqref{eq:H:n:m:alpha}
for $\ContributionH_{3g-3,\infty,1/2}(k)$
from Corollary~\ref{cor:multi:harmonic:asymptotic:all:k:bis}
we get the following
asymptotics for $\Upsilon^{(1; k)}$ as
$g \to +\infty$:
\begin{multline*}
\Upsilon_g^{(1; k)}
= \sqrt{\frac{2}{\pi}}
\frac{1}{\sqrt{3g-3}}
\left( \frac{8}{3} \right)^{4g-4}
\left(\frac{1}{2} \right)^{k-1}
\frac{\big(\log(6g-6)\big)^{k-1}}{(k-1)!}\cdot
\\
\cdot\left(
\frac{1}
{\Gamma\left(1 + \tfrac{k-1}{\log(6g-6)}\right)}
+ O\left(\frac{k-1}{(\log g)^2} \right) \right)\,.
\end{multline*}
For all $k-1 < \delta \log g$ the rightmost factor in the
above expression is greater than or equal to
$1/\Gamma(1+44/19)+O((\log g)^{-1})$. We have
$1/\Gamma(1+44/19)>1/3$.
Hence, for all $k-1 < \delta \log g$ and as $g \to
+\infty$ we have
\begin{equation}
\label{eq:conditional}
\frac{\displaystyle
\Upsilon_g^{(2; k)} + \sum_{V \geq 3} \Upsilon_g^{(3; k)}}
{\Upsilon_g^{(1; k)}}
=
O\left((\log g)^{25} \cdot
\max\left(g^{-1} \cdot 2^k,
g^{-2} \cdot
\left(\frac{9}{2}\right)^k
\right)\right).
\end{equation}
Rewriting
$$
2^{k} = g^{\tfrac{k}{\log g} \log 2}
\qquad\text{and}\qquad
\left(\frac{9}{2}\right)^k
=g^{\tfrac{k}{\log g} \left(\log 9 -\log 2\right)}
$$
we obtain
$$
\max\left(g^{-1} \cdot 2^k,
g^{-2} \cdot
\left(\frac{9}{2}\right)^k
\right)
=
g^{\max\left(-1+\tfrac{k\log 2}{\log g},
-2+\tfrac{k}{\log g} \left(\log 9 -\log 2\right)
\right)}
$$
Solving the linear equation we find that
\begin{equation}
\label{eq:winning:error:term}
\begin{aligned}
-2+x\log\tfrac{9}{2} \le -1+x\log 2
&\quad\text{for}\quad
x\le  \frac{1}{\log\frac{9}{4}}\approx 1.23315\,,
\\
-1+x\log 2 \le -2+x\log\tfrac{9}{2}<0
&\quad\text{for}\quad
\frac{1}{\log\frac{9}{4}}\le x <
\frac{2}{\log\frac{9}{2}}\approx 1.32972\,.
\end{aligned}
\end{equation}
Note that $\delta=\frac{44}{19}\approx 2.31$, so
$\delta>\frac{2}{\log\frac{9}{2}}$. Note also that
$\Upsilon_g^{(1; k)}=\Vol\Gamma_k(g)$. We conclude that
for any $x$ satisfying $x <
\frac{2}{\log\frac{9}{2}}$ we have
\begin{multline*}
\Vol_{\textrm{k-cyl}} \cQ_g
=
\Upsilon_g^{(1; k)} \left( 1 + \frac{\displaystyle \Upsilon_g^{(2; k)} + \sum_{V \geq 3} \Upsilon_g^{(3; k)}}{\Upsilon_g^{(1; k)}} \right)
=\\
=\begin{cases}
\Vol\Gamma_k(g) \left(1 +
O\left((\log g)^{25} \cdot
g^{-1+\tfrac{k \log 2}{\log g}}\right)\right)
&\text{for}\quad
k\le \frac{\log g}{\log\frac{9}{4}}\,;
\\
\Vol\Gamma_k(g) \left(1 +
O\left((\log g)^{25} \cdot
g^{-2+\tfrac{k \log\frac{9}{2}}{\log g}}\right)\right)
&\text{for}\quad
\frac{\log g}{\log\frac{9}{4}}\le k \le x\log g\,,
\end{cases}
\end{multline*}
uniformly in the corresponding range of $k$.
This completes the proof of the first assertion of Theorem~\ref{thm:p_g:individual:estimates}.
By~\eqref{eq:ProbaCylsOne:k} we have
$$
\frac{\Upsilon_g^{(1; k)}}{\Upsilon_g^{(1)}}=
\ProbaCylsOne_{g,\infty}(k)
=q_{3g-3,\infty,\frac{1}{2}}(k)
\cdot
\left(1 + O\left( \frac{(k+2\log g)^2}{g}\right) \right)
$$
uniformly for all $k\le \tfrac{2\log g}{\log\frac{9}{2}}$.
By Equation~\eqref{eq:Vol:Q:with:error:term}
from Theorem~\ref{thm:generating:series:vol}, we have
$$
\Vol\cQ_g =
\Upsilon_g^{(1)} \left(1 + O\left((\log g)^{24} g^{-1/4}\right)\right)\,.
$$
Note that
$$
\begin{aligned}
-1+x\log 2 \le -\frac{1}{4}
&\quad\text{for}\quad
x\le \frac{3}{4 \log 2}\approx 1.08202\,,
\\
-\frac{1}{4} \le -1+x\log 2
&\quad\text{for}\quad
x\ge \frac{3}{4 \log 2}\,.
\end{aligned}
$$
Combining the latter considerations
with~\eqref{eq:winning:error:term} we conclude that for any
$x<\frac{2}{\log\frac{9}{2}}$ we have
\begin{multline*}
p_g(k)=
\\=\begin{cases}
q_{3g-3,\infty,\frac{1}{2}}(k)
\left(1 + O\left((\log g)^{25} \cdot
g^{-\tfrac{1}{4}}
\right) \right)
&\text{for}\quad
k\le \frac{3\log g}{4\log 2}\,;
\\
q_{3g-3,\infty,\frac{1}{2}}(k) \left(1 +
O\left((\log g)^{25} \cdot
g^{-1+\tfrac{k \log 2}{\log g}}\right)\right)
&\text{for}\quad
\frac{3\log g}{4\log 2}\le
k\le \frac{\log g}{\log\frac{9}{4}}\,;
\\
q_{3g-3,\infty,\frac{1}{2}}(k) \left(1 +
O\left((\log g)^{25} \cdot
g^{-2+\tfrac{k \log\frac{9}{2}}{\log g}}\right)\right)
&\text{for}\quad
\frac{\log g}{\log\frac{9}{4}}\le k \le x\log g\,,
\end{cases}
\end{multline*}
uniformly for all $k$ in the corresponding ranges.
Theorem~\ref{thm:p_g:individual:estimates} is proved.
\end{proof}

\subsection{Remaining proofs}
\label{ss:Remaining:proofs}

Next we deduce the statements stated in
Section~\ref{s:intro} from
Theorems~\ref{thm:generating:series:vol} and~\ref{thm:p_g:individual:estimates}.

\begin{proof}[Proof of Theorem~\ref{thm:mod:poisson:pg}]
Let $K_g(\gamma)$ be the number of components of a random
multicurve $\gamma$ on a surface of genus $g$. Let
\[
F_g(t) = \sum_{k=1}^{+\infty} \Vol_{k\textit{-cyl}}\cQ_g\ t^k.
\]
By definition $F_g(1) = \Vol \cQ_g$ and we have
\[
\E_g(t^{K_g(\gamma)}) = \frac{F_g(t)}{F_g(1)}.
\]
Applying Theorem~\ref{thm:generating:series:vol} we obtain the result.
\end{proof}

We say that a multicurve
$\gamma=m_1\gamma_1+\dots+m_k\gamma_k$ is
\textit{non-separating} if primitive components $\gamma_1,
\dots,\gamma_k$ of $\gamma$ represent linearly independent
homology cycles. Otherwise we say that a multicurve is
\textit{separating}.
Clearly, $S\setminus\{\gamma_1\cup\dots\cup\gamma_k\}$
is connected if and only if $\gamma$ is non-separating,
so non-separating multicurves correspond to stable graphs
with a single vertex, while separating
multicurves correspond to stable graphs
with two and more vertices.

The Corollary below is a quantitative version of
an analogous statement~\cite[Proposition~10.7]{Aggarwal:intersection:numbers}
due to A.~Aggarwal. We originally conjectured a weaker form
of this assertion in~\cite[Conjecture~1.33]{DGZZ:volume}.

\begin{Corollary}
\label{cor:separating:conditional}
\label{conj:one:vertex:dominates:for:fixed:k}
The following estimate is uniform for
$k$ in the interval $\left[1, \frac{\log g}{\log\frac{9}{4}}\right]$
as $g \to +\infty$:
$$
\Proba\big(\text{$\gamma$ is separating}\ |\ K_g(\gamma) = k\big)
= O\left((\log g)^{25}\cdot g^{-1+x\log 2}\right).
$$
For any $x$ satisfying $\frac{1}{\log\frac{9}{4}}\le x<\frac{2}{\log\frac{9}{2}}$
the following estimate is uniform for
$k$ in the interval $\left[\frac{\log g}{\log\frac{9}{4}}, x \log g\right]$
as $g \to +\infty$:
$$
\Proba\big(\text{$\gamma$ is separating}\ |\ K_g(\gamma) = k\big)
= O\left((\log g)^{25}\cdot g^{-2+x\log\tfrac{9}{2}}\right).
$$
Furthermore, for any fixed $k$ we have
$$
\Proba(\text{$\gamma$ is separating}\ |\ K_g(\gamma) = k)
= O\left((\log g)^{25}\cdot g^{-1}\right).
$$
\end{Corollary}

\begin{proof}
By definition,
\[
\Proba(\text{$\gamma$ is separating}\ |\ K_g(\gamma) = k)
= \frac{\Upsilon_g^{(2; k)} + \sum_{V \geq 3} \Upsilon_g^{(3; k)}}{\Vol_{k\textit{-cyl}}\cQ_g}
\leq
\frac{\Upsilon_g^{(2; k)} + \sum_{V \geq 3} \Upsilon_g^{(3; k)}}{\Upsilon_g^{(1; k)}}.
\]
Equation~\eqref{eq:conditional} in the proof of
Theorem~\ref{thm:p_g:individual:estimates} and analysis of
the error term following it provides an upper bound for the
right hand-side of the expression above from which the
corollary follows.
\end{proof}

\begin{Remark}
We proved in~\cite[Theorem~1.27]{DGZZ:volume},
that for $k=1$ we, actually, have the following exponential decay
\[
\Proba(\text{$\gamma$ is separating}\ |\ K_g(\gamma) = 1) \sim \sqrt{\frac{2}{3 \pi g}} \cdot 4^{-g}\,.
\]
\end{Remark}

\begin{proof}[Proof of Theorems~\ref{th:multicurves:a:b:c}
and~\ref{th:square:tiled:a:b:c}]
It follows from combination of
\cite[Propositions~8.3--8.5]{Aggarwal:intersection:numbers}
proved by A.~Aggarwal that asymptotically, as
$g\to+\infty$, the relative contribution to the Masur--Veech
volume $\Vol\cQ_g$ coming from all stable graphs in $\cG_g$ which
have more than one vertex, tends to zero. Translated to the
language of multicurves or to the language of square-tiled
surfaces, this statement corresponds to assertion~(a)
of Theorems~\ref{th:multicurves:a:b:c}
and~\ref{th:square:tiled:a:b:c}.

In terms of the results of the current paper,
the same statement can be justified comparing
Theorems~\ref{thm:generating:series:vol:1}
and~\ref{thm:generating:series:vol} and observing that
the asymptotics of $\sum_{k \geq 1} \Vol(\Petal_k(g))$
and of $\Vol(\cQ_g)$ are the same up to a factor
which tends to $1$ as $g\to+\infty$.

Assertion~(b) is a particular case,
corresponding to the value $m=1$
of the parameter $m$, of more general
Theorems~\ref{th:multicurves:bounded:weights}
and~\ref{th:square:tiled:bounded:weights}. These Theorems
are proved independently below.

Let us prove assertion~(c). By Theorem~\ref{thm:p_g:individual:estimates}, in the
range $k = o(\log g)$ the volume contributions $\Vol_{\textrm{k-cyl}} \cQ_g$
and $\Upsilon_g^{(1; k)}$ are asymptotically equivalent. We have
\[
\Upsilon_g^{(1; k)} = \sum_{m_1, \ldots, m_k \geq 1}
\Vol \big(\Gamma_g(k), (m_1, \ldots m_k)\big)\,,
\]
where the contribution of primitive multicurves is equal to
$\Vol\big(\Gamma_g(k), (1, \ldots 1)\big)$. By Theorem~\ref{th:bounds:for:Vol:Gamma:k:g},
the contribution to $\Vol\cQ_g$ coming from
all non-separating multicurves and from all
primitive non-separating multicurves are respectively proportional to
$\ContributionH_{3g-3, \infty, 1/2}(k)$ and to $\ContributionH_{3g-3, 1, 1/2}(k)$
with the same coefficient of proportionality
$\frac{2 \sqrt{2}}{\sqrt{\pi}} \cdot \sqrt{3g-3} \cdot \left(\frac{8}{3}\right)^{4g-4}$.
By Corollary~\ref{cor:multi:harmonic:asymptotic:all:k:bis}, the quantities
$\ContributionH_{3g-3, \infty, 1/2}$ and $\ContributionH_{3g-3, 1, 1/2}(k)$ are
asymptotically equivalent in the range $[0, o(\log g)]$,
which completes the proof.
\end{proof}

\begin{proof}[Proof of Theorems~\ref{th:multicurves:bounded:weights} and~\ref{th:square:tiled:bounded:weights}]
Taking the ratio of expression~\eqref{eq:generating:series:vol:1}
from
Theorem~\ref{thm:generating:series:vol:1}
evaluated at $t=1$
over expression~\eqref{eq:sum:Vol:Gamma:k}
from the same Theorem we get
$$
\lim_{g\to+\infty}
\frac{\displaystyle\sum_{k=1}^{3g-3}
\sum_{m_1, \ldots, m_k \le m}
\Vol(\Gamma_k(g), (m_1, \ldots, m_k))}
{\displaystyle\sum_{k=1}^{3g-3}
\Vol\big(\Petal_k(g)\big)}
= \sqrt{\frac{m}{m+1}}\,.
$$
Since the contribution from stable graphs with $V \geq 2$
vertices is negligible, it is a fortiori negligible when we
consider bounded multiplicities $m_i\le m$. Hence,
we have as $g \to +\infty$ the asymptotics
\begin{multline*}
\lim_{g\to+\infty}
\frac{\displaystyle\sum_{k=1}^{+\infty}
\sum_{m_1, \ldots, m_k \le m}
\Vol(\Gamma_k(g), (m_1, \ldots, m_k))}
{\displaystyle\sum_{k=1}^{+\infty}
\Vol\big(\Petal_k(g)\big)}
=\\
=\lim_{g\to+\infty}
\frac{\displaystyle \sum_{\Gamma \in \cG_g}
\sum_{m_1, \ldots, m_k \le m}
\Vol(\Gamma, (m_1, \ldots, m_k))}{\Vol(\cQ_g)}\,,
\end{multline*}
which concludes the proof.
\end{proof}

\begin{proof}[Proof of Theorem~\ref{thm:CLT:multicurve}
and of Theorem~\ref{thm:CLT:square:tiled}]
The central limit theorem for $K_g(\gamma)$ follows from the general Theorem~\ref{thm:CLT}
that holds under mod-Poisson convergence. The mod-Poisson convergence was stated
in Theorem~\ref{thm:mod:poisson:pg}.
It only remains to justify the normalization
used in Theorem~\ref{thm:CLT:multicurve}
and in Theorem~\ref{thm:CLT:square:tiled}.

It follows from~\eqref{eq:Vol:sq:tiled} that
$$
\card(\cSTg(N))\sim m(g)\cdot N^{6g-6}\,,
$$
where
$$
m(g)=\frac{\Vol\cQ_g}{(12g-12)\cdot 2^{6g-6}}\,.
$$
By the central limit
theorem (Theorem~\ref{thm:CLT}) we obtain
\begin{multline*}
\lim_{g\to+\infty}
\lim_{N\to+\infty}\
\frac{1}{m(g)\cdot N^{6g-6}}
\cdot\card\left\{S\in\cSTg(N)\,\bigg|\,
\frac{K_g(S)-\lambda_{3g-3}}{\sqrt{\lambda_{3g-3}}}
\le x\right\}
\\
=\frac{1}{\sqrt{2\pi}}
\int_{-\infty}^x e^{-\frac{t^2}{2}} dt\,.
\end{multline*}
It remains to use~\eqref{eq:Vol:Qg} from
Theorem~\ref{conj:Vol:Qg} of A.~Aggarwal (see Theorem~1.7
in the original paper~\cite{Aggarwal:intersection:numbers})
to compute
$$
\frac{1}{m(g)}
=\frac{(12g-12)\cdot 2^{6g-6}}{\Vol\cQ_g}
\sim 3\pi g\cdot\left(\frac{9}{8}\right)^{2g-2}\,,
$$
which proves Theorem~\ref{thm:CLT:square:tiled}.

The proof of Theorem~\ref{thm:CLT:multicurve}
analogous.

M.~Mirzakhani proved in~\cite{Mirzakhani:Thesis}
that for any integral multicurve $\lambda\in\cML(\Z)$
one has
$$
\card\left(\left\{\gamma\in\cML_g(\Z)\,\bigg|\,
\iota(\lambda,\gamma)\le N\right\}/\Stab(\lambda)\right)
\sim \tilde c(\lambda)\cdot N^{6g-6}\,.
$$
Now let $\lambda=\rho_g$, where
$\rho_g$ is a simple closed
non-separating curve on a surface of genus $g$.
Note that the stable graph associated to $\rho_g$
is $\Gamma_1(g)$ and that the associated weight $m_1$
is equal to $1$.
By~\cite[Proposition~8.8]{Erlandsson:Souto}
the asymptotic frequency $\tilde c(\rho_g)$
in the expression above is proportional to
the asymptotic frequency
$c(\rho_g)$
defined in~\eqref{eq:frequency:c}
with the following factor:
$$
c(\rho_g)=2^{2g-3}\cdot \tilde{c}(\rho_g)\,.
$$
Combining this relation with~\eqref{eq:Vol:gamma:c:gamma}
and~\eqref{eq:const:g:n} (where we let $n=0$)
we get
$$
\Vol(\Gamma_1(g),1)=2(6g-6)\cdot(4g-4)!
\cdot 2^{6g-6}\cdot\tilde{c}(\rho_g)\,,
$$
Since $\Vol(\Gamma_1(g))=\zeta(6g-6)\cdot\Vol(\Gamma_1(g),1)$
and since $\zeta(6g-6)\sim 1$ as $g\to+\infty$ we conclude that
$$
\frac{1}{\tilde c(g)}
\sim
\frac{12g\cdot(4g-4)!\cdot 2^{6g-6}}{\Vol(\Gamma_1(g))}
\sim \sqrt{\frac{3\pi g}{2}}\cdot 12g
\cdot (4g-4)!
\cdot\left(\frac{9}{8}\right)^{2g-2}
\,,
$$
where we used
$$
\Vol\Graph_1(g)
=\sqrt{\frac{2}{3\pi g}}
\cdot\left(\frac{8}{3}\right)^{4g-4}
\cdot (1+o(1))
\quad\text{as }g\to+\infty\,.
$$
that is obtained by a combination of
Theorem~\ref{th:bounds:for:Vol:Gamma:k:g} and
Corollary~\ref{cor:multi:harmonic:asymptotic:all:k:bis}.
Actually, the latter asymptotic equivalence was originally
proved in Equation~(4.5) from Theorem~4.2
in~\cite{DGZZ:volume}.
\end{proof}

\begin{proof}[Proof of Theorem~\ref{thm:pg:asymptotics}]
At the current stage, we can prove mod-Poisson convergence
of $p_g$ only for a relatively small radius $R=8/7\approx
1.14286$. Thus, a straightforward application of
Corollary~\ref{cor:multi:harmonic:asymptotic:all:k} to
$p_g$ does not provide sufficiently strong estimates for
the distribution $p_g$. This is why we proceed differently.

Relation~\eqref{eq:pg:equivalent} follows from
combination of relations~\eqref{eq:th:permutations:probabiliy:k}
for $q_{3g-3,\infty,1/2}(k)$ with relations
expressing $p_g(k)$ through $q_{3g-3,\infty,1/2}(k)$
proved in Theorem~\ref{thm:p_g:individual:estimates}.

To estimate the left and right tails of the distribution $p_g$ we use
relation~\eqref{eq:p:through:q:minus:1:4}
$$
p_g(k)
=q_{3g-3,\infty,1/2}(k)
\cdot \left(1 + O\left(g^{-1/4} \cdot (\log g)^{24}\right)\right)\,,
$$
proved in Theorem~\ref{thm:p_g:individual:estimates}
for $k$ satisfying $k\le \frac{3}{4\log 2} \log g$.

Estimate~\eqref{eq:head:p:0:1} for the left tails follows
directly from Equation~\eqref{eq:th:permutations:head} of
Theorem~\ref{thm:permutation:asymptotics}.

For the right tail, the equivalence between $p_g(k)$ and $q_{3g-3,\infty,1/2}(k)$
is not known beyond $\frac{3}{4 \log 2}$ and we need to pass to estimates on
the complementary event.
Let $\lambda_{3g-3}=\frac{1}{2}\log(6g-6)$.
Relation~\eqref{eq:p:through:q:minus:1:4} implies, that for $x$ in the range
$0\le x \le \frac{3}{2\log 2}\approx 2.16$ we have
$$
\sum_{k=1}^{\lceil x\lambda_{3g-3}\rceil} p_g(k)
=\left(
\sum_{k=1}^{\lceil x\lambda_{3g-3}\rceil}
q_{3g-3,\infty,1/2}(k)
\right)
\cdot \left(1 + O\left(g^{-1/4} \cdot (\log g)^{24}\right)\right)\,.
$$
In particular, this relation is applicable to $x_1=1.236$
and to $x_2=1.24$.
Passing to complementary probabilities, we get
for for $0\le x \le \frac{3}{2\log 2}$:
\begin{equation}
\label{eq:two:error:terms:for:p}
\sum_{k=\lceil x\lambda_{3g-3}\rceil+1}^{3g-3} p_g(k)
=
\sum_{k=\lceil x\lambda_{3g-3}\rceil+1}^{3g-3}
q_{3g-3,\infty,1/2}(k)
+ O\left(g^{-1/4} \cdot (\log g)^{24}\right)\,.
\end{equation}
We use now relation~\eqref{eq:th:permutations:tail}
from Theorem~\ref{thm:permutation:asymptotics}
for the bound for the tail of distribution $q_{3g-3,\infty,1/2}$.
Applying~\eqref{eq:th:permutations:tail}
with $n=3g-3$, $\lambda=\frac{1}{2}\log(6g-6)$,
we get
\begin{equation*}
\begin{aligned}
\sum_{k=\lceil x\lambda_{3g-3}\rceil+1}^{3g-3}
q_{3g-3,\infty,1/2}(k)
&=O\Big(\exp\big(-\lambda_{3g-3}
(-x \log x - x + 1)\big)\Big)\,.
\\
&=O\left(g^{\left(-x\log x-x+1\right)/2}\right)\,.
\end{aligned}
\end{equation*}
For $x_1=1.236$ we have
$$
(-x_1\log x_1-x_1+1)/2>-0.249\,.
$$
Since the function $-x\log x-x+1$ takes value $0$ at $x=1$
and is monotonously decreasing on $[1,+\infty]$
we conclude that for $x\in(1,x_1]$ we have
$$
g^{-1/4} \cdot (\log g)^{24}
=o\left(g^{\left(-x\log x-x+1\right)/2}
\cdot \frac{1}{\log g}\right)\,.
$$
This implies that for this range of $x$ the error term
in the right-hand side of~\eqref{eq:two:error:terms:for:p}
is negligible with respect to the error term
in~\eqref{eq:th:permutations:tail} evaluated with
parameters $n=3g-3$, $\lambda=\frac{1}{2}\log(6g-6)$,
and~\eqref{eq:tail:p:1:236} follows.

For $x=x_2=1.24$ we have
$$
(-x_2\log x_2-x_2+1)/2< -0.253\,,
$$
so
$$
O\left(g^{\left(-x\log x-x+1\right)/2}\right)
=o\left(g^{-1/4}\right)\,.
$$
Taking into consideration monotonicity of
$-x\log x-x+1$ for $x\ge 1$, this
implies that for $x\ge x_2$ the first summand
in the right-hand side of~\eqref{eq:two:error:terms:for:p}
becomes negligible with respect to the second summand.
Note also that
$$
x_2\lambda_{3g-3}=0.62\log(6g-6)<0.62\log g +0.62\log 6<0.62\log g +1.12\,.
$$
Thus, the sum of $q_{3g-3,\infty,1/2}(k)$
starting from $k=\lfloor 0.62 \log g\rfloor+1$
might contain at most
three extra terms with respect to the sum starting from
$\lceil x_2\lambda_{6g-6}\rceil+1$. Clearly, each of these three
terms has order $o\left(g^{-1/4}\right)$. We have proved that
$$
\sum_{k=\lceil x_0\log g\rceil+1}^{3g-3}
q_{3g-3,\infty,1/2}(k)
=o\left(g^{-1/4}\right)\,.
$$
Plugging the above estimates in~\eqref{eq:two:error:terms:for:p} we obtain our
estimate for the right part. To estimate the left part we
rely~\eqref{eq:head:p:0:1} that we already proved. It is sufficient to notice
that for $x_0 = 0.18$ we have
$$
- (x_0 \log x_0 - x_0 + 1) / 2 < -0.255
$$
and~\eqref{eq:tail:1:4} follows.
\end{proof}

\begin{proof}[Proof of Theorem~\ref{th:pg:cumulants}]
The convergence mod-Poisson of $p_g(k)$ proved in
Theorem~\ref{thm:mod:poisson:pg} together with the general
asymptotics of cumulants in Theorem~\ref{thm:cumulants}
implies Theorem~\ref{th:pg:cumulants}.
\end{proof}

\section{Numerical and experimental data and further conjectures}

\subsection{Numerical and experimental data}
\label{s:numerics}

In this section we compare the distribution $p_g(k)$ of the number of components
of a random multicurve in genus $g$ (see Theorems~\ref{th:multicurves:a:b:c}
and~\ref{thm:pg:asymptotics} from
Section~\ref{s:intro}) with the approximation
given by the mod-Poisson convergence.

\begin{figure}[htb]
\includegraphics{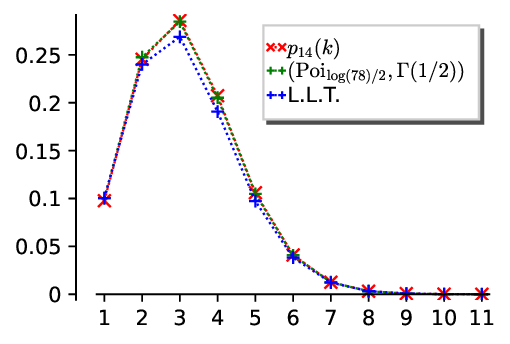}
\vspace{200bp}
\caption{
Exact distribution $p_g(k)$ of the number of components $k$ of a
random multicuvrves (red), the $(\Poisson_{\lambda_{3g-3}}, \Gamma(\frac{1}{2}))$-distribution
(green) and the distribution given by the local limit theorem (blue)
for $g=14$.}
\label{fig:poisson:versus:experimental:14}
\end{figure}

Recall that for any $\lambda > 0$, we defined
in~\eqref{eq:poisson:gamma:1:2} the real numbers $u_{\lambda, 1/2}(k)$,
for $k\in\N$, as the coefficients of the Taylor
expansion of
$$
e^{\lambda (t-1)}\cdot
\displaystyle
\frac{t \cdot \Gamma\left(\tfrac{3}{2}\right)}{\Gamma\left(1 + \tfrac{t}{2}\right)}
=
\sum_{k \geq 1} u_{\lambda, 1/2}(k) \cdot t^k
$$
We have the formula
$$
u_{\lambda,1/2}(k)
=\sqrt{\pi}\cdot e^{-\lambda}\cdot
\frac{1}{k!}\cdot
\sum_{i=1}^k \binom{k}{i} \cdot \phi_i \cdot
\left(\frac{1}{2}\right)^i \cdot \lambda^{k-i}\,,
$$
where $\phi_k$ are the coefficients
of the Taylor series of $1 / \Gamma(t)$.
Even though the sequence
$\{u_{\lambda, 1/2}(k)\}_{k \geq 1}$
is not a probability distribution we refer to this collection of numbers
as the $(\Poisson_\lambda, \Gamma(\frac{1}{2}))$-distribution

Corollary~\ref{cor:approximation:q:u:introduction}
shows that $q_g(k)$ (and hence also $p_g(k)$) is well-approximated by
$u_{\lambda_{3g-3}, 1/2}(k)$. Theorem~\ref{thm:pg:asymptotics}
also shows that $u_{\lambda_{3g-3}, 1/2}(k)$ can be approximated by a much simpler formula, namely
\[
u_{\lambda_{3g-3}, 1/2}(k+1) \sim p_g(k + 1) \sim
e^{-\lambda_{3g-3}} \frac{(\lambda_{3g-3})^{k}}{k!} \frac{\sqrt{\pi}}{2 \cdot \Gamma(1 + k / 2 \lambda_{3g-3})}.
\]

\begin{figure}[htb]
\includegraphics{poisson-gamma_vs_theoretical_g26.eps}
\vspace{200bp}
\caption{
(Experimental) distribution $p_g(k)$ of the number of components $k$ of a
random multicuvrves (red), the $(\Poisson_{\lambda_{3g-3}}, \Gamma(\frac{1}{2}))$-distribution
(green) and the distribution given by the local limit theorem (blue)
for $g=26$.}
\label{fig:poisson:versus:experimental:26}
\end{figure}

In the tables below and in
Figures~\ref{fig:poisson:versus:experimental:14},
\ref{fig:poisson:versus:experimental:26}
and~\ref{fig:poisson:versus:experimental:ratio}
we refer to the approximation given
by the function $u_{\lambda_{3g-3},1/2}$ as
the $(\Poisson_{\lambda_{3g-3}}, \Gamma(\frac{1}{2}))$-approximation,
and to the approximation
in the right-hand side of the latter expression
as ``LLT''-approximation (for ``Local Limit Theorem'').
We provide numerical data comparing the three quantities in
the tables below. For $g=14$ the distribution $p_{14}$ was
rigorously computed as sequence of explicit rational numbers.
For $g=26$ the distribution $p_{26}$ was
computed experimentally, collecting statistics
of random integra; generalized interval exchange transformations
(linear involutions).
The graphic comparison of this data
is presented in Figures~\ref{fig:poisson:versus:experimental:14}--\ref{fig:poisson:versus:experimental:ratio}.

\[
\begin{array}{l|ccc}
k  & p_{26}
   & (\Poisson_{\lambda_{3\cdot 26-3}}, \Gamma(\frac{1}{2}))
   & \text{LLT}
\\
\hline &&&\\
[-\halfbls]
1  & 0.0713 & 0.0724 & 0.0724 \\
2  & 0.2009 & 0.2022 & 0.1974 \\
3  & 0.2679 & 0.2675 & 0.2559 \\
4  & 0.2260 & 0.2251 & 0.2123 \\
5  & 0.1369 & 0.1361 & 0.1276 \\
6  & 0.0634 & 0.0633 & 0.0596 \\
7  & 0.0237 & 0.0237 & 0.0226 \\
8  & 0.0073 & 0.0073 & 0.0072 \\
9  & 0.0019 & 0.0019 & 0.0020 \\
10 & 0.0005 & 0.0004 & 0.0005 \\
11 & 0.0001 & 0.0001 & 0.0001 \\
\end{array}
\]

\begin{figure}[htb]
\includegraphics{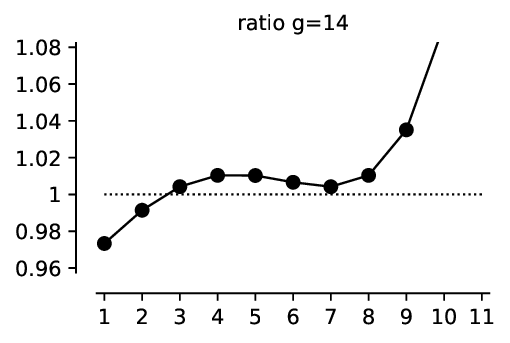}
\includegraphics{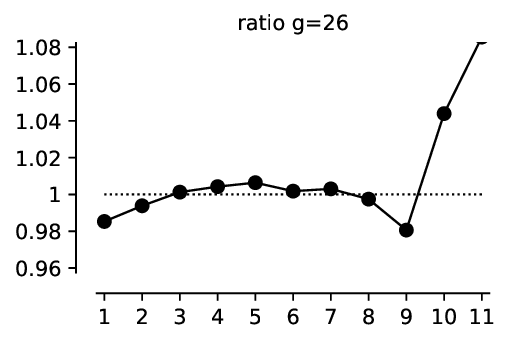}
\vspace{100bp}
\caption{Ratios $p_{g}(k)/u_{\lambda_{3g-3}, 1/2}(k)$
for $g=14$ (exact) and for $g=26$ (experimental).}
\label{fig:poisson:versus:experimental:ratio}
\end{figure}

\subsection{Further conjectures}
\label{s:speculations}

Recall that the square-tiled surfaces which
we study in this paper are integer points in
the total space of the
the bundle of quadratic differentials $\cQ_g$ over $\cM_g$.
Recall that Abelian square-tiled surfaces correspond to integer
points in the total space of the Hodge bundle $\cH_g$ over
$\cM_g$. In this section, we present two
conjectures on asymptotic statistics of cylinder
decomposition of random Abelian square-tiled surfaces.
We also present a conjecture
on asymptotic statistics of cylinder
decomposition of random square-tiled surfaces
in individual strata of holomorphic quadratic differentials.

We conjecture the following mod-Poisson convergence.
\begin{Conjecture}
\label{conj:abelian:weak:form}
Let $p^{Ab}_g(k)$ be the probability that a random Abelian
square-tiled surface in $\cH_g$ has $k$ cylinders. Then
for all $x > 0$, uniformly for $k$ in $\{0, 1, \ldots, \lfloor x \log(g) \rfloor\}$
we have as $g \to \infty$
$$
p^{Ab}_g(k + 1)
=
e^{-\mu_g} \cdot \frac{(\mu_g)^k}{k!}
\cdot
\left( \frac{1}{\Gamma(t)}+ o(1)) \right).
$$
where $\mu_g = \log(4g - 3)$.
\end{Conjecture}

In plain words, Conjecture~\ref{conj:abelian:weak:form}
implies that the statistics $\ProbaCyls^{Ab}_g (k)$ becomes
practically indistinguishable from the statistics of the
number of disjoint cycles in the cycle decomposition of a
random permutation in $S_{4g-3}$, with respect to the
uniform probability measure on the symmetric group of
$4g-3$ elements. The latter was denoted by
$\Proba_{4g-3,\infty,1}$ in
Section~\ref{s:sum:over:single:vertex:graphs}.

Recall that both the space $\cQ_g$ and $\cH_g$ of
respectively quadratic and Abelian differentials are
stratified by the partition of order of the zeros. The
parameters $6g-6$ and $4g-3$ that appear in the mod-Poisson
convergence of random square-tiled surfaces coincide with
the dimensions $\dim_\C \cQ_g = 6g-6$ and $\dim_\C \cH_g =
4g-3$. We conjecture the following strong form of
mod-Poisson convergence uniform for all non-hyperelliptic
connected components of all strata.

\begin{Conjecture}
\label{conj:quadratic:strong:form}
There exist a constant $R_2>1$ such that the mod-Poisson
convergence as in Theorem~\ref{thm:mod:poisson:pg} but with
radius $R_2$ holds uniformly for all non-hyperelliptic
components of strata of holomorphic quadratic
differentials.

More precisely, let $\cC$ be a
non-hyperelliptic
component of a stratum of
holomorphic quadratic differentials. Let $p_\cC(k)$ denote the probability that a
random quadratic square-tiled surface in $\cC$ has $k$
cylinders. Then
$$
\sum_{k \geq 1} p_\cC(k) t^k
=
(\dim_\C \cC)^\frac{t - 1}{2} \cdot
\frac{\sqrt{\pi}}{\Gamma(t/2)}
\left(1 + O \left( \frac{1}{\dim \cC} \right) \right)\,,
$$
where the error term is uniform over all
non-hyperelliptic components of all
strata of quadratic differentials
and uniform over all $t$ over compact subsets of the
complex disk $|t|<R_2$.
\end{Conjecture}

\begin{Conjecture}
\label{conj:abelian:strong:form}
Conjecture~\ref{conj:abelian:weak:form} holds uniformly for
all non-hyperelliptic connected components
of all strata of Abelian differentials.

More precisely, let $x > 0$. Let $\cC$ be a non-hyperelliptic connected
component of a stratum of Abelian differentials. Let
$p_\cC(k)$ denote the probability that a random Abelian
square-tiled surface in $\cC$ has $k$ cylinders. Then
$$
p^{Ab}_\cC(k) 
=
e^{-\mu_g} \cdot \frac{(\mu_g)^k}{k!}
\cdot
\left( \frac{1}{\Gamma(t)}+ o(1)) \right).
$$
where $\mu_g = \log(4g - 3)$.
and the error term is uniform over all non-hyperelliptic
components of all strata of Abelian differentials and for
$k$ in $\{0, 1, \ldots, \lfloor x \log(g) \rfloor\}$.
\end{Conjecture}

Conjecture~\ref{conj:abelian:strong:form} is based on analyzing huge experimental
data. We experimentally collected statistics of
the number $K_{\cC}(S)$ of maximal horizontal cylinders in
cylinder decompositions of random square-tiled surfaces
in about $30$ connected components $\cC$
of strata in genera from $40$ to
$10\,000$. In particular, the least squares linear
approximation for components $\cC$ of dimension
$\dim_\C \cC$ between $400$ and
$20\,000$ gives:
\begin{eqnarray*}
\E(K_\cC) &\sim 0.999 \log \dim_\C \cC+ 0.581\\
\Var(K_\cC) &\sim 0.996 \log \dim_\C \cC -1.043
\end{eqnarray*}
(compare to~\eqref{eq:mean:variance:perm}). Visually the graphs of
distributions $p^{Ab}_{\cC}(k)$ and $\frac{s(\dim \cC,k)}{(\dim \cC)!}$ are,
basically, indistinguishable for large genera.


\end{document}